\documentclass{amsbook}

\usepackage{mathrsfs}
\usepackage{amsfonts}
\usepackage{stmaryrd}
\usepackage{amscd,amssymb,verbatim}
\usepackage{amssymb,amsmath,amsthm,epsfig}

\addtocontents{toc}{\setcounter{tocdepth}{4}}

\newtheorem{theorem}{Theorem}[chapter]
\newtheorem{lemma}[theorem]{Lemma}

\theoremstyle{definition}
\newtheorem{definition}[theorem]{Definition}
\newtheorem{example}[theorem]{Example}

\newtheorem{proposition}[theorem]{Proposition}
\newtheorem{corollary}[theorem]{Corollary}
\theoremstyle{remark}
\newtheorem{remark}[theorem]{Remark}

\numberwithin{section}{chapter}
\numberwithin{equation}{chapter}
\newcommand{\N}{\mathbb{N}}

\newcommand{\J}{\mathbb{J}}

\newcommand{\C}{\mathbb{C}}

\newcommand{\cG}{\mathcal{G}}
\newcommand{\cB}{\mathcal{B}}

\newcommand{\cF}{\mathcal{F}}
\newcommand{\cH}{\mathcal{H}}
\newcommand{\cK}{\mathcal{K}}
\newcommand{\cD}{\mathcal{D}}
\newcommand{\cM}{\mathcal{M}}

\def\sideremark#1{\ifvmode\leavevmode\fi\vadjust{\vbox
to0pt{\vss \hbox to 0pt{\hskip\hsize\hskip1em
\vbox{\hsize2cm\tiny\raggedright\pretolerance10000
\noindent#1\hfill}\hss}\vbox to8pt{\vfil}\vss}}}
\newcommand{\edz}[1]{\sideremark{#1}}

\newcommand{\ip}[2]{\left\langle\,#1 ,  #2 \, \right\rangle}

\newcommand{\bmx}{\left[ \begin{matrix} }
\newcommand{\emx}{\end{matrix} \right] }
\newcommand{\littleroman}{\renewcommand{\labelenumi}{(\roman{enumi})}}

\begin{document}
\frontmatter
\title{Operator-Valued Measures, Dilations, and the Theory of Frames}

\author{Deguang Han}
\address{Department of Mathematics, University of Central
Florida, Orlando, USA} \email{deguang.han@ucf.edu}

\author{David R. Larson}
\address{Department of Mathematics, Texas A\&M University, College Station, USA}
\email{larson@math.tamu.edu}

\author{Bei Liu}
\address{Department of Mathematics, Tianjin University of Technology, Tianjin
300384, P.R. China}
\email{liubei@mail.nankai.edu.cn}

\author{Rui Liu}
\address{Department of Mathematics and LPMC, Nankai University, Tianjin 300071, P.R. China}
\email{ruiliu@nankai.edu.cn}

\thanks{{\it Acknowledgements}: The authors were all participants in the NSF funded Workshop
in Analysis and Probability at Texas A\&M University. The first author acknowledges partial support by a grant from the NSF. The third and fourth authors received partial support from the NSFC}

\date{}
\subjclass[2010]{Primary 46G10, 46L07, 46L10, 46L51, 47A20;\\Secondary 42C15, 46B15, 46B25, 47B48}

\keywords{operator-valued measures, von Neumann algebras,
dilations, normal maps, completely bounded maps, frames.}

\begin{abstract}

We develop elements of a general dilation theory for operator-valued
measures. Hilbert space operator-valued measures are closely related
to bounded linear maps on abelian von Neumann algebras, and some of
our results include new dilation results for bounded linear maps
that are not necessarily completely bounded, and from domain
algebras that are not necessarily abelian. In the non-cb case the
dilation space often needs to be a Banach space. We give
applications to both the discrete and the continuous frame theory.
There are natural associations between the theory of frames
(including continuous frames and framings), the theory of
operator-valued measures on sigma-algebras of sets, and the theory
of continuous linear maps between $C^*$-algebras. In this connection
frame theory itself is identified with the special case in which the
domain algebra for the maps is an abelian von Neumann algebra and
the map is normal (i.e. ultraweakly, or $\sigma$-weakly, or w*)
continuous. Some of the results for maps  extend to the case where
the domain algebra is non-commutative. It has been known for a long
time that a necessary and sufficient condition for a bounded linear
map from a unital C*-algebra into $B(H)$ to have a Hilbert space
dilation to a $*$-homomorphism is that the mapping needs to be
completely bounded. Our theory shows that even if it is not
completely bounded it still has a Banach space dilation to a
homomorphism. For the special case when the domain algebra is an
abelian von Neumann algebra and the map is normal, we show that the
dilation can be taken to be normal with respect to the usual Banach
space version of ultraweak topology on the range space.  We view
these results as generalizations of the known result of Cazzaza, Han
and Larson that arbitrary framings have Banach dilations, and also
the known result that completely bounded maps have Hilbertian
dilations. Our methods extend to some cases where the domain algebra
need not be commutative, leading to new dilation results for maps of
general von Neumann algebras. This paper was motivated by some
recent results in frame theory and the observation that there is a
close connection between the analysis of dual pairs of frames (both
the discrete and the continuous theory) and the theory of
operator-valued measures.

\end{abstract}

\maketitle

\setcounter{page}{4}
\tableofcontents

\mainmatter

\chapter*{Introduction}

We investigate some natural associations between the theory of
frames (including continuous frames and framings), the theory of
operator-valued measures (OVM's) on sigma algebras of sets, and the
theory of normal (ultraweakly or w* continuous) linear maps on von
Neumann algebras.  Our main focus is on the dilation theory of these
objects.

Generalized analysis-reconstruction schemes include dual pairs of
frame sequences, framings, and continuous versions of these. We
observe that all of these induce operator-valued measures on an
appropriate $\sigma$-algebra of Borel sets in a natural way. The
dilation theories for frames, dual pairs of frames, and framings,
have been studied in the literature and many of their properties
are well known.  The continuous versions also have a dilation
theory, but their properties are not as well understood. We show
that all these can be perhaps better understood in terms of
dilations of their operator-valued measures and their associated
linear maps.

There is a well known dilation theory for those operator-valued
measures that are completely bounded in the sense that their
associated bounded linear maps between the operator algebra
$L^\infty$ of the sigma algebra and the algebra of all bounded
linear operators on the underlying Hilbert space are completely
bounded maps (cb maps for short).  In this setting the dilation
theory for operator-valued measures is obtained naturally
from the dilation theory for  cb maps, and  cb maps dilate to
*-homomorphisms while OVM's dilate to projection-valued measures
(PVM's), where the projections are orthogonal projections. We
develop a general dilation theory for operator valued measures
acting on Banach spaces where operator-valued measure (or maps)
are are not necessarily completely bounded.  Our first main
result (Theorem \ref{main-thm}) shows that  any operator-valued
measure (not necessarily completely bounded) always has a dilation
to a projection-valued measure acting on a Banach space. Here the
dilation space often needs to be a Banach space and the
projections are idempotents that are not necessarily self-adjoint
(c.f. \cite{DSc}).

\bigskip

\noindent{\bf Theorem A \footnote{We are enumerating what we feel are perhaps the most
important of our contributions by labeling them A, B, C, D, E, ..., with
the order not necessarily by order of importance but simply by the order
of appearance in this manuscript. We thank the referee for making a
suggestion along these lines in order to help the reader.}}
 {\it Let $E:\Sigma \to B(X,Y)$ be an operator-valued measure. Then there exist a Banach space $Z$,
 bounded linear operators $S:Z\to Y$ and $T:X\to Z$, and a projection-valued probability measure
 $F:\Sigma\to B(Z)$ such that $$E(B)=SF(B)T$$ for all $B\in\Sigma$. }

\bigskip

\noindent  We will call $(F, Z, S, T)$ in the above theorem a {\it
Banach space dilation system}, and a {\it Hilbert dilation system}
if $Z$ can be taken as a Hilbert space. This theorem  generalizes
Naimark's (Neumark's) Dilation Theorem for positive operator
valued measures.  But even in the case that the underlying space
is a Hilbert space the dilation space cannot always be taken to be
a Hilbert space.  Thus elements of the theory of Banach spaces are
essential in this work.  A key idea is the introduction of the
elementary dilation space (Definitions \ref{de:418} and
\ref{de:420}) and the minimal dilation norm $\|\cdot\|_\alpha$
(Definition \ref{de:421}) on the space $M_E$ of bounded measurable
functions on the measure space for an OVM: The minimal dilation
norm $\|\cdot\|_\alpha$ on $M_E$ is defined by
\[\left\|\sum^{N}_{i=1}C_iE_{ {B_i}, {x_i}}\right\|_\alpha
=\sup_{B\in\Sigma}\left\|\sum_{i=1}^N C_iE(B\cap
B_i)x_i\right\|_Y\] for all $\sum^{N}_{i=1}C_iE_{ {B_i}, {x_i}}\in
M_E$. Using this we show that every OVM has a projection valued
dilation to the elementary dilation space, and moreover,  $||\cdot
||_{\alpha}$ is a minimal norm on the elementary dilation space
(see Theorem \ref{th:422} and Theorem \ref{th:426} ).

\bigskip

\noindent{\bf Theorem B}
 {\it Let $E:\Sigma\to B(X,Y)$ be an operator-valued measure and $(F,Z,S,T)$ be a corresponding Banach
 space dilation system.  Then we have the following:

(i) There exist an elementary Banach space dilation system $(F_\cD,
\widetilde{M}_{E,\cD},S_\cD,T_\cD)$ of $E$ and a linear isometric
embedding
\[U:\widetilde{M}_{E,\cD}\to Z\]
such that
\[S_\mathcal {D}=SU,\ F(\Omega)T=UT_\cD,\ UF_\cD(B)=F(B)U,
\quad \forall B\in \Sigma.\]

(ii) The norm $\|\cdot\|_\alpha$ is indeed a dilation norm. Moreover,  If $\cD$ is a dilation norm of $E,$ then there exists a
constant $C_\cD$ such that for any $\sum_{i=1}^NC_iE_{ {B_i},
{x_i}}\in M_{E,\cD},$
\[
\sup_{B\in\Sigma}\left\|\sum_{i=1}^N C_iE(B\cap B_i)x_i\right\|_Y
\leq C_\cD \left\|\sum_{i=1}^NC_iE_{ {B_i}, {x_i}}\right\|_\cD,
\]
where $N>0,$ $\{C_i\}_{i=1}^{N}\subset\mathbb{C},$
$\{x_i\}_{i=1}^{N}\subset X$ and $\{B_i\}_{i=1}^{N}\subset \Sigma. $
Consequently
\[ \|f\|_\alpha\le C_\cD\|f\|_\cD, \qquad \forall f\in M_{E}.\]
}


\bigskip





Framings are the natural generalization
of discrete frame theory (more specifically, dual-frame
pairs) to non-Hilbertian settings. Even if the underlying space is a Hilbert space,  the dilation space for framing induced operator valued measures can fail to be Hilbertian.
    This theory was originally developed by Casazza, Han and Larson in \cite{CHL} as an attempt to introduce {\it frame theory
    with dilations} into a Banach space context.  The initial motivation for the present manuscript was to completely understand
    the dilation theory of framings.  In the context of Hilbert spaces, we realized that the dilation theory for discrete framings from \cite{CHL}
    induces a dilation theory for discrete operator valued measures that may fail to be completely bounded in the sense of (c.f.
    \cite{Pa}). While in general an operator-valued probability
    measure does not admit a Hilbert space dilation, the dilation
    theory can be strengthened in the case that it does admit a Hilbert space dilation (Theorem \ref{pr:Hp1}):

\bigskip

\noindent{\bf Theorem C}
 {\it Let $E:\Sigma\to B(\cH)$ be an operator-valued probability measure.
 If $E$ has a Hilbert dilation system $(\widetilde{E},\widetilde{H},S,T)$, then there exists a
corresponding Hilbert dilation system $(F,\cK, V^*, V)$ such that
$V:\cH\to\cK$ is an isometric embedding. }

\bigskip

\noindent   This theorem turns out to have some interesting
applications to framing induced operator valued measure dilation.
In particular,  it led to a complete characterization of framings
whose induced operator valued measures are completely bounded. We
include here a few sample examples with the following theorem:
\bigskip

\noindent{\bf Theorem D}
 {\it Let $(x_i,y_i)_{i\in\N}$ be a non-zero framing for a Hilbert space
$\cH.$, and  $E$ be the operator-valued probability measure induced by
$(x_i,y_i)_{i\in\N}$. Then we have the following:

(i) $E$ has a Hilbert dilation space $\cK$ if
and only if there exist $\alpha_i,\beta_i \in \C, i\in\N$ with
$\alpha_i\bar{\beta_i}=1$ such that $\{\alpha_ix_i\}_{i\in\N}$ and
$\{\beta_iy_i\}_{i\in\N}$ both are the frames for the Hilbert space
$\cH.$

(ii)  $E$ is a completely bounded map if and only if
$\{x_i,y_i\}_{i\in\N}$ can be re-scaled to dual frames.

(iii)   If $\ \inf\|x_i\|\cdot\|y_i\|>0,$ then we can find $\alpha_i,\beta_i
\in \C, i\in\N$ with $\alpha_i\bar{\beta_i}=1$ such that
$\{\alpha_ix_i\}_{i\in\N}$ and $\{\beta_iy_i\}_{i\in\N}$ both are
frames for the Hilbert space $\cH$. Hence the operator-valued
measure induced by $\{x_i,y_i\}_{i\in\N}$ has a Hilbertian dilation.
}

\bigskip

\noindent For the existence of non-rescalable (to dual frame pairs)
framings, we obtained the following:

\bigskip

\noindent{\bf Theorem E}
 {\it There exists a framing for a Hilbert space such that its induced operator-valued measure is not completely bounded,
and consequently it can not be re-scaled to obtain a framing that
admits a Hilbert space dilation. }

\bigskip

The second part of Theorem E follows from the first part of the
theorem and Theorem  D (ii). This result also gives an example of
a framing for a Hilbert space which is not rescalable to a dual
frame pair. For the existence of such an example, the motivating
example of framing constructed by Casazza, Han and Larson (Example
3.9 in \cite{CHL}) can not be dilated  to an unconditional basis
for a Hilbert space, although it can be dilated to an
unconditional basis for a Banach space (Theorem 4.6, \cite{CHL}).
We originally conjectured that this is an example that fails to
induce a completely bounded operator valued measure. However, it
turns out that this framing can be re-scaled to a framing that
admits a Hilbert space dilation (see Theorem \ref{th:81}), and
consequently disproves our conjecture. Our construction of the new
example in Theorem E uses a non-completely bounded map to
construct a non-completely bounded OVM which yields the required
framing. This delimiting example shows that the dilation theory
for framings developed in \cite{CHL}  gives a true generalization
of Naimark's Dilation Theorem for the discrete case. This
nontrivial example led us to consider general
(non-necessarily-discrete) operator valued measures, and to the
results of Chapter 2 that lead to the dilation theory for general
(not necessarily completely bounded) OVM's that completely
generalizes Naimark's Dilation theorem in a Banach space setting,
and which is new even for Hilbert spaces.

Part (iii) of Theorem D provides us a sufficient condition under
which a framing induced operator-valued measure has a Hilbert
space dilation. This can be applied to framings that have nice
structures. For example, the following is an unexpected result for
unitary system induced framings, where a unitary system is a
countable collection of unitary operators. This clearly applies to
wavelet and Gabor systems.







\bigskip



\noindent{\bf Corollary F}
 {\it Let $\mathscr{U}_1$ and $\mathscr{U}_2$ be unitary systems on a
separable Hilbert space $\cH.$ If there exist $x,y\in\cH$ such that
$\{\mathscr{U}_1x,\mathscr{U}_2y\}$ is a framing of $\cH,$ then
$\{\mathscr{U}_1x\}$ and $\{\mathscr{U}_2y\}$ both are frames for
$\cH.$ }

\bigskip





One of the important applications  of  our main dilation theorem  (Theorem \ref{main-thm}) is the  dilation for
not necessarily cb-maps with appropriate continuity properties from
a commutative von Neumann algebra into $B(H)$ (Theorem \ref{th:t531}):

\bigskip

\noindent{\bf Theorem G}
 {\it If $\mathcal{A}$ is a purely
atomic abelian von Neumann algebra acting on a separable Hilbert
space, then for every ultraweakly continuous linear map
$\phi:\mathcal{A}\to B(\cH)$, there exists a Banach space $Z,$ an
ultraweakly continuous unital homomorphism $\pi:\mathcal{A}\to
B(Z)$, and bounded linear operators $T:\cH\to Z$ and $S:Z\to\cH$
such that $$\phi(a)=S\pi(a)T$$ for all $a\in\mathcal{A}$. }

\bigskip

\noindent The proof of Theorem G uses some special properties of
the minimal dilation system for the $\phi$ induced operator valued
measure on the space $(\Bbb{N}, 2^{\Bbb{N}})$. Motivated by some
ideas used in the proof of Theorem G,  we then obtained a
universal dilation result, Theorem \ref{th:5B1}, for bounded
linear mappings between Banach algebras.

\bigskip

\noindent{\bf Theorem H}  {\it Let $\mathscr{A}$ be a Banach
algebra, and let $\phi:\mathscr{A}\to B(\cH)$ be a bounded linear
operator, where $H$ is a Banach space. Then there exists a Banach
space $Z,$ a bounded linear unital homomorphism
$\pi:\mathscr{A}\to B(Z)$, and bounded linear operators $T:\cH\to
Z$ and $S:Z\to\cH$ such that
\[\phi(a)=S\pi(a)T\]
for all $a\in\mathscr{A}.$}

\bigskip

We prove that
this is a true generalization of our commutative theorem in an
important special case (see Remark \ref{re:431}), and generalizes some of our results for maps of
commutative von Neumann algebras to the case where the von Neumann
algebra is non-commutative (see Theorem \ref{th:t47} and Corollary
\ref{cor:e1}).  For the case when $\mathscr{A}$ is a von Neumann algebra acting on
a separable Hilbert space and $\phi$ is ultraweakly continuous
(i.e., normal) we conjecture that  the dilation space $Z$ can
be taken to be separable and the dilation homomorphism $\pi$ is
also ultraweakly continuous. While we are not able to confirm this
conjecture we shall prove the following:

\bigskip

\noindent{\bf Theorem I}  {\it Let $K,H$
be Hilbert spaces, $A\subset B(K)$ be a von Neumann algebra, and
$\phi:A\rightarrow B(H)$ be a bounded linear operator which is
ultraweakly-\textsc{SOT} continuous on the unit ball $B_A$ of $A$.
Then there exists a Banach space $Z$, a bounded linear
homeomorphism $\pi:A\rightarrow B(Z)$ which is
\textsc{SOT}-\textsc{SOT} continuous on $B_A$, and bounded linear
operator $T:H\rightarrow Z$ and $S:Z\rightarrow H$ such that
$$\phi(a)=S \pi(a)T$$ for all $a\in A.$ If in addition that $K,H$ are separable, then
the Banach space $Z$ can be taken to be separable.}

\bigskip

These results are apparently new for mappings of von Neumann
algebras. They generalize special cases of Stinespring's Dilation
Theorem. The standard discrete Hilbert space frame theory is
identified with the special case of our theory in which the domain
algebra is abelian and purely atomic, the map is completely
bounded, and the OVM is purely atomic and completely bounded with
rank-1 atoms (Remark \ref{re:415}).

The universal  dilation result has connections with
Kadison's similarity problem for bounded homomorphisms between von
Neumann algebras (see the Remark 4.14). For example,  if $\mathscr{A}$ belongs to one of the following classes:  nuclear; $\mathscr{A} = B(H)$;  $\mathscr{A}$ has no tracial states;  $\mathscr{A}$ is
commutative; $II_{1}$-factor with Murry
and von Neumann's property $\Gamma$, then any non completely bounded map $\phi:\mathscr{A}\to
B(H)$  can never have a Hilbertian dilation (i.e.  the
dilation space $Z$ can never be a Hilbert space)  since otherwise
$\pi:\mathscr{A}\to B(Z)$ would be similar to a *-homomorphism and hence completely bounded and so would
be $\phi$. On the other hand, if there exists a  von Neumann algebra $\mathscr{A}$  and a non completely bounded
map $\phi$ from $\mathscr{A}$ to $B(H)$ that has a Hilbert space
dilation: $\pi:\mathscr{A}\to B(Z)$ (i.e., where $Z$ is a Hilbert
space), then $\pi$ will be a counterexample to the Kadison's
similarity problem since in this case $\pi$ is a homomorphisim that is not completely
bounded and consequently can not be similar to a *-homomorphisim.

It is well known that there is a  theory establishing a
connection between general bounded linear mappings from the
$C^*$-algebra $C(X)$ of continuous functions on a compact Hausdorf
space $X$ into $B(H)$ and operator valued measures on the sigma
algebra of Borel
    subsets of $X$ (c.f. \cite{Pa}).  If $A$ is an abelian $C^*$-algebra then $A$ can be identified with $C(X)$ for a topological space $X$
     and can also be identified with $C(\beta X)$ where  $\beta X$ is the Stone-Cech compactification of X.  Then the
     support $\sigma$-algebra for the OVM is the sigma algebra of Borel subsets of $\beta X$  which is enormous.  However in our generalized
     (commutative) framing theory $\mathcal{A}$ will always be an abelian von Neumann algebra presented up front
     as $L^\infty(\Omega, \Sigma, \mu) $, with $\Omega$ a topological space and $\Sigma$ its algebra of Borel sets,
     and the maps on $A$ into $B(H)$ are  normal.  In particular, to model the discrete  frame and framing theory $\Omega$
     is a countable index set with the discrete topology (most often $\Bbb{N}$), so $\Sigma$ is its power set, and $\mu$ is
     counting measure.   So in this setting it is more natural to work directly with this presentation in developing
     dilation theory rather than passing to $\beta \Omega$, and we take this approach in this paper.

     We feel that the connection we make with established discrete
frame and framing theory is transparent,
    and then the OVM dilation theory for the continuous case becomes a natural but nontrivial generalization of the theory
    for the discrete case that was inspired by framings. After doing this we attempted to apply our techniques to the case where the domain algebra
    for a map is non-commutative. This led to Theorem H.    However, additional hypotheses are needed
    if dilations of maps are to have strong
    continuity and structural properties.  For a map between C*-algebras it is well-known that there is a Hilbert space dilation if the map is
    completely bounded. (If the domain algebra is commutative this statement is an iff.) Even if a map is not cb it has a Banach space dilation.
     We are interested in the continuity and structural properties a dilation can have.  Theorem G  shows that in  the discrete abelian case,
 the dilation of a normal map can be taken to be normal and the dilation space can be taken to be separable, and Theorem I tells us that with suitable hypotheses this type of result
 can be generalized to the noncommutative setting.



    The dilation theory developed in this paper uses Hilbert space operator algebra theory and aspects of Banach space theory,
   so we try to present Banach space versions of Hilbert space results when we can obtain them.  Some of the essential Hilbert space results
   we use are proven more naturally in a wider Banach context. The rest of the paper is organized as follows: Chapter one
    contains preliminary results and some exposition.
    In chapter two we develop our theory of Banach space
    operator-valued measures and the accompanying dilation theory.  Operator valued measures have many different dilations
    to idempotent valued measures on larger Banach spaces (even if the measure to be dilated is a cb measure on a Hilbert space)
    and a part of the theory necessarily deals with classification issues.
    Chapter three is devoted to some additional results
    and exposition for Hilbert space operator-valued framings and measures, including the non-cb measures and their Banach dilations.
    It contains exposition and examples on the manner in which frames and framings on a Hilbert space induce natural operator valued
    measures on that Hilbert space.  The reader might well benefit by reading this chapter first, although doing that would not be
    the natural order in which this theory is presented. In chapter four we present our results on
    dilations of linear maps in the non-commutative case. In chapter five, we give the detailed construction of the important example in Theorem E, and
    prove that the example constructed in \cite{CHL} indeed
    induces a completely bounded operator valued measure.

    We wish to thank a list of our friends and colleagues for many
    useful comments and suggestions on the preliminary version of
    this work. This list includes (but not limited to) David Blecher, Pete Casazza, Don Hadwin, Richard Kadison, Victor Kaftal, Vern Paulsen, Gelu
    Popescu, Zhongjin Ruan and Roger Smith. We would also like to thank  the anonymous referee
    for many good suggestions that help us improve the
    presentation of the paper.

\chapter{Preliminaries}

\section{Frames}

A \textit{frame} $\mathcal{F}$ for a Hilbert space $\cH$ is a
sequence of vectors $\{x_n\} \subset \cH$ indexed by a countable
index set $\J$ for which there exist constants $0<A\leq B<\infty$
such that, for every $x \in \cH$,
\begin{equation}\label{frame}
A\|x\|^2 \leq \sum_{n\in\J} |\ip{x}{x_n}|^2 \leq B\|x\|^2.
\end{equation}
The optimal constants (maximal for $A$ and minimal for $B$) are
known respectively as the upper and lower \textit{frame bounds}. A
 frame is called a \textit{tight frame} if $A=B$,  and
is called a \textit{ Parseval frame} if $A=B=1$. If we only
require that a sequence $\{x_{n}\}$ satisfies the upper bound
condition in (\ref{frame}), then $\{x_{n}\}$ is also called a {\it
Bessel sequence}.

A frame which is a basis is called a Riesz basis. Orthonormal bases
are special cases of Parseval frames. It is elementary that a
Parseval frame $\{x_n\}$ for a Hilbert space $\cH$ is an orthonormal
basis if and only if each $x_n$ is a unit vector.

For a Bessel sequence $\{x_{n}\}$, its {\it analysis operator}
$\Theta$ is a bounded linear operator from $\cH$ to
$\ell^{2}(\mathbb{N})$ defined by
\begin{equation}\label{frameDef}
\Theta x = \sum_{n\in \N}\ip{x}{x_{n}}e_{n},
\end{equation}
 where $\{e_{n}\}$ is the standard
orthonormal basis for $\ell^{2}(\mathbb{N})$. It can be easily
verified that
$$
\Theta^{*}e_{n} = x_{n}, \ \ \forall n\in\N
$$
The Hilbert space adjoint $\Theta^{*}$ is called the {\it synthesis
operator} for $\{x_{n}\}$. The positive operator
$S:=\Theta^{*}\Theta:\cH \rightarrow \cH$ is called the {\it frame
operator}, or sometimes the {\it Bessel operator} if the Bessel
sequence is not a frame, and we have
\begin{equation}\label{frameOp}
Sx = \sum_{n\in \N}\ip{x}{ x_{n}}x_{n},  \ \ \ \forall x\in \cH.
\end{equation}

A sequence $\{x_{n}\}$ is a  frame for $\cH$ if and only if its
analysis operator $\Theta$ is bounded, injective and has closed
range, which is, in turn, equivalent to the condition that the
frame operator $S$ is bounded and invertible. In particular,
$\{x_{n}\}$ is a  Parseval frame for $\cH$ if and only if $\Theta$
is an isometry or equivalently  if $S = I$.

Let $S$ be the frame operator for a frame $\{x_{n}\}$. Then the
lower frame bound is $1/||S^{-1}||$ and the upper frame bound is
$||S||$. From (\ref{frameOp}) we obtain the {\it reconstruction
formula (or frame decomposition)}:
$$
x = \sum_{n\in \N}\ip{x}{ S^{-1}x_{n}}x_{n}, \ \ \forall x\in \cH
$$

or equivalently

$$ x = \sum_{n\in \N}\ip{x}{x_{n}}S^{-1}x_{n}, \ \ \forall x\in
\cH.
$$

(The second equation is obvious.  The first can be obtained from the
second by replacing $x$ with $S^{-1}x$ and multiplying both sides by
$S$.)

The frame $\{S^{-1}x_{n}\}$ is called the {\it canonical or standard
dual} of $\{x_{n}\}$. In the case that $\{x_{n}\}$ is a Parseval
 frame for $\cH$, we have that $S = I$ and hence
$$
x = \sum_{n\in \N}\ip{x}{x_{n}}x_{n}, \ \ \forall x\in \cH.
$$

More generally, if a Bessel sequence $\{y_{n}\}$ satisfies
\begin{equation}\label{frame-dual}
x = \sum_{n\in \N}\ip{x}{y_{n}}x_{n}, \ \ \forall x\in \cH,
\end{equation}
where the convergence is in norm of $\cH$, then $\{y_{n}\}$ is
called an {\it alternate dual } of $\{x_{n}\}$. (Then $\{y_{n}\}$
is also necessarily a frame.) The canonical and alternate duals
are usually simply referred to as {\it duals}, and $\{x_{n},
y_{n}\}$ is called a {\it dual frame pair.} It is a well-known
fact that that a  frame $\{x_{n}\}$ is a Riesz basis if and only
if $\{x_{n}\}$  has a unique dual frame (cf. \cite{HL}).

There is a geometric interpretation of Parseval frames and general
frames. Let $P$ be an orthogonal projection from a Hilbert space
$\cK$ onto a closed subspace $\cH$, and let $\{u_{n}\}$ be a
sequence in $\cK$. Then $\{Pu_{n}\}$ is called the {\it orthogonal
compression} of $\{u_{n}\}$ under $P$, and correspondingly
$\{u_{n}\}$ is called an {\it orthogonal dilation } of
$\{Pu_{n}\}$. We first observe that if $\{u_n\}$ is a  frame for
$\cK$, then $\{Pu_n\}$ is a  frame for $\cH$ with frame bounds at
least as {\it good} as those of $\{u_n\}$ (in the sense that the
lower frame cannot decrease and the upper bound cannot increase).
In particular, $\{Pu_n\}$ is a  Parseval frame for $\cH$ when
$\{u_{n}\}$ is an orthonormal basis for $\cK$, i.e., every
orthogonal compression of an orthonormal basis (resp. Riesz basis)
is a Parseval frame (resp.  frame) for the projection subspace.
The converse is also true: every  frame can be orthogonally
dilated to a Riesz basis, and every  Parseval frame can be dilated
to an orthonormal basis. This was apparently first shown
explicitly by Han and Larson in Chapter 1 of \cite{HL}. There,
with appropriate definitions it had an elementary two-line proof.
And as noted by several authors, it can be alternately derived by
applying the Naimark (Neumark) Dilation theorem for operator
valued measures by first passing from a frame sequence to a
natural discrete positive operator-valued measure on the power set
of the index set. So it is sometimes referred to as the Naimark
dilation theorem for frames.  In fact, this is the observation
that inspired much of the work in this paper.

For completeness we formally state this result:
\begin{proposition} \cite{HL} \label{prop:orthog} Let $\{x_{n}\}$ be a
sequence in a Hilbert space $\cH$. Then
\begin{enumerate}\littleroman

\item  $\{x_n\}$ is a Parseval frame for $\cH$ if and only if
there exists a Hilbert space $\cK \supseteq \cH$ and an
orthonormal basis $\{u_n\}$ for $\cK$ such that $x_n = Pu_n$,
where $P$ is the orthogonal projection from $\cK$ onto $\cH$.

\item $\{x_n\}$ is a  frame for  $\cH$ if and only if there exists
a Hilbert space $\cK \supseteq \cH$ and a Riesz basis $\{v_n\}$
for $\cK$ such that $x_n = Pv_n$, where $P$ again is the
orthogonal projection from $\cK$ onto $\cH$.
\end{enumerate}
\end{proposition}

The above dilation result was later generalized in \cite{CHL} to
dual frame pairs.

\begin{theorem} \label{ch2-dualpair-dilation}
Suppose that  $\{x_n\}$ and $\{y_{n}\}$ are two frames for a
Hilbert space $\cH$. Then the following are equivalent:
\begin{enumerate}\littleroman
\item $\{y_{n}\}$ is a dual for $\{x_{n}\}$;

\item  There exists a Hilbert space $\cK \supseteq \cH$ and a Riesz
basis $\{u_n\}$ for $\cK$ such that $x_n = Pu_n$, and $y_{n} =
Pu_{n}^{*}$, where $\{u_{n}^{*}\}$ is the (unique) dual of the Riesz
basis $\{u_{n}\}$ and $P$  is the orthogonal projection from $\cK$
onto $\cH$.
\end{enumerate}
\end{theorem}

As in  \cite{CHL}, a {\it framing} for a Banach space $X$ is a pair
of sequences $\{x_i,y_i\}$ with $\{x_i\}$ in $X$, $\{y_i\}$ in the
dual space $X^*$ of $X$, satisfying the condition that
\[x=\sum_i\langle x, y_i\rangle x_i,\]
where this series converges unconditionally for all $x\in X.$

The definition of a framing is a natural generalization of the
definition of a dual frame pair. Assume that $\{x_i\}$ is a  frame
for $\cH$ and $\{y_i\}$ is a dual frame for $\{x_i\}$. Then
$\{x_i,y_i\}$ is clearly a framing for $\cH.$ Moreover, if
$\alpha_i$ is a sequence of non-zero constants, then
$\{\alpha_ix_i, \bar{\alpha}^{-1}_iy_i\}$ (called a rescaling of
the pair) is also a framing, although a simple example (Example
\ref{ex:e11}) shows that it need not be a pair of frames, even if
$\{\alpha_ix_i\}$, $\{\bar{\alpha}^{-1}_iy_i\}$ are bounded
sequence.

\begin{example}\label{ex:e11}
Let $\cH$ be a separable Hilbert space and let $\{e_i\}_{i\in\N}$ be an
orthonormal basis of $\cH.$ Let
$$\{x_i\}=\Big\{e_1, e_2,
e_2, e_3, e_3, e_3, .\, .\, . \Big\}$$ and
$$\{y_i\}=\Big\{e_1, \frac{1}{2}e_2,
\frac{1}{2}e_2, \frac{1}{3}e_3, \frac{1}{3}e_3, \frac{1}{3}e_3, .\,
.\, . \Big\}.$$ Then $\{x_i,y_i\}_{i\in\N}$ is a framing of $\cH,$
$\|x_i\|\leq 1,\|y_i\|\leq 1,$ but neither $\{x_i\}$ nor $\{y_i\}$ are
frames for $\cH.$
\end{example}
\begin{proof}
Let $$\{\alpha_i\}=\Big\{1, \frac{1}{\sqrt{2}}, \frac{1}{\sqrt{2}},
\frac{1}{\sqrt{3}}, \frac{1}{\sqrt{3}}, \frac{1}{\sqrt{3}}, .\, .\,
. \Big\},$$ then
$$\{\alpha_ix_i\}=\{y_i/\alpha_i\}=\Big\{e_1,
\frac{1}{\sqrt{2}}e_2, \frac{1}{\sqrt{2}}e_2, \frac{1}{\sqrt{3}}e_3,
\frac{1}{\sqrt{3}}e_3, \frac{1}{\sqrt{3}}e_3, .\, .\, . \Big\}$$ is
a Parseval frame of $\cH.$ Thus for any $x\in\cH,$ we have
\[x=\sum_{i\in\N}\langle x,\alpha_ix_i\rangle\frac{y_i}{\alpha_i}=
\sum_{i\in\N}\langle x,x_i\rangle y_i,\] and this series converges unconditionally. Hence $\{x_i,y_i\}_{i\in\N}$ is a framing of $\cH$ and
$\|x_i\|\leq 1,\|y_i\|\leq 1.$ But for any $j\in\N,$
\[\sum_{i\in \N}|\langle e_j,x_i\rangle|^{2}=j\|e_j\|=j\] and
\[\sum_{i\in\N}|\langle e_j,y_i\rangle|^{2}=\frac{\|e_j\|}{j}=\frac{1}{j}.\]
So $\{x_i\}$ is not a Bessel sequence and $\{y_i\}$ is a Bessel
sequence but not a  frame.
\end{proof}

If $\{x_i, y_i\}$ is a framing, and if $\{x_i, y_i\}$ are both
Bessel sequences, then $\{x_i\}, \{y_i\}$ are frames and $\{x_i,
y_i\}$ is a dual frame pair. Indeed,  if $B, C$ are the Bessel
bounds for $\{x_i\}, \{y_i\}$ respectively, then for each $x,$
\begin{eqnarray*}
\langle x,x\rangle &=&\left\langle\sum_{i\in\N}\langle x,y_i\rangle x_i,x\right\rangle\\
&=&\sum_{i\in\N}\langle x,y_i\rangle \langle x_i,x\rangle\\
&\leq& \left(\sum_{i\in\N}\left|\langle x,y_i\rangle\right|^2\right)^{1/2}
\left(\sum_{i\in\N}\left|\langle x_i,x\rangle\right|^2\right)^{1/2}\\
&\leq&C^{1/2}\|x\|\left(\sum_{i\in\N}\left|\langle x_i,x\rangle\right|^2\right)^{1/2},
\end{eqnarray*}
and so $C^{-1}\|x\|^2\leq\sum_{i\in\N}\left|\langle
x_i,x\rangle\right|^2$ as required. So our interests will involve
framings for which $\{x_i\}, \{y_i\},$ or both, are not Bessel
sequences.

\begin{definition}\label{de:31} \cite{CHL}
A sequence $\{x_i\}_{i\in \N}$ in a Banach space $X$ is a {\it
projective frame for $X$} if there is a Banach space $Z$ with an
unconditional basis $\{z_i, z^*_i\}$ with $X\subset Z$ and a
(onto) projection $P: Z\to X$ so that $Pz_i = x_i$ for all $i
\in\N$. If $\{z_i\}$ is a 1-unconditional basis for $Z$ and $\|P\|
= 1$, we will call $\{x_i\}$ a Projective Parseval frame for $X$.
\end{definition}
In this case, we have for all $x \in X$ that
\[x =\sum_i\langle
x,z^*_i\rangle z_i = Px =\sum_i\langle x,z^*_i\rangle Pz_i
=\sum_i\langle x,z^*_i\rangle x_i,\] and this series converges
unconditionally in X. So this definition recaptures the
unconditional convergence from the Hilbert space definition.

We note that there exist projective frames in the sense of
Definition \ref{de:31} for an infinite dimensional Hilbert space
that fail to be  frames.  An example is contained in Chapter 5.

\begin{definition} \label{de:32} \cite{CHL}
A framing model is a Banach space $Z$ with a fixed unconditional
basis $\{e_i\}$ for $Z.$ A framing modeled on $(Z, \{e_i\}_{i\in\N}
)$ for a Banach space $X$ is a pair of sequences $\{y_i\}$ in $X^*.$
and $\{x_i\}$ in $X$ so that the operator $\theta: X\to Z$ defined
by
\[\theta u =\sum_{i\in\N}\langle u,y_i\rangle e_i,\]
 is an into isomorphism and $\Gamma: Z \to X$ given by
 \[\Gamma(\sum_{i\in\N}
a_ie_i)=\sum_{i\in\N}a_ix_i\] is bounded and $\Gamma \theta=I_X$.
\end{definition}

In this setting, $\Gamma$ becomes the reconstruction operator for
the frame.  The following result due to Casazza, Han and Larson
\cite{CHL}
 shows that these three methods for defining a frame on a
Banach space are really the same.
\begin{proposition}\label{pr:33}
Let $X$ be a Banach space and $\{x_i\}$ be a sequence of elements of
$X.$ The following are equivalent:
\begin{enumerate}
\item[(1)]$\{x_i\}$ is a projective frame for $X$.

\item[(2)]There exists a sequence $y_i\in X^*$  so that $\{x_i,
y_i\}$ is a framing for $X.$

\item[(3)]There exists a sequence $y_i\in X^*$  and a framing
model $(Z, \{e_i\})$ so that $\{x_i, y_i\}$ is a framing modeled
on $(Z, \{e_i\}).$
\end{enumerate}
\end{proposition}

This proposition tells us that if $\{x_i, y_i\}$ is a framing of
$X,$ then $\{x_i, y_i\}$ can be dilated to an unconditional basis.
That is, we can find a Banach space $Z$ with an unconditional
basis $\{e_i,e^*_i\},$ $X\subset Z $ and two bounded linear maps
$S$ and $T$ such that $Se_i=x_i$ and $Te^*_i=y_i.$

\begin{definition}
Let $\cH$ be a separable Hilbert space and $\Omega$ be a $\sigma$-locally compact ($\sigma$-compact and locally compact) Hausdorff space endowed
with a positive Radon measure $\mu$ with $\mbox{supp}(\mu) =\Omega$. A weakly continuous function $\cF:\Omega\to \cH$ is called a
\emph{continuous frame} if there exist constants $0 < C_1\le C_2 < \infty$ such that
\begin{eqnarray*}
C_1\|x\|^2 \le \int_\Omega |\langle x,\cF (\omega) \rangle|^2 d
\,\mu(\omega) \le C_2\|x \|^2, \quad \forall\, x\in \cH.
\end{eqnarray*}
\end{definition}
If $C_1 = C_2$ then the frame is called \emph{tight}. Associated to
$\cF $ is the frame operator $S_\cF $ defined in the weak sense by
\begin{eqnarray*} S_\cF  : \cH \to \cH, \quad \langle S_\cF  (x),y\rangle
:= \int_\Omega \langle x, \cF (\omega)\rangle\cdot\langle \cF
(\omega),y\rangle d \,\mu(\omega).
\end{eqnarray*}
It follows from the definition that $S_\cF $ is a bounded, positive, and invertible operator. We define the following transform associated to
$\cF $,
\begin{eqnarray*}
V_\cF  : \cH\to L^2(\Omega, \mu), \quad  V_\cF  (x)(\omega) :=
\langle x, \cF (\omega)\rangle.
\end{eqnarray*}
This operator is called the {\it analysis operator} in the literature and its adjoint operator is given by
\begin{eqnarray*}
V_\cF ^* : L^2(\Omega, \mu) \to \cH, \quad \langle V_\cF
^*(f),x\rangle :=\int_{\Omega} f(\omega)\langle\cF (\omega),x\rangle
d\,\mu(\omega).
\end{eqnarray*}
Then we have $ S_\cF = V_\cF ^{*} V_\cF$, and
\begin{eqnarray} \label{cfdual}
\langle x,y\rangle = \int_\Omega \langle x, \cF
(\omega)\rangle\cdot\langle \cG (\omega),y\rangle d \,\mu(\omega),
\end{eqnarray}
where $\cG (\omega) := S_\cF ^{-1}\cF (\omega)$ is the {\it standard dual} of $\cF$. A weakly continuous function $\cF :\Omega\to \cH$ is called
\emph{Bessel} if there exists a positive constant $C$ such that
\begin{eqnarray*}
\int_\Omega |\langle x,\cF (\omega) \rangle|^2 d \,\mu(\omega) \le
C\|x \|^2, \quad \forall\, x\in \cH.
\end{eqnarray*}

It can be easily shown  that if $\cF :\Omega\to \cH$ is  Bessel,
then it is a frame for $\cH$ if and only if there exists a Bessel
mapping $\cG$ such that the reconstruction formula (\ref{cfdual})
holds. This $\cG$ may not be the standard dual of $\cF$. We will
call $(\cF, \cG)$ a {\it dual pair}.

A discrete frame is a Riesz basis if and only if its analysis
operator is surjective. But for a continuous frame $\cF$, in
general we don't have $V_\cF (\cH) = L^{2}(\Omega, \mu)$. In fact,
this could happen only when $\mu$ is  purely atomic. Therefore
there is no Riesz basis type dilation theory for continuous frames
(however, we will see later that in contrast the induced
operator-valued measure does have projection valued measure
dilations). The following modified dilation theorem was due to
Gabardo and Han \cite{GH}:

\begin{theorem} Let $\cF$ be a $(\Omega, \mu)$-frame for
$\cH$ and $\cG$ be one of its duals. Suppose that both $V_\cF(\cH)$
and $V_\cG (\cH)$ are contained in the range space $\cM$ of the
analysis operator for some $(\Omega, \mu)$-frame. Then there is a
Hilbert space $\cK\supset \cH$ and a $(\Omega, \mu)$-frame
$\tilde{\cF}$ for $\cK$ with  $P\tilde{\cF} = \cF$, $P\tilde{\cG} =
\cG$ and $V_{\tilde{\cF}}(\cH) = \cM$, where $\tilde{\cG}$ is the
standard dual of $\tilde{\cF}$ and $P$ is the orthogonal projection
from $\cK$ onto $\cH$.
\end{theorem}

\begin{definition}
Let $X$ be a Banach space and $\Omega$ be a $\sigma$-locally
compact Hausdorff space. Let $\mu$ be a Borel measure on $\omega$.
A continuous framing on $X$ is a pair of maps
$(\mathscr{F},\mathscr{G}),$
\[\mathscr{F}:\Omega\to X, \quad\mathscr{G}:\Omega\to X^*,\]
such that the equation
\begin{eqnarray*}
\left\langle E_{(\mathscr{F},\mathscr{G})}(B)x,y\right\rangle
=\int_B\langle x,\mathscr{G}(\omega)\rangle\langle
\mathscr{F}(\omega),y\rangle d\mu(\omega)
\end{eqnarray*} for $x\in X$, $y\in X^*$, and $B$ a Borel subset of $\Omega$,
defines an operator-valued probability measure on $\Omega$ taking
value in $B(X)$ (see Definition \ref{de:35}). In particular, we
require the integral on the right to converge for each
$B\subset\Omega.$ We have
\begin{eqnarray}\label{eq:c71}
E_{(\mathscr{F},\mathscr{G})}(B)=\int_B
\mathscr{F}(\omega)\otimes\mathscr{G}(\omega)d E(\omega)
\end{eqnarray}
where the integral converges in the sense of Bochner. In
particular, since $E_{(\mathscr{F},\mathscr{G})}(\Omega)=I_X,$ we
have for any $x\in X$ that
\[\langle x, y\rangle =\int_\Omega\langle x,\mathscr{G}(\omega)\rangle\langle\mathscr{F}(\omega), y\rangle d E(\omega).\]
\end{definition}

\section{Operator-valued Measures}

This section briefly  discusses the well-known dilation theory for operator-valued measures, and establishes the connections between framing
dilations and dilations of their associated operator-valued measures (more detailed discussion and investigation will be given in the subsequent
chapters).

In operator theory, Naimark's dilation theorem is a result that characterizes positive operator-valued measures.  Let $\Omega$ be a compact
Hausdorff space,and let $\cB$ be the $\sigma$-algebra of all the Borel subsets of $\Omega$. A $B(\cH)$-valued measure on $\Omega$ is a mapping
$E:\cB\to B(\cH)$ that is weakly countably additive, i.e., if $\{B_i\}$ is a countable collection of disjoint Borel sets with union $B$,
then
\[\langle E(B)x,y\rangle=\sum_i\langle E(B_i)x,y\rangle\]
holds for all $x,y$ in $\cH$. The measure is called {\it bounded}  provided that
\[\sup\{\|E(B)\|:B\in \cB\}<\infty,\]
and we let $\|\varphi\|$ denote this supremum. The measure is
called {\it regular} if for all $x,y$ in $\cH$, the complex
measure given by
\begin{eqnarray}\label{eq:w141}
 \mu_{x,y}(B)=\langle E(B)x,y \rangle
\end{eqnarray} is regular.

Given a regular bounded $B(\cH)$-valued measure $E$, one
obtains a bounded, linear map
\[\phi_E:C(\Omega)\to B(\cH)\] by
\begin{eqnarray}\label{eq:w142}
\langle \phi_E(f)x,y \rangle=\int_\Omega f\, d\,
\mu_{x,y}.
\end{eqnarray}
Conversely, given a bounded, linear map $\phi:C(\Omega)\to B(\cH)$,
if one defines regular Borel measures $\{\mu_{x,y}\}$ for each
$x,y$ in $\cH$ by the above formula (\ref{eq:w142}), then for each
Borel set $B$, there exists a unique, bounded operator $E(B)$,
defined by formula (\ref{eq:w141}), and the map $B\to
E(B)$ defines a bounded, regular $B(\cH)$-valued measure.
Thus, we see that there is a one-to-one correspondence between the
bounded, linear maps of $C(\Omega)$ into $B(\cH)$ and the regular
bounded $B(\cH)$-valued measures. Such measures are called
\begin{enumerate}
\item[(i)] {\it spectral}  if $E(B_1\cap B_2)=E(B_1)\cdot E(B_2),$ \item[(ii)] {\it positive} if $E(B)\geq 0$,
\item[(iii)] {\it self-adjoint} if $E(B)^*=E(B),$
\end{enumerate}
for all Borel sets $B,B_1$ and $B_2$.

Note that if $E$ is spectral and self-adjoint, then
$E(B)$ must be an orthogonal projection for all $B\in \mathcal{B}$, and hence $E$
is positive.

\begin{theorem}[Naimark]\label{th:23}
Let $E$ be a regular, positive, $B(\cH)$-valued measure on
$\Omega$. Then there exist a Hilbert space $\cK$, a bounded linear
operator $V:\cH\to\cK$, and a regular, self-adjoint, spectral,
$B(\cK)$-valued measure $F$ on $\Omega$, such that
\[E(B)=V^*F(B)V.\]
\end{theorem}

Stinespring's dilation theorem is for completely positive maps on
$C^*$-algebras. Let $\mathcal {A}$ be a unital $C^*$- algebra. An
operator-valued linear map $\phi:\mathcal {A}\to B(\cH)$ is said to
be {\it positive} if $\phi(a^{*}a) \geq 0$ for every $a\in
\mathcal{A}$, and it is called {\it completely positive} if for
every $n$-tuple
 $a_{1}, ... , a_{n}$ of elements in $\mathcal{A}$, the matrix $(\phi(a_{i}^{*}a_{j}))$ is positive in the usual sense
 that for every $n$-tuple of vectors $\xi_{1}, ... ,
\xi_{n} \in \cH$, we have
\begin{eqnarray}
\sum_{i, j=1}^{n}\langle \phi(a_{i}a_{j}^{*})\xi_{j}, \xi_{i}\rangle \geq 0
\end{eqnarray}
or equivalently, $(\phi(a_i^*a_j))$ is a positive operator on the Hilbert space $\cH\otimes\C^n$.

\begin{theorem}[Stinespring's dilation theorem]\label{th:22}
Let $\mathcal {A}$ be a unital $C^*$-algebra, and let $\phi:\mathcal
{A}\to B(\cH)$ be a completely positive map. Then there exists a
Hilbert space $\cK,$ a unital $*-$homomorphism $\pi: \mathcal {A}\to
B(\cK),$ and a bounded operator $V:\cH\to\cK$ with
$\|\phi(1)\|=\|V\|^2$ such that
\[\phi(a)=V^*\pi(a)V.\]
\end{theorem}

 The following is also well known for commutative  $C^*$-algebras:

\begin{theorem}[cf. Theorem 3.11, \cite{Pa}]\label{th:21}
Let $\cB$ be a $C^*$-algebra, and let $\phi:C(\Omega)\to\cB$ be
positive. Then $\phi$ is completely positive.
\end{theorem}

This result together with Theorem \ref{th:22} implies that
Stinespring's dilation theorem holds for positive maps when
$\mathcal{A}$ is a unital commutative $C^*$-algebra.

A proof of Naimark dilation theorem by using Stinespring's
dilation theorem can be sketched as follows: Let $\phi:{\mathcal
A}\to B(\cH)$ be the natural extension of $E$ to the $C^*$-algebra
${\mathcal A}$ generated by all the characteristic functions of
measurable subsets of $\Omega$.  Then $\phi$ is positive, and
hence is completely positive by Theorem \ref{th:21}. Apply
Stinespring's dilation theorem to obtain a $*-$homomorphism $\pi:
{\mathcal A} \to B(\cK),$ and a bounded, linear operator
$V:\cH\to\cK$ such that $\phi(f)=V^*\pi(f)V$ for all $f$ in
${\mathcal A}$.  Let $F$ be the $B(\cK)-$valued measure
corresponding to $\pi.$ Then it can be verified that $F$ has the
desired properties.

Let $\mathcal {A}$ be a $C^*$- algebra. We use  $M_n$ to denote
the set of all $n\times n$ complex matrices, and  $M_n(\mathcal
{A})$ to denote the set of all $n\times n$ matrices with entries
from $\mathcal {A}.$ For the following theory see (c.f.
\cite{Pa}):

Given two $C^*$-algebras $\mathcal {A}$ and $\mathcal {B}$ and a
map $\phi:\mathcal {A}\to\mathcal {B}$, obtain maps
$\phi_n:M_n(\mathcal {A})\to M_n(\mathcal {B})$ via the formula
\[\phi_n((a_{i,j}))=(\phi(a_{i,j})).\]
The map $\phi$ is called completely bounded if $\phi$ is bounded
and $\|\phi\|_{cb}=\sup_n\|\phi_n\|$ is finite.

Completely positive maps are completely bounded. In the other
direction we have Wittstock's decomposition theorem \cite{Pa}:

\begin{proposition}
Let $\mathcal {A}$ be a unital $C^*$-algebra, and let $\phi:\mathcal
{A}\to B(\cH)$ be a completely bounded map. Then $\phi$ is a linear
combination of two completely positive maps.
\end{proposition}

The following  is a generalization of Stinespring's representation
theorem.

\begin{theorem}\label{de:25}
Let $\mathcal {A}$ be a  unital $C^*$-algebra, and let
$\phi:\mathcal {A}\to B(\cH)$ be a completely bounded map. Then
there exists a Hilbert space $\cK,$ a $*-$homomorphism $\pi:
\mathcal {A}\to B(\cK),$ and bounded operators $V_i:\cH\to\cK,
i=1,2,$ with $\|\phi\|_{cb}=\|V_1\|\cdot\|V_2\|$ such that
\[\phi(a)=V_1^*\pi(a)V_2\]
for all $a\in\mathcal {A}.$ Moreover, if $\|\phi\|_{cb}=1,$ then
$V_1$ and $V_2$ may be taken to be isometries.
\end{theorem}

Now let $\Omega$ be a compact Hausdorff space, let $E$ be a
bounded, regular, operator-valued measure on $\Omega,$ and let
$\phi:C(\Omega)\to B(\cH)$ be the bounded, linear map associated
with $E$ by integration as described in section 1.4.1. So for any
$f\in C(\Omega),$
\[\langle \phi(f)x,y \rangle=\int_\Omega f\, d\, \mu_{x,y},\]
where
\[ \mu_{x,y}(B)=\langle E(B)x,y \rangle \]
The OVM $E$ is called completely bounded when $\phi$ is completely
bounded. Using Wittstock's decomposition theorem, $E$ is completely
bounded if and only if it can be expressed as a linear combination
of positive operator-valued measures.

Let $\{x_i\}_{i\in\J}$ be a non-zero frame for a separable Hilbert
space $\cH.$ Let $\Sigma$ be the $\sigma$-algebra of all subsets
of $\J$. Define the mapping
$$E: \Sigma\to B(\cH), \quad
E(B)=\sum_{i\in B} x_i\otimes x_i$$
where $x\otimes y$ is the mapping on $\cH$ defined by $(x\otimes
y)(u)=\langle u,y\rangle x.$ Then $E$ is a regular, positive
$B(\cH)$-valued measure. By Naimark's dilation Theorem
\ref{th:23}, there exists a Hilbert space $\cK$, a bounded linear
operator $V:\cH\to\cK$, and a regular, self-adjoint, spectral,
$B(\cK)$-valued measure $F$ on $\J$, such that
\[E(B)=V^*F(B)V.\]
We will show (easily) that this Hilbert space $\cK$ can be $\ell_2$,
and the atoms $x_{i}\otimes x_{i}$ of the measure dilates to rank-1
projections $e_{i}\otimes e_{i}$, where $\{e_{i}\}$ is the standard
orthonormal basis for $\ell^{2}$.  That is $\cK$ can be the same as
the dilation space in Proposition \ref{prop:orthog} (ii).

Similarly, suppose that $\{x_i,y_i\}_{i\in\J}$ is a non-zero framing
for a separable Hilbert space $\cH.$ Define the mapping
$$E:\Sigma\to B(\cH), \quad
E(B)=\sum_{i\in B} x_i\otimes y_i, $$ for all $B\in \Sigma$. Then
$E$ is a $B(\cH)$-valued measure. We will show that this $E$ also
has a dilation space $Z$. But this dilation space is not
necessarily  a Hilbert space, in general, it is a Banach space and
consistent with Proposition \ref{pr:33}. The dilation is
essentially constructed using Proposition \ref{pr:33} (ii), where
the dilation of the atoms $x_{i}\otimes y_{i}$ corresponds to the
projection $u_{i}\otimes u_{i}^{*}$ and $\{u_{i}\}$ is an
unconditional basis for the dilation space $Z$.

\chapter{Dilation of Operator-valued Measures}

We develop a dilation theory of
operator-valued measures  to projection (i.e. idempotent)-valued
measures. Our main results show that this can be always achieved for
any operator-valued measure on a Banach space. Our approach is to dilate an
operator-valued measure to a linear space which is ``minimal" in a certain sense
and complete it in a norm compatable with the dilation geometry.  Such ``dilation norms" are not
unique, and we w:ill focus on two: one we call the
minimal dilation norm and one we call the maximal dilation norm. The
applications of these new dilation results to Hilbert space
operator-valued measures  will be discussed in the subsequent
chapters.

A positive operator-valued measure is a measure whose values are
non-negative operators on a Hilbert space. From Naimark's dilation
Theorem, we know that every positive operator-valued measure can be
dilated to a self-adjoint, spectral operator-valued measure on a
larger Hilbert space. But not all of the operator-valued measures
can have a Hilbert dilation space. For example, the operator-valued
measure induced by the framing in Chapter 5 does not have a Hilbert
dilation space. However it always admits a Banach dilation space,
and moreover it can be dilated to a spectral operator-valued measure
on a larger Banach space. Thus, it is necessary to develop the
dilation theory to include \emph{non-Hilbertian } operator-valued
measures on a Hilbert space; i.e. $B(\cH)$-valued measures that have
no Hilbert dilation space. We obtain a similar result to Naimark's
dilation Theorem for non-Hilbertian $B(\cH)$-valued measures. That
is, for an arbitrary non-necessarily-Hilbertian $B(\cH)$-valued
measure $E$ on $(\Omega,\Sigma)$, we show that there exists a Banach
space $X,$ a spectral $B(X)$-valued measure $F$ on
$(\Omega,\Sigma)$, and bounded linear maps $S$ and $T$ such that
$E(B)=SF(B)T$ holds for every  $ B\in\Sigma$.

\section{Basic Definitions}

\begin{definition}\label{de:35}
Let $X$ and $Y$ be Banach spaces, and let $(\Omega,\Sigma)$ be a
measurable space. A \emph{$B(X,Y)$-valued measure on $\Omega$} is a
map $E:\Sigma\to B(X,Y)$ that is countably additive in the weak
operator topology; that is, if $\{B_i\}$ is a disjoint countable
collection of members of $\Sigma$ with union $B$, then
\begin{equation*}
y^{*}(E(B)x)=\sum_i y^{*} (E(B_i)x)
\end{equation*}
for all $x\in X$ and $y^*\in Y^*.$
\end{definition}

We will use the symbol $(\Omega, \Sigma, E)$ if the range space is
clear from context, or $(\Omega, \Sigma, E, B(X, Y))$, to denote
this operator-valued measure system.

\begin{remark}\label{re:36}
The Orlicz-Pettis theorem  states that weak unconditional
convergence and norm unconditional convergence of a series are the
same in every Banach space (c.f.\cite{DJT}). Thus  we have  that
$\sum_i E(B_i)x$ weakly unconditionally converges to $E(B)x$ if
and only if $\sum_i E(B_i)x$ strongly unconditionally converges to
$E(B)x.$ So Definition \ref{de:35} is equivalent to saying that
$E$ is strongly countably additive, that is, if $\{B_i\}$ is a
disjoint countable collection of members of $\Sigma$ with union
$B$, then
\begin{equation*}
E(B)x=\sum_i E(B_i)x,\quad  \ \forall x\in X.
\end{equation*}
\end{remark}

\begin{definition}\label{de:37}
Let $E$ be a $B(X,Y)$-valued measure on $(\Omega,\Sigma).$ Then the  norm of $E$ is defined by
\[\|E\|=\sup_{B\in\Sigma}\|E(B)\|.\]
We call $E$ normalized if $\|E\|=1.$
\end{definition}

\begin{remark}\label{re:38}
A $B(X,Y)$-valued measure $E$ is always bounded, i.e.
\begin{equation}\label{eq:31}
\sup_{B\in\Sigma}\|E(B)\|<+\infty.
\end{equation}
Indeed, for all $x\in X$ and $y^*\in Y^*,$
$\mu_{x,y^*}(B):=y^*(E(B)x) $ is a complex measure on
$(\Omega,\Sigma).$ From complex measure theory (c.f.\cite{Rud}),
we know that $\mu_{x,y^*}$ is bounded, i.e.
\[\sup_{B\in\Sigma}|y^*(E(B)x)|<+\infty.
\]
By the Uniform Boundedness Principle, we get (\ref{eq:31}).
\end{remark}

\begin{definition}\label{de:7}
A $B(X)$-valued measure $E$ on $(\Omega,\Sigma)$ is called:
\begin{enumerate}
\item[(i)] an operator-valued probability measure if
$E(\Omega)=\mathrm{I}_X,$

\item[(ii)]a projection-valued measure if $E(B)$ is a projection
on $X$ for all $B\in\Sigma$,

\item[(iii)] a spectral operator-valued measure if for all $A,
B\in \Sigma, E(A\cap B)=E(A)\cdot E(B)$ (we will also use the term
idempotent-valued measure to mean a spectral-valued measure.)
\end{enumerate}
\end{definition}
With Definition \ref{de:35}, a $B(X)$-valued measure which is a
projection-valued measure is always a spectral-valued measure (c.f.
\cite{Pa}). We give a proof for completeness. (Note that spectral
operator-valued measures are clearly projection-valued measures).

\begin{lemma}\label{le:39}
Let $E$ be a $B(X)$-valued measure. If for each $B\in\Sigma,$ $E(B)$ is a projection (i.e. idempotent),  then for any $A, B
\in\Sigma,$ $E(A\cap B)=E(A)\cdot E(B).$
\end{lemma}
\begin{proof}
If $P$ and $Q$ are projections on $X$, then $P+Q$ is a projection on
$X$ if and only if $PQ=QP=0$.

For disjoint $A, B\in \Sigma,$ we have
$E(A)+E(B)=E(A\cup B)$ which is a projection. From
above, we obtain that
$E(A)\cdot E(B)=E(A)\cdot E(B)=0$. Then, for
all $A, B\in \Sigma$, we have
\begin{eqnarray*}
E(A)\cdot E(B) &=&E((A\cap B^c)\cup(A\cap B))\cdot
E((B\cap A^c)\cup(B\cap A))\\
&=&(E(A\cap B^c)+E(A\cap B))\cdot(E(B\cap
A^c)+E(B\cap A))\\
&=&E(A\cap B^c)\cdot E(B\cap A^c)+E(A\cap
B^c)\cdot E(B\cap A)+\\
&&+E(A\cap B)\cdot E((B\cap
A^c)+E(A\cap B)\cdot E(B\cap A)\\
&=&E(A\cap B)\cdot E(B\cap A)=E(A\cap B).
\end{eqnarray*}

\end{proof}

\begin{definition}\label{de:41}
Let $E:\Sigma\to B(X,Y)$ be an operator-valued measure. Then
the dual of $E$, denoted by $E^*$, is a mapping from
$\Sigma$ to $B(Y^*,X^*)$ which is defined by
$E^*(B)=[E(B)]^*$ for all $B\in\Sigma$.
\end{definition}

\begin{proposition}\label{th:43}
Let $X$ and $Y$ be Banach spaces. If $X$ is reflexive, then
$E^*:\Sigma\to B(Y^*,X^*)$ is an operator-valued measure.
\end{proposition}
\begin{proof}
Let $\{B_i\}$ be a disjoint countable collection of members of $\Sigma$ with union $B$. Then for all $x\in X$ and $y^*\in Y^*,$ we have
\begin{eqnarray*}
(E^*(B)y^*)(x)=y^{*}(E(B)x)=\sum_i
y^{*}(E(B_i)x) =\sum_i (E^*(B_i)y^*)(x).
\end{eqnarray*}
Since $X=X^{**}$, this shows that $E$ is weakly countably additive in the sense of Definition \ref{de:35}.
\end{proof}

\begin{remark}\label{re:44}
 The following example shows that there exists
 an operator-valued measure $E$ for which
$E^*$ is not weakly countably additive. So the above result is sharp.
\end{remark}

\begin{example}\label{ex:45}
Let $\{x_i,y_i\}^{\infty}_{i=1}\subset l_1\times l_\infty$ and $\textbf{1}=(1,1,1,\cdots), $ where
\[\{x_i\}^{\infty}_{i=1}=\{e_1,e_1,e_1,e_2,e_3,e_4,e_5,e_6,\cdots\}\]
\[\{y_i\}^{\infty}_{i=1}=\{\textbf{1},\textbf{-1},e^*_1,e^*_2,e^*_3,e^*_4,e^*_5,e^*_6,\cdots\}, \]
and $\{e_i\}^{\infty}_{i=1}$ is the standard unit vector basis. Let $\Omega=\mathbb{N},\Sigma=2^{\mathbb{N}}$ and
\[E(B)=\sum_{i\in B}x_i\otimes y_i.\]

Then it can be verified that $E^*$ is not weakly countably
additive.
\end{example}

Given two operator-valued measure space systems $(\Omega, \Sigma, E,B(X))$ and $(\Omega, \Sigma, F,B(Y)).$ We say that
\begin{enumerate}
\item[(i)] $E$ and $F$ are \emph{isometrically equivalent} (or \emph{isometric}) if there is a surjective isometry $U:X\to Y$ such that
$E(B)=U^{-1}F(B)U$ for all $B\in\Sigma.$

\item[(ii)] $E$ and $F$ are \emph{similar} (or
\emph{isomorphic}) if there is a bounded linear invertible operator
$Q:X\to Y$ such that $E(B)=Q^{-1}F(B)Q$ for all
$B\in\Sigma.$
\end{enumerate}

The following property follows immediately from the definition:

\begin{proposition}\label{pr:47}
Let $X$ and $Y$ be Banach spaces and let $E:\Sigma\to B(X,Y)$ be an operator-valued measure. Assume that $W$ and $Z$ are Banach spaces and
that $T:W\to X$ and $S: Y\to Z$ both are bounded linear operators. Then the mapping $F:\Sigma\to B(W,Z)$ defined by
\begin{eqnarray*}
F(B)=S E(B)T, \ \ \forall\, B\in\Sigma,
\end{eqnarray*}
is an operator-valued measure.
\end{proposition}

Let $X$ be a Banach space and $\{x_i,y_i\}_{i\in\N}$ be a framing
for $X.$ Then as we have introduced in Chapter 1 that the mapping
$E$ defined by
\[E:2^{\N}\to
B(X), \quad  E(B)=\sum_{i\in B} x_i\otimes y_i ,\]  is an
operator-valued probability measure. We will call it the
operator-valued probability measure {\it induced} by the framing
$\{x_i,y_i\}_{i\in\N}.$ The following lemma shows that every
operator-valued probability measure system  $(\N, 2^\N, E,B(X))$
with rank one atoms is induced by a framing.

\begin{lemma}\label{le:312}
Let $(\N, 2^\N, E,B(X))$ be an operator-valued probability
measure system with
\begin{eqnarray*}\mathrm{rank}(E(\{i\}))\in\{0,1\}
\end{eqnarray*}
for all $i\in\N$. Then there exists a framing $\{x_i,y_i\}_{i\in\N}$
of $X$ such that $E$ is induced by $\{x_i,y_i\}_{i\in\N}$. Moreover,
if $X$ is a Hilbert space and $E$ is spectral, then the framing
$\{x_i,y_i\}_{i\in\N}$ can be chosen as a pair of Riesz bases of
$X$.
\end{lemma}

\begin{proof}
Without loss of generality, we can assume that for any $i\in\N,$
$\mathrm{rank}(E(\{i\}))=1.$ Then we can find $x_i\in X$ and $y_i\in
X^*$ such that for any $x\in X,$
\[E(\{i\})(x)=y_i(x)x_i=(x_i\otimes y_i)(x).\]
It follows that
\begin{equation}\label{eq:32}
x=I_X(x)=E(\N)(x)=\sum_{i\in \N}E(\{i\})(x)=\sum_{i\in
\N}(x_i\otimes y_i)(x).
\end{equation}
By Definition \ref{de:35}, we know that the series in (\ref{eq:32})
converges unconditionally for all $x\in X.$ Hence
$\{x_i,y_i\}_{i\in\N}$ is a framing of $X.$

When $X$ is a Hilbert space and $E$ is spectral, since for all
$B\in 2^\N,$
\begin{eqnarray*}
E(B)x=\sum_{i\in B} \langle x ,y_i\rangle x_i,\quad \forall
x\in X,
\end{eqnarray*}
we have for any $i,j\in\N$ and $x\in X,$
\begin{eqnarray*}
E(\{i\})E(\{j\})x=\langle x ,y_j\rangle \langle
x_j,y_i\rangle x_i.
\end{eqnarray*}

When $i=j,$ since
\begin{eqnarray*}
E(\{i\})E(\{i\})x=\langle x ,y_i\rangle \langle
x_i,y_i\rangle x_i=\langle x ,y_i\rangle x_i=E(\{i\})x,
\end{eqnarray*}
we get $\langle x_i,y_i\rangle=1.$

When $i\neq j,$ we have
\begin{eqnarray*}
E(\{i\})E(\{j\})x=\langle x ,y_j\rangle \langle
x_j,y_i\rangle x_i=0,
\end{eqnarray*}
and so  $\langle x_j,y_i\rangle=0.$ If $\sum_i a_ix_i=0,$ then
\begin{eqnarray*} E(\{j\})\left(\sum_i
a_ix_i\right)=\sum_i a_i E(\{j\})( x_i)=a_jx_j=0,
\end{eqnarray*}
and hence $a_j=0.$ Thus $\{x_i\}_{i\in\N}$ is an unconditional basis. By Lemma 3.6.2 in \cite{Ch}, we know that $\{x_i/\|x_i\|\}_{i\in\N}$ is a
Riesz basis of $X,$ and hence $\{\|x_i\|y_i\}_{i\in\N}$ is also a Riesz basis of $X.$ Clearly  for all $B\in 2^\N,$ we also have
\begin{eqnarray*}
E(B)x=\sum_{i\in E} \langle x ,\|x_i\|y_i\rangle
x_i/\|x_i\|,\quad \forall x\in X.
\end{eqnarray*}

\end{proof}

\section{Dilation Spaces and Dilations}

Let $(\Omega,\Sigma)$ be a measurable space and $E:\Sigma\to B(X,Y)$ be an operator-valued measure.

\begin{definition}\label{de:411}
A Banach space $Z$ is called \textbf{a dilation space} of an
operator-valued measure space $(\Omega,\Sigma,E)$ if there exist
bounded linear operators $S:Z\to Y$ and $T:X\to Z$, and a
projection-valued measure space $(\Omega,\Sigma,F,B(Z))$ such that
for any $B\in\Sigma,$
\[E(B)=S F(B)T.\]
We call $S$ and $T$ the corresponding analysis operator and
synthesis operator, respectively, and use
$(\Omega,\Sigma,F,B(Z),S,T)$ to denote the corresponding dilation
projection-valued measure space system.
\end{definition}

It is easy to see that the mappings
\[G_1:\Sigma\to B(X,Z),\quad G_1(B)=F(B)T\]
and
\[G_2:\Sigma\to B(X,Z),\quad G_2(B)=S F(B)\]
are both operator-valued measures. We call $G_1$ and $G_2$ the
corresponding {\it analysis operator-valued measure} and {\it
synthesis operator-valued measure}, respectively.

Clearly, $E=G_2\cdot G_1.$ Indeed,
\[E(B)=S F(B)T=S F(B)F(B)T=G_2(B)G_1(B).\]

\begin{remark}\label{re:412}
If $X=Y$ and $E(\Omega)=I_X,$ then $S$ is a surjection, $T$ is an isomorphic embedding, and $TS: Z\to Z$ is a projection onto
$T(X).$
\end{remark}
Further, we have the following result:

\begin{lemma}\label{le:413}
Let $(\Omega,\Sigma,E)$ be an operator-valued measure space and $(\Omega,\Sigma,F,B(Z),S,T)$ be a corresponding dilation
projection-valued measure space system.
\begin{enumerate}
\item[(i)]If $E(\Omega): X\to Y$ is an isomorphic operator, then $S: Z\to Y$ is a surjection and $T: X\to Z$ is an isomorphic embedding.

\item[(ii)] If $X=Y$ and $E(\Omega)=\mathrm{I}_X,$ then
$TS F(\Omega), F(\Omega)TS$ and $F(\Omega)TS F(\Omega)$
are all projections on $Z.$
\end{enumerate}
\end{lemma}

\begin{proof} (i) Since $E(\Omega)$ is an isomorphic operator, for any $y\in Y,$ there exists $x\in X$ such that $E(\Omega)x=y.$ Hence
$(S F(\Omega)T)(x)=E(\Omega)(x)=y,$ and so $S$ is surjective.

Suppose that $Tx_1=Tx_2$ for  $x_1,x_2\in X.$ Then we have $(S F(\Omega)T)(x_1)=(S F(\Omega)T)(x_2),$ and so $
E(\Omega)x_1=E(\Omega)x_2.$ Since $E(\Omega)$ is invertible,  we get $x_1=x_2.$ Thus $T$ is an isomorphic embedding.

(ii) When $X=Y$ and $E(\Omega)=\mathrm{I}_X,$ we have
\begin{eqnarray*}
 (TS F(\Omega))^{2}&=&TS F(\Omega)TS F(\Omega)=T E(\Omega)S F(\Omega)\\
&=&T\mathrm{I}_XS F(\Omega)=TS F(\Omega),\\
(F(\Omega)TS)^{2}&=&F(\Omega)TSF(\Omega)TS=F(\Omega)T\mathrm{I}_XS\\
&=&F(\Omega)TS,\\
\end{eqnarray*}
and
\begin{eqnarray*}
 (F(\Omega)TSF(\Omega))^{2}
 &=&F(\Omega)TSF(\Omega)F(\Omega)TSF(\Omega)\\
&=&F(\Omega)TSF(\Omega)TSF(\Omega)\\
&=&F(\Omega)T\mathrm{I}_XSF(\Omega)\\
&=&F(\Omega)TSF(\Omega).
\end{eqnarray*}
Thus $TSF(\Omega), F(\Omega)TS$ and
$F(\Omega)TSF(\Omega)$ are all projections on $Z.$
\end{proof}

In general, the corresponding dilation projection-valued measure
system is not unique. However, we can always find a
projection-valued probability measure system from a known dilation
projection-valued measure system.

\begin{proposition}\label{pr:414}
Let $(\Omega,\Sigma,E)$ be an operator-valued measure space and let
$(\Omega,\Sigma,F,B(Z),S,T)$ be a corresponding dilation
projection-valued measure space system. Define a mapping
$\hat{F}:\Sigma\to B(F(\Omega)Z)$ by
\[\hat{F}(B)=F(B)|_{F(\Omega)Z}, \quad\mbox{for all}\ B\in\Sigma .\]
Then
\[\hat{S}=S|_{F(\Omega)Z}: F(\Omega)Z\to Y\]
and
\[\hat{T}=F(\Omega)T: X\to F(\Omega)Z\]
are both bounded linear operators. And $\hat{F}$ is a spectral operator-valued probability measure such that
\[E(B)=\hat{S}\hat{F}(B)\hat{T}, \ B\in\Sigma.\]

\end{proposition}

\begin{proof}
Clearly $\hat{S}$ and $\hat{T}$ are both well defined and are
bounded linear operators. Since
\[F(B)F(\Omega)Z=F(B)Z=F(\Omega)F(B)Z\subset F(\Omega)Z,\]
$\hat{F}$ is well defined. It is clearly that $\hat{F}$ is
an operator-valued probability measure. Moreover, for all
$A,B\in\Sigma,$
\begin{eqnarray*}
\hat{F}(A)\hat{F}(B)
&=&F(A)|_{F(\Omega)Z}F(B)|_{F(\Omega)Z}\nonumber\\
&=&F(A)F(B)|_{F(\Omega)Z}\nonumber\\
&=&F(A\cap B)|_{F(\Omega)Z}\nonumber\\
&=&\hat{F}(A\cap B).
\end{eqnarray*}
It follows that $\hat{F}$ is a spectral operator-valued
measure. Finally, we have for any $B\in\Sigma$ and $x\in X,$
\begin{eqnarray*}
\hat{S}\hat{F}(B)\hat{T}(x)
&=&S|_{F(\Omega)Z}F(B)|_{F(\Omega)Z}F(\Omega)T(x)\nonumber\\
&=&S|_{F(\Omega)Z}F(B)F(\Omega)T(x)\nonumber\\
&=&SF(B)T(x)\nonumber\\
&=&E(B)(x).
\end{eqnarray*}
\end{proof}
\begin{remark}\label{re:01}
It is easy to see that
\[\hat{S}\hat{T}=\hat{S}\hat{F}(\Omega)\hat{T}=E(\Omega).\]
If $X=Y$ and $E(\Omega)=\mathrm{I}_X,$ then
\[\hat{T}\hat{S}=\hat{T}SF(\Omega)=F(\Omega)T\hat{S}
=F(\Omega)TSF(\Omega)\] is a projection from $F(\Omega)Z$
onto $\hat{T}Z=F(\Omega)T(Z).$
\end{remark}

As a consequence of Proposition \ref{pr:414}, we obtain:

\begin{corollary} If an operator-valued measure can be dilated to a
spectral operator-valued measure, then it can be dilated to a
spectral operator-valued probability measure.
\end{corollary}

\section{Elementary Dilation Spaces}

This section provides the first step in the construction of dilation
spaces for operator-valued measures. The main result (Theorem
\ref{th:422}) shows that for any dilation projection-valued measure
system $(\Omega,\Sigma,F,B(Z),S,T)$ of  an operator-valued measure
system $(\Omega,\Sigma,E,B(X,Y))$, there exists an ``elementary"
dilation operator-valued measure system which can be linearly
isometrically embedded into $(\Omega,\Sigma,F,B(Z),S,T)$.

\begin{definition}\label{de:415}
Let $Y$ be a Banach space and $(\Omega,\Sigma)$ be a measurable
space. A mapping $\nu:\Sigma\to Y$ is called \textbf{a vector-valued
measure} if $\nu$ is countably additive; that is, if $\{B_i\}$ is a
disjoint countable collection of members of $\Sigma$ with union $B$,
then
\begin{equation*}
\nu(B)=\sum_i \nu(B_i).
\end{equation*}
\end{definition}

We use the notation  $(\Omega,\Sigma,\nu,Y)$ for a vector-valued
measure system.
\begin{remark}\label{re:416}
By the Orlicz-Pettis Theorem, we know that $\nu$ countably additive
is equivalent to $\nu$ weakly countably additive. That is
\begin{equation*}
y^{*}\left(\nu(B)\right)=\sum_i y^{*} (\nu(B_i)),
\end{equation*}
for all $y^*\in Y^*.$
\end{remark}

We use the symbol $\mathfrak{M}^Y_\Sigma$ to denote the linear space
of all vector-valued measures on $(\Omega,\Sigma)$ of $Y$.

Let $X,Y$ be Banach spaces and $(\Omega,\Sigma,E,B(X,Y))$  an
operator-valued measure system. For any $B\in \Sigma$ and $x\in X,$
define
\[E_{ B, x }:\Sigma\to Y,\quad E_{ B, x }(A)=E(B\cap A)x,
\quad \forall A\in \Sigma.\]
 Then it is easy to see that
$E_{ B, x }$ is a vector-valued measure on $(\Omega,\Sigma)$ of
$Y$ and $E_{ B, x }\in \mathfrak{M}^Y_\Sigma.$

Let $M_E=\mbox{span}\{E_{ B, x }: x\in X, B\in \Sigma\}.$ Obviously, $M_E\subset\mathfrak{M}^Y_\Sigma,$ and we will refer it as
the space induced by $(\Omega,\Sigma,E,B(X,Y)).$ We first introduce some linear mappings on the spaces $X,$ $Y$ and $M_E.$

\begin{lemma}\label{le:417}
Let $X,Y$ be Banach spaces and $(\Omega, \Sigma, E,B(X,Y))$ an
operator-valued measure system. For any $\{C_i\}_{i=1}^N\subset
\mathbb{C},$ $\{B_i\}_{i=1}^N\subset \Sigma$ and
$\{x_i\}_{i=1}^N\subset X$, the mappings
$$S:M_E\to Y,\quad
S\left(\sum_{i=1}^NC_i E_{{B_i}, {x_i}}\right)=\sum_{i=1}^NC_i
E(B_i)x_i$$
$$T: X\to M_E,\quad T(x)=E_{\Omega,x}$$
and
$$F(B): M_E\to M_E,\quad F(B)\left(\sum_{i=1}^NC_i E_{ {B_i}, {x_i}}\right)
=\sum_{i=1}^NC_i E_{{B\cap B_i}, {x_i}}, \quad \forall B\in\Sigma$$
are well-defined and linear.
\end{lemma}
\begin{proof} Obviously $T$ is well-defined. For any
$\{C_i\}_{i=1}^N\subset \mathbb{C},$  $\{x_i\}_{i=1}^N\subset X$
and $\{B_i\}_{i=1}^N\subset \Sigma$, if
\[\sum_{i=1}^NC_i E_{ {B_i}, {x_i}}=0,\]
then, by definition of $E_{x_i,B_i}$, we have for any $A\in \Sigma$ that
\[\sum_{i=1}^NC_i E_{ {B_i}, {x_i}}(A)
=\sum_{i=1}^NC_i E(B_i\cap A)x_i=0.\] Let $B=\Omega.$ Then
\[\sum_{i=1}^NC_i E(B_i)x_i=0.\]
Hence $S$ is well-defined. Similarly we can verify that $F(B)$ is
also well-defined. Clearly for any $B\in \Sigma,$ $S,T$ and $F(B)$
are linear.
\end{proof}

With the aid of Lemma \ref{le:417}, we can now give the definition
of a dilation norm.

\smallskip

Note: In this manuscript we will in most cases use the traditional
notation $\|\cdot\|$ for a norm;  however, in the case of dilation
norms (especially) we will frequently find it convenient to use the
functional notation, typically $\cD(\cdot)$, for a norm, because of
the length of the expressions being normed. In this case we will
sometimes also write $\|\cdot\|_{\cD}$ for this same norm when the
meaning is clear, using the norming function $\cD$ to subscript the
traditional norm notation.

\begin{definition}\label{de:418}
Let $M_E$ be the space induced by $(\Omega,\Sigma,E,B(X,Y)).$ Let
$\|\cdot\|$ be a norm on $M_E$. Denote this normed space by
$M_{E,\|\cdot\|}$ and its completion $\widetilde{M}_{E,\|\cdot\|}.$
The norm on $\widetilde{M}_{E,\|\cdot\|}$, with $\|\cdot\| :=
\|\cdot\|_{\cD}$ given by a norming function $\cD$ as discussed
above, is called a \textbf{dilation norm of $E$}if the following
conditions are satisfied:
\begin{enumerate}
\item[(i)] The mapping $S_\cD: \widetilde{M}_{E,\cD}\to Y$ defined on $M_E$ by
\begin{eqnarray*}
S_\cD\left(\sum_{i=1}^NC_i E_{ {B_i}, {x_i}}\right)=\sum_{i=1}^NC_i E(B_i)x_i
\end{eqnarray*}
is bounded. \item[(ii)]
 The mapping $T_\cD: X\to \widetilde{M}_{E,\cD}$ defined by
\begin{eqnarray*}
T_\cD(x)=E_{\Omega,x}
\end{eqnarray*}
is bounded. \item[(iii)]
 The mapping $F_\cD: \Sigma\to B(\widetilde{M}_{E,\cD})$ defined
 by
\begin{eqnarray*}
F_\cD(B)\left(\sum_{i=1}^NC_i E_{ {B_i}, {x_i}}\right)
=\sum_{i=1}^NC_i E_{{B\cap B_i},{x_i}}
\end{eqnarray*}
is an operator-valued measure,
\end{enumerate}
where $\{C_i\}_{i=1}^N\subset \mathbb{C},$ $\{x_i\}_{i=1}^N\subset
X$ and $\{B_i\}_{i=1}^N\subset \Sigma$.
\end{definition}

\begin{theorem}\label{th:419}
Let $(\Omega,\Sigma,E,B(X,Y))$ be an operator-valued measure system.
If a norm $\cD$ is a dilation norm of $E,$ then the Banach space $
\widetilde{M}_{E,\cD}$ is a dilation space of
$(\Omega,\Sigma,E,B(X,Y)).$ Moreover,
$(\Omega,\Sigma,F_\cD,B(\widetilde{M}_{E,\cD}),S_\cD,T_\cD)$ is the
corresponding dilation projection-valued probability measure system.
\end{theorem}
\begin{proof}
For any $x\in X$ and $B_1,B_2,A\in \Sigma,$ we have
\begin{equation*}
F_\cD(B_1\cap B_2)(E_{A, x})=E_{B _1\cap B_2\cap A, x}
=F_\cD(B_1)(E_{B_2\cap A, x})=F_\cD(B_1)F_\cD(B_2) (E_{A, x}).
\end{equation*} So $F_\cD$ is a spectral
operator-valued measure.

For any $B\in \Sigma,$ we get
\[S_\cD F_\cD(B)T_\cD(x)=S_\cD F_\cD(B)(E_{\Omega,x})
=S_\cD(E_{B,x})=E(B)(x),\qquad \forall x\in X, \] and thus $S_\cD
F_\cD(E)T_\cD=E(B). $ Therefore
$(\Omega,\Sigma,F_\cD,B(\widetilde{M}_{\varphi,\cD}),S_\cD,T_\cD)$
is the corresponding dilation projection-valued measure system.

Observe that for any $x\in X$ and $B\in \Sigma,$ we have $F_\cD(\Omega)(E_{B,x})=E_{B,x}.$ Hence
$F_\cD(\Omega)=\mathrm{I}_{\widetilde{M}_{\varphi,\cD}}.$ The proof is complete.
\end{proof}

\begin{definition}\label{de:420}
Let $M_E$ be the space induced by $(\Omega,\Sigma,E,B(X,Y))$ and $\cD$ be a dilation norm of $E$. We call the Banach space
$\widetilde{M}_{E,\cD}$ \textbf{the elementary dilation space of $E$} and
$(\Omega,\Sigma,F_\cD,B(\widetilde{M}_{E,\cD}),S_\cD,T_\cD)$ the elementary dilation operator-valued measure system.
\end{definition}

In order to prove the main result (the existence of a dilation norm) in this section, we need the following lemma:

\begin{lemma}\label{le:421}
Let $X$ and $Y$ be Banach spaces, and let $(\Omega,\Sigma)$ be a
measurable space. Assume that $X_0\subset X$ is dense in $X.$ If a
mapping $E:\Sigma\to B(X,Y)$ is strongly countably additive on
$X_0$, i.e. if $\{B_i\}$ is a disjoint countable collection of
members of $\Sigma$ with union $B$ then
\begin{equation}\label{eq:41}
E(B)x=\sum_i E(B_i)x,\quad \ \forall x\in X_0,
\end{equation}
and if $E$ is uniformly bounded on $X_0$, i.e. there exists a
constant $C>0$ such that for any $B\in \Sigma$
\[\|E(B)x\|\leq C \|x\|,\quad \forall x\in X_0 ,\]
then $E$ is an operator-valued measure.
\end{lemma}
\begin{proof} Since $\overline{X_0}=X,$ $E$ is uniformly bounded
on $X.$ For any $N>0,$ by (\ref{eq:41}), if $\{B_i\}^{N}_{i=1}$ is
a disjoint collection of members of $\Sigma,$ then
\begin{equation}\label{eq:42}
E\left(\bigcup^{N}_{i=1}B_i\right)x=\sum^{N}_{i=1}
E(B_i)x,\quad \ \forall x\in X_0.
\end{equation}
Hence $E\left(\bigcup^{N}_{i=1}B_i\right)=\sum^{N}_{i=1}E(B_i)$ on
$X.$

We have
$$\left\|\sum^{N}_{i=1}E(B_i)\right\|=
\left\|E\left(\bigcup^{N}_{i=1}B_i\right)\right\|\leq C.$$

Let $\{B_i\}^{\infty}_{i=1}$ be a  countable disjoint collection of
elements of $\Sigma$ with union $B$. For any $x\in X,$ there exists
a sequence $\{x_j\}^{\infty}_{j=1}\subset X_0$ such that
$$\lim_{j\rightarrow\infty}x_j=x.$$
Observe that
\begin{eqnarray*}\label{eq:43}
&&\left\|E(B)x-\sum^{N}_{i=1} E(B_i)x\right\|\nonumber\\
&=&\left\|E(B)x-E(B)x_M+E(B)x_M-\sum^{N}_{i=1}
E(B_i)x_M+\sum^{N}_{i=1} E(B_i)x_M-\sum^{N}_{i=1}
E(B_i)x\right\|\nonumber\\
&\leq&\left\|E(B)x-E(B)x_M\right\|+\left\|E(B)x_M-\sum^{N}_{i=1}
E(B_i)x_M\right\|+\left\|\sum^{N}_{i=1}
E(B_i)x_M-\sum^{N}_{i=1} E(B_i)x\right\|
\nonumber\\
&\leq&\left\|E(B)\right\|\cdot\left\|x-x_M\right\|+\left\|E(B)x_M-\sum^{N}_{i=1}
E(B_i)x_M\right\|+\left\|E\left(\bigcup^{N}_{i=1}B_i\right)\right\|\cdot\left\|x_M-x\right\|
\nonumber\\
&\leq&2C\left\|x-x_M\right\|+\left\|E(B)x_M-\sum^{N}_{i=1}
E(B_i)x_M\right\|.\nonumber\\
\end{eqnarray*}

For any $\varepsilon>0,$ we can find $M>0$ such that
\begin{equation*}\label{eq:44}
\left\|x-x_M\right\|\leq\frac{\varepsilon}{3C}.
\end{equation*}
By (\ref{eq:41}), for $x_M\in X_0,$ we can find a sufficiently large
$N>0$ such that
\begin{equation*}\label{eq:45}
\left\|E(B)x_M-\sum^{N}_{i=1}
E(B_i)x_M\right\|\leq\frac{\varepsilon}{3}.
\end{equation*}
Therefore
\begin{equation*}
\left\|E(B)x-\sum^{N}_{i=1} E(B_i)x\right\|\leq \varepsilon
\end{equation*}
when $N$ is sufficiently large.  Hence
\begin{equation*}
E(B)x=\sum_i E(B_i)x,\quad \ \forall x\in X.
\end{equation*}
Therefore $E$ is an operator-valued measure.
\end{proof}

Now we state and prove the main result of this section.

\begin{theorem}\label{th:422}
Let $(\Omega,\Sigma,E,B(X,Y))$ be an operator-valued measure system and $(\Omega,\Sigma,F,B(Z),S,T)$ be a corresponding dilation
projection-valued measure space system. Then there exist an elementary dilation operator-valued measure system
$(\Omega,\Sigma,F_\cD,B(\widetilde{M}_{E,\cD}),S_\cD,T_\cD)$ of $E$ and a linear isometric embedding
\[U:\widetilde{M}_{E,\cD}\to Z\]
such that
\[S_\mathcal {D}=SU,\ F(\Omega)T=UT_\cD,\ UF_\cD(B)=F(B)U,
\quad \forall B\in \Sigma.\]
\end{theorem}
\begin{proof}
Define $\cD:M_E\to R^+\cup\{0\}$ by
\[\cD\left(\sum^{N}_{i=1}C_iE_{ {B_i}, {x_i}}\right)=\left\|\sum^{N}_{i=1}C_iF(B_i)T(x_i)\right\|_Z,\]
where $N>0,$ $C_i$ are all complex numbers and $\{x_i\}_{i=1}^{N}\subset X$ and $\{B_i\}_{i=1}^{N}\subset \Sigma.$ We first show that  $\cD$ is
a norm on $M_E.$

(i) Obviously,
\[\cD\left(\sum^{N}_{i=1}C_iE_{ {B_i}, {x_i}}\right)=\left\|\sum^{N}_{i=1}C_iF(B_i)T(x_i)\right\|_Z
\geq 0.\]

(ii) Take $\{\widetilde{C}_j\}_{j=1}^{M},$ $\{y_j\}_{j=1}^{M}\subset
X$ and $\{A_j\}_{j=1}^{M}\subset \Sigma.$ We have
\begin{eqnarray*}
\cD\left(\sum^{N}_{i=1}C_iE_{ {B_i},
{x_i}}+\sum^{M}_{j=1}\widetilde{C}_jE_{{A_j}, {y_j}}\right)
&=&\left\|\sum^{N}_{i=1}C_iF(B_i)T(x_i)+\sum^{M}_{j=1}\widetilde{C}_jF(A_j)T(y_j)\right\|_Z\\
&\leq&\left\|\sum^{N}_{i=1}C_iF(B_i)T(x_i)\right\|_Z+\left\|\sum^{M}_{j=1}\widetilde{C}_jF(A_j)T(y_j)\right\|_Z\\
&=&\cD\left(\sum^{N}_{i=1}C_iE_{ {B_i}, {x_i}}\right)+
\cD\left(\sum^{M}_{j=1}\widetilde{C}_jE_{{A_j}, {y_j}}\right).
\end{eqnarray*}

(iii) If
$\cD\left(\sum^{N}_{i=1}C_iE_{ {B_i}, {x_i}}\right)=0,$
then $\left\|\sum^{N}_{i=1}C_iF(B_i)T(x_i)\right\|_Z=0,$ and hence\\
$\sum^{N}_{i=1}C_iF(B_i)T(x_i)=0.$ This implies that for any $A\in\Sigma,$ we have
\begin{eqnarray*}
\sum^{N}_{i=1}C_iE_{ {B_i}, {x_i}}(A)
&=&\sum^{N}_{i=1}C_iE(B_i\cap A)(x_i)
=\sum^{N}_{i=1}C_iSF(B_i\cap A)T(x_i)\\
&=&SF(A)\sum^{N}_{i=1}C_i F(B_i)T(x_i)=0.
\end{eqnarray*}
Thus $\cD$ is a norm on $M_E.$

Denote this norm by $\|\cdot\|_\cD$, and let $\widetilde{M}_{E,\cD}$
be the completion of $M_{E,\cD}$ under this norm. We will show that
$\cD$ is a dilation norm of $E.$

First, since
\begin{eqnarray*}
&&\left\|S_\cD\left(\sum_{i=1}^NC_iE_{ {B_i}, {x_i}}\right)\right\|_Y
=\left\|\sum_{i=1}^NC_iE(B_i)x_i\right\|_Y\\
&=&\left\|\sum_{i=1}^NC_iSF(B_i)Tx_i\right\|_Y
=\left\|S\left(\sum_{i=1}^NC_iF(B_i)Tx_i\right)\right\|_Y\\
&\leq& \|S\|\cdot\left\|\sum_{i=1}^NC_iF(B_i)Tx_i\right\|_Z
=\|S\|\cdot\left\|\sum_{i=1}^NC_iE_{ {B_i}, {x_i}}\right\|_\cD,
\end{eqnarray*}
and
\begin{eqnarray*}
\|T_\cD(x)\|_\cD&=&\|E_{\Omega,x}\|_\cD
=\|F(\Omega)T(x)\|_Z\\
&\leq& \|F(\Omega)\|\cdot\|T\|\cdot\|x\|_X,
\end{eqnarray*}
 we have that the mappings
$$S_\cD: \widetilde{M}_{E,\cD}\to Y,\qquad
S_\cD\left(\sum_{i=1}^NC_iE_{ {B_i}, {x_i}}\right)=\sum_{i=1}^NC_iE(B_i)x_i
$$
and $$T_\cD: X\to \widetilde{M}_{E,\cD},\qquad
T_\cD(x)=E_{\Omega,x}$$ are both well-defined, linear and
bounded.

We prove that the mapping $F_\cD: \Sigma\to B(\widetilde{M}_{E,\cD})$ defined by
\begin{eqnarray*}
F_\cD(B)\left(\sum_{i=1}^NC_iE_{ {B_i}, {x_i}}\right)
=\sum_{i=1}^NC_iE_{{B\cap B_i}, {x_i},}
\end{eqnarray*}
is an operator-valued measure. By Lemma \ref{le:421}, we only need
to show that $F_\cD$ is strongly countably additive and uniformly
bounded on $M_{E,\cD}.$

If $\{A_j\}^\infty_{j=1}$ is a disjoint countable collection of
members of $\Sigma$ with union $A$, since $F$ is an
operator-valued measure, we have
\[F(A\cap B_i)T(x_i)=\sum_{j=1}^\infty F(A_j\cap B_i)T(x_i).\]
It follows that
\begin{eqnarray*}
F_\cD(A)(E_{B_i,x_i})=\sum_{j=1}^\infty F_\cD(A_j)(E_{B_i,x_i}).
\end{eqnarray*}
Hence
\begin{eqnarray*}
F_\cD(A)\left(\sum_{i=1}^NC_i E_{ {B_i}, {x_i}}\right)
=\sum_{j=1}^\infty F_\cD(A_j)
\left(\sum_{i=1}^NC_i E_{ {B_i}, {x_i}}\right).
\end{eqnarray*}
Thus $F_\cD$ is strongly countably additive on
$M_{E,\cD}.$

For any $A\in \Sigma,$ we have
\begin{eqnarray*}\label{eq:47}
&&\left\|F_\cD(A)\left(\sum^{N}_{i=1}C_iE_{B_i,x_i}\right)\right\|_\cD
=\left\|\sum^{N}_{i=1}C_iF_\cD(A)\left(E_{B_i,x_i}\right)\right\|_\cD\nonumber\\
&=&\left\|\sum^{N}_{i=1}C_iE_{A\cap B_i, x_i,}\right\|_\cD
=\left\|\sum^{N}_{i=1}C_i F(A\cap B_i)T(x_i)\right\|_Z\nonumber\\
&=&\left\|\sum^{N}_{i=1}C_iF(A)F(B_i)T(x_i)\right\|_Z
\leq\left\|F(A)\right\|\left\|\sum^{N}_{i=1}C_iF(B_i)T(x_i)\right\|_Z\nonumber\\
&\leq&\sup_{A\in\Sigma}\left\|F(A)\right\|\left\|\sum^{N}_{i=1}C_iE_{B_i,x_i}\right\|_\cD.
\end{eqnarray*}
By Remark \ref{re:38}, we know that
\[\sup_{A\in\Sigma}\left\|F(A)\right\|<+\infty,\]
and thus $F_\cD$ is uniformly bounded on $M_{E,\cD}.$

Define a mapping $U:\widetilde{M}_{E,\cD}\to Z$ by
\begin{eqnarray*}
U\left(\sum^{N}_{i=1}C_iE_{ {B_i}, {x_i}}\right)
=\sum^{N}_{i=1}C_iF(B_i)T(x_i).
\end{eqnarray*} It is easy to see
that $U$ is a well-defined and linear isometric embedding mapping.

Finally, we have
\begin{eqnarray*}
S_\cD\left(\sum_{i=1}^NC_iE_{ {B_i}, {x_i}}\right)
&=&\sum_{i=1}^NC_iE(B_i)x_i
=\sum_{i=1}^NC_iSF(B_i)T(x_i)\\
&=&S\sum_{i=1}^NC_iF(B_i)T(x_i)
=SU\left(\sum_{i=1}^NC_iE_{ {B_i}, {x_i}}\right),
\end{eqnarray*}
\begin{eqnarray*}
UT_\cD(x)=U\left(E_{\Omega,x}\right) =F(\Omega)T(x)
\end{eqnarray*}
and
\begin{eqnarray*}
UF_\cD(B)\left(\sum_{i=1}^NC_iE_{ {B_i}, {x_i}}\right)
&=&U\left(\sum_{i=1}^NC_iE_{{B\cap B_i}, x_i}\right)
=\sum^{N}_{i=1}C_iF(B\cap B_i)T(x_i)\\
&=&\sum^{N}_{i=1}C_iF(B)F(B_i)T(x_i)
=F(B)\left(\sum^{N}_{i=1}C_iF(B_i)T(x_i)\right)\\
&=&F(B)U\left(\sum_{i=1}^NC_iE_{ {B_i}, {x_i}}\right).
\end{eqnarray*}
Thus we get \[S_\mathcal {D}=SU,\ F(\Omega)T=UT_\cD,\ UF_\cD(B)=F(B)U, \quad \forall B\in \Sigma ,\] as claimed.
\end{proof}

 In general, we can find different dilation norms on $M_E,$ and
hence $E$ may have different dilation projection-valued measure systems. The following result shows that if $E$ is a projection
$B(X)$-valued probability measure, then these different dilation projection-valued measure systems are all similar, that is, $E$ is
isomorphic to $F_\cD$ for every dilation norm $\cD$ on $M_E.$

\begin{proposition}\label{pr:423}
Let $(\Omega,\Sigma,E,B(X))$ be a projection-valued
probability measure system and $\cD$ be a dilation norm of
$E$. Then $S_\cD$ and $T_\cD$ both are invertible operators
with $S_\cD T_\cD=\mathrm{I}_X$ and $T_\cD S_\cD=
\mathrm{I}_{\widetilde{M}_{E,\cD}}$. Thus, $E$ is
isomorphic to $F_\cD$ for every dilation norm $\cD$ of
$E$.
\end{proposition}
\begin{proof}
By Theorem \ref{th:419}, we know that
$(\Omega,\Sigma,F_\cD,B(\widetilde{M}_{E,\cD}),S_\cD,T_\cD)$
is a projection-valued probability measure space system. Hence
$F_\cD(\Omega)=\mathrm{I}_{\widetilde{M}_{E,\cD}}$ and
\[S_\cD T_\cD=S_\cD F_\cD(\Omega)T_\cD=E(\Omega)=\mathrm{I}_X.\]
By Remark \ref{re:412}, we know that $S_\cD$ is a surjection and
$T_\cD$ is an isomorphic embedding.

Let $N>0,$ $\{C_i\}_{i=1}^{N}\subset\mathbb{C},$
$\{x_i\}_{i=1}^{N}\subset X$ and $\{B_i\}_{i=1}^{N}\subset
\Sigma.$ If
\[S_\cD\left(\sum_{i=1}^NC_iE_{ {B_i}, {x_i}}\right)
=\sum_{i=1}^NC_iE(B_i)x_i=0,\] then for any $A\in \Sigma,$
\begin{eqnarray*}
\sum_{i=1}^NC_iE_{ {B_i}, {x_i}}(A) =\sum_{i=1}^NC_iE(B_i\cap A)(x_i) =E(A)\left(\sum_{i=1}^NC_iE(B_i)x_i\right) =0.
\end{eqnarray*}
So we get $\sum_{i=1}^NC_iE_{ {B_i}, {x_i}}=0, $ and hence $S_\cD$ is invertible.

Observe that for any $A\in \Sigma,$
\begin{eqnarray*}
E_{\Omega, \sum_{i=1}^NC_iE(B_i)x_i}(A)
=E(A)\left(\sum_{i=1}^NC_iE(B_i)x_i\right) =\sum_{i=1}^NC_i E_{
{B_i}, {x_i}}(A).
\end{eqnarray*}
Thus we have
\begin{eqnarray*}
&&T_\cD S_\cD\left(\sum_{i=1}^NC_i E_{ {B_i}, {x_i}}\right)
=T_\cD\left(\sum_{i=1}^NC_iE(B_i)x_i\right)\\
&=&E_{\Omega,\sum_{i=1}^NC_iE(B_i)x_i}
=\sum_{i=1}^NC_iE_{ {B_i}, {x_i}}.
\end{eqnarray*}
So we get $T_\cD S_\cD=\mathrm{I}_{\widetilde{M}_{\varphi,\cD}}$ and $T_\cD$ is invertible. Therefore $E$ is isomorphic to $F_\cD$ for
every dilation norm $\cD$ of $E$.
\end{proof}

\section{The Minimal Dilation Norm $\|\cdot\|_\alpha$}

While the connection between a corresponding dilation
projection-valued measure system and the  elementary dilation space
has been established in the previous section, we still need to
address the existence issue for a corresponding dilation
projection-valued measure system. So in this and the next two
sections we will focus on constructing two special dilation norms
for the (algebraic) elementary dilation space $M_{E}$ which we call
the minimal and maximal dilation norms.

\begin{definition}\label{de:421}
Define  $\|\cdot\|_\alpha:M_E\to \mathbb{R}^+\cup\{0\}$ by
\[\left\|\sum^{N}_{i=1}C_iE_{ {B_i}, {x_i}}\right\|_\alpha
=\sup_{B\in\Sigma}\left\|\sum_{i=1}^N C_iE(B\cap B_i)x_i\right\|_Y.\]
\end{definition}

\begin{proposition}\label{pr:425} $\|\cdot\|_\alpha$ is a norm on $M_E$.
\end{proposition}
\begin{proof}
Let $N>0,$ $\{C_i\}_{i=1}^{N}\subset\mathbb{C},$
$\{x_i\}_{i=1}^{N}\subset X$ and $\{B_i\}_{i=1}^{N}\subset
\Sigma.$ If $\sum^{N}_{i=1}C_iE_{ {B_i}, {x_i}}=0,$ then for
any $B\in\Sigma,$ $\sum_{i=1}^N C_iE(B\cap B_i)x_i=0,$ hence
$$\left\|\sum^{N}_{i=1}C_iE_{ {B_i}, {x_i}}\right\|_\alpha
=\sup_{B\in\Sigma}\left\|\sum_{i=1}^N C_iE(B\cap
B_i)x_i\right\|_Y =0.$$ Thus $\|\cdot\|_\alpha$ is well-defined.

If $\left\|\sum^{N}_{i=1}C_iE_{ {B_i}, {x_i}}\right\|_\alpha =0,$
then for any $B\in \Sigma,$ $\sum_{i=1}^N C_iE(B\cap B_i)x_i =0,$
thus
\[\sum^{N}_{i=1}C_iE_{ {B_i}, {x_i}}=0.\]

 Now we prove that $\|\cdot\|_\alpha$ satisfies the triangle
inequality. Let $M>0,$ $\{\widetilde{C}_j\}_{j=1}^{M}\subset\mathbb{C},$ $\{y_j\}_{j=1}^{M}\subset X$ and $\{A_j\}_{j=1}^{M}\subset \Sigma.$
Then
\begin{eqnarray*}
&&\left\|\sum^{N}_{i=1}C_iE_{ {B_i}, {x_i}}+\sum^{M}_{j=1}\widetilde{C}_jE_{{A_j}, {y_j}}\right\|_\alpha\\
&=&\sup_{B\in\Sigma} \left\|\sum_{i=1}^N C_iE(B_i\cap B)
x_i+\sum_{j=1}^M\widetilde{C}_jE(A_j\cap B)
y_j\right\|_Y\\
&\le&\sup_{B\in\Sigma} \left\|\sum_{i=1}^N C_iE(B_i\cap B)
x_i\right\|_Y+\sup_{B\in\Sigma}\left\|\sum_{j=1}^M
\widetilde{C}_jE(A_j\cap B)
y_j\right\|_Y\\
&=&\left\|\sum^{N}_{i=1}C_iE_{ {B_i}, {x_i}}\right\|_\alpha
+\left\|\sum^{M}_{j=1}C_jE_{{A_j}, {y_j}}\right\|_\alpha.
\end{eqnarray*}
So $\|\cdot\|_\alpha$ is a norm on $M_E$.
\end{proof}

We denote its normed space by $M_{E,\alpha}$ and its
completion by $\widetilde{M}_{E,\alpha}.$

\begin{theorem}\label{th:426}
$\|\cdot\|_\alpha$ is a dilation norm of $E$.
\end{theorem}
\begin{proof}Let $N>0,$ $\{C_i\}_{i=1}^{N}\subset\mathbb{C},$
$\{x_i\}_{i=1}^{N}\subset X$ and $\{B_i\}_{i=1}^{N}\subset \Sigma.$ Since the linear map $S_\alpha: \widetilde{M}_{E,\alpha}\to Y$ is defined by
$$S_\alpha\left(\sum^{N}_{i=1}C_iE_{ {B_i}, {x_i}}\right)=\sum_{i=1}^NC_iE(B_i) x_i,$$  we have that
\begin{eqnarray*}
&&\left\|S_\alpha\left(\sum^{N}_{i=1}C_iE_{ {B_i}, {x_i}}\right)\right\|_Y
=\left\|\sum_{i=1}^NC_iE(B_i) x_i\right\|_Y\\
& \leq &\sup_{B\in\Sigma}\left\|\sum_{i=1}^N C_iE(B\cap
B_i)x_i\right\|_Y=\left\|\sum^{N}_{i=1}C_iE_{ {B_i}, {x_i}}\right\|_\cM,
\end{eqnarray*}
which implies that $S_\alpha$ is bounded and $\|S_\alpha\|\le 1$.

Recall that the map $T_\alpha:X\to \widetilde{M}_{E,\alpha}$ is defined by $T_\alpha(x)=E_{\Omega,x}.$ Obviously,
we have
\begin{equation*}
\left\|T_\alpha\,x\right\|_\alpha=\left\|E_{\Omega,x}\right\|_\alpha =\sup_{B\in\Sigma}\left\|E(B) x\right\|_Y\le
\|E\|\cdot\|x\|,
\end{equation*}
and hence $\|T_\alpha\|\le \|E\|$.

Now we prove the mapping $F_{\alpha}: \Sigma\to B(\widetilde{M}_{E,\alpha})$ defined by
\begin{eqnarray*}
F_{\alpha}(B)\left(\sum_{i=1}^NC_iE_{ {B_i}, {x_i}}\right)
=\sum_{i=1}^NC_iE_{{B\cap B_i}, {x_i}}
\end{eqnarray*}
is an operator-valued measure. By Lemma \ref{le:421}, we only need to prove $F_{\alpha}$ is strongly countably additive and uniform
bounded on $M_{E,\alpha}.$

Since
\begin{eqnarray*}
&&\left\|F_{\alpha}(B)\left(\sum_{i=1}^NC_iE_{ {B_i}, {x_i}}\right)\right\|_\alpha
=\left\|\sum_{i=1}^NC_iE_{{B\cap B_i}, {x_i}}\right\|_\alpha\\
&=&\sup_{B'\in\Sigma}\left\|\sum_{i=1}^NC_iE(B_i\cap B\cap
B')(x_i)\right\|_Y
=\sup_{B'\in\Sigma}\left\|\sum_{i=1}^NC_i E(B_i\cap B')x_i\right\|_Y\\
&=&\left\|\sum^{N}_{i=1}C_iE_{ {B_i}, {x_i}}\right\|_Y,
\end{eqnarray*}
we get that $\|F_\alpha(B)\|=1$ for any $B\in \Sigma$.

For the strong countable additivity,  let $\{A_j\}_{j=1}^\infty$ be
a countable disjoint collection of members of $\Sigma$ with union
$A$. Then we have
\begin{eqnarray*}
&&\left\|\sum_{j=1}^MF_{\alpha
}(A_j)\left(\sum_{i=1}^NC_iE_{{B_i},{ x_i}}\right)-F_\alpha(A)\left(\sum_{i=1}^NC_iE_{{B_i},{ x_i}}\right)\right\|_\alpha\\
&=&\left\|\sum_{j=1}^M\left(\sum_{i=1}^NC_iE_{{A_j\cap
A_i},{x_i}}\right)-\left(\sum_{i=1}^NC_iE_{{A\cap B_i},{x_i}}\right)\right\|_\alpha\\
&=&\left\|\sum_{i=1}^NC_i\left(\sum_{j=1}^ME_{{A_j\cap
B_i},{x_i}}-E_{{A\cap B_i},{x_i}}\right)\right\|_\alpha\\
&=&\sup_{B'\in\Sigma}\left\|\sum_{i=1}^NC_i\left(\sum_{j=1}^ME(A_j\cap
B_i\cap B')x_i-E(A\cap
B_i\cap B')x_i\right)\right\|_Y\\
&=&\sup_{B'\in\Sigma}\left\|\sum_{i=1}^NC_iE\left(\bigcup_{j=M+1}^\infty
\left(A_j\cap B_i\cap B'\right)\right)x_i\right\|_Y\\
&\le&\sum_{i=1}^N|C_i|\sup_{B'\in\Sigma}\left\|B\left(\bigcup_{j=M+1}^\infty
\left(A_j\cap B_i\cap B'\right)\right)x_i\right\|_Y\\
&=&\sum_{i=1}^N|C_i|\sup_{B'\in\Sigma}\left\|E\left(\bigcup_{j=M+1}^\infty
\left(A_j\cap B'\right)\right)x_i\right\|_Y\\
&=&\sum_{i=1}^N|C_i|\sup_{B'\in\Sigma}\left\|\sum_{j=M+1}^\infty E(A_j\cap
B')x_i\right\|_Y.
\end{eqnarray*}
If $\sup_{B'\in\Sigma}\left\|\sum_{j=M+1}^\infty E(A_j\cap B')x_i\right\|_Y$ does not tend to 0 as $M\rightarrow \infty,$ then we can find
$\delta>0$, a sequence of $n_1\le m_1<n_2\le m_2<n_3\le m_3<.\,.\,.$ , and $\{B_l'\}_{l=1}^\infty\subset \Sigma$ such that
\begin{eqnarray*}
\left\|\sum_{j=n_l}^{m_l} E(A_j\cap B_l')x_i\right\|\ge
\delta, \ \ \forall\, l\in\N.
\end{eqnarray*}
Since for $l\in\N$ and $n_l\le j\le m_l$, $A_j\cap B_l'$ are disjoint from each other, we have
\begin{eqnarray*}
E\left(\bigcup_{l=1}^\infty\bigcup_{j=n_l}^{m_l} A_j\cap
B_l'\right)x_i=\sum_{l=1}^\infty\sum_{j=n_l}^{m_l} E(B_j\cap
B_l')x_i,
\end{eqnarray*}
which implies
 $\left\|\sum_{j=n_i}^{m_i} E(B_j\cap
B_i')x_i\right\|\rightarrow 0, $ which is a contradiction. Hence
\[F_\alpha(A)\left(\sum_{i=1}^NC_iE_{{B_i},{ x_i}}\right)=\sum_{j=1}^\infty F_{\alpha}(A_j)\left(\sum_{i=1}^NC_iE_{{B_i},{ x_i}}\right),\]
 as expected.
\end{proof}

Combining Theorem \ref{th:426} and Theorem \ref{th:419} we have the following:

\begin{theorem} \label{main-thm} For any operator-valued
measure system $(\Omega,\Sigma,E,B(X,Y))$, there exist a Banach space $Z$,  two bounded linear operators $S:Z\to Y$ and $T:X\to Z$, and a
projection-valued probability measure system $(\Omega,\Sigma,F,B(Z))$ such that $E(B)=SF(B)T$ for any $B\in\Sigma$. In other words,
every operator-valued measure can be dilated to a projection-valued measure.
\end{theorem}

\begin{lemma}\label{le:424}
Let $(\Omega,\Sigma,E,B(X,Y))$ be an operator-valued measure system. If $\cD$ is a dilation norm of $E,$ then there exists a
constant $C_\cD$ such that for any $\sum_{i=1}^NC_iE_{ {B_i}, {x_i}}\in M_{E,\cD},$
\[
\sup_{B\in\Sigma}\left\|\sum_{i=1}^N C_iE(B\cap B_i)x_i\right\|_Y \leq C_\cD \left\|\sum_{i=1}^NC_iE_{ {B_i}, {x_i}}\right\|_\cD,
\]
where $N>0,$ $\{C_i\}_{i=1}^{N}\subset\mathbb{C},$ $\{x_i\}_{i=1}^{N}\subset X$ and $\{B_i\}_{i=1}^{N}\subset \Sigma. $ Consequently we have
that
\[ \|f\|_\alpha\le C_\cD\|f\|_\cD, \qquad \forall f\in M_{E}.\]
\end{lemma}
\begin{proof}
Let $C_\cD=\sup_{B\in\Sigma}\left\|S_\cD F_\cD(B)\right\|.$ Then obviously $C_\cD<+\infty$. We also have
\begin{eqnarray*}
\sup_{B\in\Sigma}\left\|\sum_{i=1}^N C_iE(B\cap B_i)x_i\right\|_Y &=&\sup_{B\in\Sigma}\left\|\sum_{i=1}^N C_i S_\cD F_\cD(B\cap
B_i)T_\cD (x_i)\right\|_Y\\
&=&\sup_{B\in\Sigma}\left\|\sum_{i=1}^N
C_iS_\cD F_\cD(B)F_\cD(B_i)T_\cD (x_i)\right\|_Y\\
&\leq&\sup_{B\in\Sigma}\left\|S_\cD F_\cD(B)\right\|
\left\|\sum_{i=1}^N C_iF_\cD(B_i)T_\cD (x_i)\right\|_\cD \\
&=&C_\cD \left\|\sum_{i=1}^N C_iE_{ {B_i}, {x_i}}\right\|_\cD.
\end{eqnarray*}
\end{proof}

This lemma justifies the following definition for the dilation norm $\alpha$.

\begin{definition}\label{de:427}
Let $(\Omega,\Sigma,E,B(X,Y))$ be an operator-valued measure
system. We call $\alpha$  the \emph{minimal dilation norm}
of $E$ and $\widetilde{M}_{E,\alpha}$ the $\alpha$-dilation
space.
\end{definition}

The next result shows that if we start with a projection-valued
probability measure $E$, then the dilated projection-valued
probability measure on the elementary dilation space
$\widetilde{M}_{E,\alpha}$ is isometric to $E$.

\begin{proposition}\label{pr:429}
Let $(\Omega,\Sigma,E,B(X))$ be a normalized
projection-valued probability measure system and $\alpha$ be the
minimal dilation norm of $E$. Then $S_\alpha$ and
$T_\alpha$ both are isometries and $E$ is isometric to
$F_\alpha$.
\end{proposition}
\begin{proof}
Since $\alpha$ is a dilation norm of $E,$ by Proposition \ref{pr:423}, $S_\alpha$ and $T_\alpha$ both are invertible operators with $S_\alpha
T_\alpha=\mathrm{I}_X$ and $T_\alpha S_\alpha= \mathrm{I}_{\widetilde{M}_{E,\alpha}}$. For any $x\in X,$ we have
\[\|T_\alpha x\|_\alpha=\|E_{\Omega,x} \|_\alpha
=\sup_{B\in\Sigma}\left\|E(B)x\right\| \leq
\sup_{B\in\Sigma}\left\|E(B)\right\|\|x\| =\|x\|,
\]
and thus $\|T_\alpha \|\leq 1.$

On the other hand, for any $N>0,$  $\{C_i\}_{i=1}^{N}\subset\mathbb{C},$ $\{x_i\}_{i=1}^{N}\subset X$ and $\{B_i\}_{i=1}^{N}\subset \Sigma,$ we
have that
\begin{eqnarray*}
&&\left\|S_{\alpha}\left(\sum_{i=1}^NC_iE_{{B_i},{ x_i}}\right)\right\|_X
=\left\|\sum_{i=1}^NC_iE(B_i)x_i\right\|_X\\
&\le&\sup_{B'\in\Sigma}\left\|\sum_{i=1}^NC_iE\left(B_i\cap
B'\right)x_i\right\|_X =\left\|\sum_{i=1}^NC_iE_{{B_i},{ x_i}}\right\|_\alpha,
\end{eqnarray*}
and so $\left\|S_{\alpha}\right\|\leq 1.$

Since $S_\alpha T_\alpha=\mathrm{I}_X$ and $T_\alpha S_\alpha=
\mathrm{I}_{\widetilde{M}_{E,\alpha}}$, we have that for each $x\in
X$ we have $$||x|| = ||S_{\alpha}T_{\alpha}x|| \leq ||T_{\alpha}x||
\leq ||x||.$$ Hence $T_{\alpha}$ is a surjective isometry. Similarly
$S_\alpha$ is also a surjective isometry, and therefore $E$ is
isometric to $F_\alpha$.
\end{proof}

\begin{proposition}\label{pr:430}
Let $(\Omega,\Sigma,E,B(X,Y))$ be an operator-valued measure system and $\alpha$ be the minimal dilation norm of $E$. If $P_j:Y\to Y$
are norm one projections  such that $\sum^{M}_{j=1} P_j=\mathrm{I}_{Y}$, then
$$E_j(B):\Sigma\to B(X,Y), \quad E_j(B)x=P_jE(B)x,\quad x\in X$$
is an operator-valued measure for $1\le j\le M$. Moreover for any $\sum_{i=1}^NC_iE_{{B_i},{ x_i}}\in M_{E}$, we have
\begin{eqnarray*}
\max_{1\le j\le M} \left\|\sum_{i=1}^NC_i(E_j)_{{
B_i},{x_i}}\right\|_{\alpha_j}\le \left\|\sum_{i=1}^NC_iE_{{B_i},{ x_i}}\right\|_{\alpha}\le \sum_{1\le j\le M}
\left\|\sum_{i=1}^NC_i(E_j)_{{ B_i},{x_i}}\right\|_{\alpha_j},
\end{eqnarray*}
where $\alpha_j$ is the minimal dilation norm of $E_j$.
\end{proposition}
\begin{proof} It is obvious that for each $1\le j\le M,$
$$E_j(B):\Sigma\to B(X,Y), \quad E_j(B)x=P_jE(B)x,\quad x\in X$$
is an operator-valued measure.

The ``moreover"  part follows from:
\begin{eqnarray*}
 \left\|\sum_{i=1}^NC_i(E_j)_{{
B_i},{x_i}}\right\|_{\alpha_j}
&=&\sup_{B'\in\Sigma}\left\|\sum_{i=1}^NC_iP_jE\left(B_i\cap
B'\right)x_i\right\|\\
&=&\sup_{B'\in\Sigma}\left\|P_j\left(\sum_{i=1}^NC_iE\left(B_i\cap
B'\right)x_i\right)\right\|\\
&\leq&\|P_j\|\sup_{B'\in\Sigma}\left\|\sum_{i=1}^NC_iE\left(B_i\cap
B'\right)x_i\right\|\\
&=&\sup_{B'\in\Sigma}\left\|\sum_{i=1}^NC_iE\left(B_i\cap
B'\right)x_i\right\|\\
&=&\sup_{B'\in\Sigma}\left\|\sum^{M}_{j=1}
P_j\left(\sum_{i=1}^NC_iE\left(B_i\cap
B'\right)x_i\right)\right\| \\
&\leq&\sum^{M}_{j=1}\sup_{B'\in\Sigma}\left\|\sum_{i=1}^NC_iP_j\left(E\left(B_i\cap
B'\right)x_i\right)\right\|\\
&=&\sum_{1\le j\le M} \left\|\sum_{i=1}^NC_i(E_j)_{{
B_i},{x_i}}\right\|_{\alpha_j}.
\end{eqnarray*}
\end{proof}

\begin{lemma} \label{le:new37} If  $\{i\}$ is an atom in $\Sigma$, then the rank of $F_\alpha(\{i\})$ is equal to the rank of  $E(\{i\})$.
\end{lemma}
\begin{proof}  Suppose that the rank of  $E(\{i\})$ is $k$. Then there exist $x_{1}, ... , x_{k}\in X$ such that $range\ E(\{i\}) =
span \{E(\{i\})x_{1}, ..., E(\{i\})x_{k}\}$. We show that $range\
F(\{i\}) = span \{E_{x_{1}, \{i\}}, ..., E_{x_{k}, \{i\}} \}$. In
fact, for any $x\in X, B\in \Sigma$, we have that $F(\{i\})E_{
B,x} = E_{\{i\},x}$ or $0$. Now for any $A\in \Sigma$, we have $
E_{\{i\},x}(A) = E(\{i\})x$ if $i\in A$ and $0$ if $i\notin A$.
Write  $E(\{i\})x = \sum_{j=1}^{k}c_{j}E(\{i\})x_{j}$. Then we get
$E_{\{i\},x}(A)  = \sum_{j=1}^{k}c_{j}E(\{i\})x_{j} =
\sum_{j=1}^{k}c_{j}E_{\{i\},x_{j}}(A)$ if $i\in A$ and $0$
otherwise. Hence $$E_{\{i\}, x}  = \sum_{j=1}^{k}c_{j}E_{\{i\},
x_{j}},
$$ and therefore $range \ F(\{i\}) = span \{E_{\{i\},x_{1}}, ...,
E_{\{i\},x_{k}},\}$ as claimed.
\end{proof}

\noindent{\bf Problem A.} Is it always true that with an appropriate
notion of rank function for an operator valued measure, that $r (F
(B)) = r( E (B))$ for every $B\in \Sigma$? The previous lemma tells
us that it is true when $B$ is an atom. A possible ``rank "
definition might be: $r(B) = \sup \{rank E(A): A\subset B,
A\in\Sigma\}$.

\begin{example}\label{ex:435}
Assume that $(\N, 2^\N, E,B(X))$ is an induced operator-valued
probability measure system by a framing $\{x_i,y_i\}_{i\in\N}$ of
$X$,  where $E(B)=\sum_{i\in B} x_i\otimes y_i$ for every
$B\in 2^\N$.

We characterize its minimal $\alpha$-dilation space
$\widetilde{M}_{E,\alpha}.$ Let $(\N,2^\N,
F_\alpha,B(\widetilde{M}_{E,\alpha}))$ be the corresponding spectral
operator-valued probability measure. Then we have
\begin{eqnarray*}
f=\sum_{i\in\N} F_\alpha(\{i\})f, \ \ \forall\, f\in
\widetilde{M}_{E,\alpha}.
\end{eqnarray*}

By Lemma \ref{le:new37}, all $F_\alpha(\{i\})$'s are rank-one projections. Choose $\tilde{x}_i$ such that $y_i(\tilde{x}_i)=1$. Then
\begin{eqnarray*}
F_\alpha(\{i\})\left(E_{ \{i\},\tilde{x}_i}\right)=E_{ \{i\},\tilde{x}_i},\quad  \left\|
E_{ \{i\},\tilde{x}_i}\right\|_\alpha=\|x_i\|, \ \ \forall\,
i\in\N.
\end{eqnarray*}
And because $F_\alpha(\{i\})$'s are projections with
$F_\alpha(\{i\})F_\alpha(\{j\})=\delta_{ij}F_\alpha(\{i\})$,
we know that $\{E_{ \{i\},\tilde{x}_i}\}$ is a basis of
$\widetilde{M}_{E,\alpha},$ which is just our minimal framing
model basis. Actually, for every $(a_i)\in c_{00}$, we have
\begin{eqnarray*}
\left\|\sum_{i\in\N} a_i E_{ \{i\},\tilde{x}_i}\right\|_\alpha
&=& \sup_{B\subset \N}\left\|\sum_{i\in\N} a_i E(\{i\}\cap
B)\tilde{x}_i\right\|\\
&=&\sup_{B\subset \N}\left\|\sum_{i\in B} a_i E(\{i\})\tilde{x}_i\right\|\\
&=&\sup_{B\subset \N}\left\|\sum_{i\in B} a_i (x_i\otimes y_i)\tilde{x}_i\right\|\\
&=&\sup_{B\subset \N}\left\|\sum_{i\in B} a_i y_i(\tilde{x}_i) x_i\right\|\\
&=&\sup_{B\subset \N}\left\|\sum_{i\in B} a_i  x_i\right\|,
\end{eqnarray*}
which is equivalent to the minimal framing model basis because of
the following inequality.
\begin{eqnarray*}
\frac{1}{2}\sup_{\epsilon_i=\pm1}\left\|\sum_{i\in \N} \epsilon_i
a_i x_i\right\|\le \sup_{B\subset \N}\left\|\sum_{i\in B} a_i
x_i\right\|\le \sup_{\epsilon_i=\pm1}\left\|\sum_{i\in \N}
\epsilon_i a_i x_i\right\|.
\end{eqnarray*}
This is because
\begin{eqnarray*}
\frac{1}{2}\sup_{\epsilon_i=\pm1}\left\|\sum_{i\in \N} \epsilon_i
a_i x_i\right\| &=&
\frac{1}{2}\sup_{\epsilon_i=\pm1}\left\|\sum_{{i\in
\N}\atop{\epsilon_i=1}}a_i x_i-\sum_{{i\in
\N}\atop{\epsilon_i=-1}}a_i x_i\right\|\\
&\leq&\frac{1}{2}\sup_{\epsilon_i=\pm1}\left(\sup_{B\subset
\N}\left\|\sum_{{i\in B}}a_i x_i\right\|+\sup_{B\subset
\N}\left\|\sum_{{i\in
B}}a_i x_i\right\|\right)\\
&=&\sup_{B\subset
\N}\left\|\sum_{{i\in B}}a_i x_i\right\|\\
&=&\frac{1}{2}\sup_{B\subset
\N}\left\|\sum_{{i\in\N}\atop{\epsilon_i=1}}\epsilon_ia_i
x_i+\sum_{{\epsilon_i=1, i\in B
}\atop{\epsilon_i=-1, i\not\in B}}\epsilon_ia_i x_i\right\|\\
&\leq&\sup_{B\subset \N}\sup_{\epsilon_i=\pm1}\left\|\sum_{i\in \N}
\epsilon_i a_i x_i\right\|\\
&=&\sup_{\epsilon_i=\pm1}\left\|\sum_{i\in \N} \epsilon_i a_i
x_i\right\|.
\end{eqnarray*}
\end{example}

The above example shows that in the case that the operator-valued probability measure system is induced by a framing $\{x_i,y_i\}_{i\in\N}$ of
$X$, then the minimal $\alpha$-dilation space is exactly the one constructed in \cite{CHL} and hence it is always separable. However, it is
not clear whether this is true in general. So we ask:

\bigskip

\noindent{\bf Problem B.} Assume that  $(\Omega,\Sigma,E,B(X))$ is  an operator-valued measure system such that $X$ is a separable Banach
space. Is the dilation space $\alpha$ equipped with the minimal dilation norm always separable? If not, does there exist a dilation space which is
separable?

\bigskip

We end this section with a result concerning this problem. We say
that an operator valued measure $E$ on $(\Omega, \Sigma)$ into
$B(X)$ is {\it totally strongly countably additive} if whenever
$\{B_{n}\}$ is a disjoint sequence in $\Sigma$ with union $B$, and
$x\in X$, then given any $\epsilon > 0$, there exists $N\in\Bbb{N}$
such that
$$
sup_{A\in\Sigma} ||\sum_{j=n}^{\infty}E(B_{j}\cap A)x|| < \epsilon
$$
holds for all $n\geq N$. We say that $(\Omega, \Sigma, E)$ has a {\it countable Borel basis}  if there exists a countable set
$\mathcal{B}$ of subsets from $\Sigma$ such that for each $B\in\Sigma$, $B = \sum_{n=1}^{\infty}B_n$ for some  subsets $B_{n}$ from $\mathcal{B}$
up to a $E$-null set (A $E$-null set is a set $C\in\Sigma)$ such that $E(A) = 0$ if $A\subseteq C$ and $A\in\Sigma$).

\begin{proposition} Assume that an operator valued measure $E$ on $(\Omega, \Sigma)$ into $B(X)$ is
totally strongly countably additive, and that $(\Omega, \Sigma, E)$ has a  countable Borel basis. If  $X$ is separable, then so is the
dilation space $\alpha$ with the minimal dilation norm.
\end{proposition}
\begin{proof} Fix $x\in X$, and let $\{B_{i}\}$ be a countable Borel basis for $(\Omega, \Sigma, E)$. Then it is easy to prove that
$\overline{\text{span}}\{E_{B_{i}, x}: i\in\Bbb{N}\}$ contains  $E_{B, x}$ for every $B\in\Sigma$. Now fix $B\in\Sigma$, and let
$\{x_{i}\}$ be a dense subset of $X$. Then $\overline{\text{span}}\{E_{B, x_{i}}: i\in\Bbb{N}\}$ contains  $E_{B, x}$ for every
$x\in X$. Therefore we get that $span \{E_{B_{i}, x_{j}}: i, j\in\Bbb{N}\}$ is dense in $\alpha$ and hence $\alpha$ is separable.
\end{proof}

\section{The Dual Space of the Minimal Dilation Space}

In this section we give a concrete description for the dual space of the $\alpha$-dilation space. Let $(\Omega,\Sigma,E,B(X,Y))$ be
an operator-valued measure system, $Y^*$ be the dual space of $Y$ and $\widetilde{M}^*_{E,\alpha}$ be the dual space of
$\widetilde{M}_{E,\alpha}$. For any $y^*\in Y^*$ and $ B\in \Sigma,$ define a mapping $\phi_{B,y^*}:\widetilde{M}_{E,\alpha}\to
\mathbb{C}$ by
\[\phi_{B,y^*}\left(\sum_{i=1}^NC_iE_{{B_i},{ x_i}}\right)=y^*\left(\sum_{i=1}^NC_iE(B\cap B_i)(x_i)\right).\]  If $\sum_{i=1}^NC_iE_{{B_i},{ x_i}}=0,$ then
\begin{eqnarray*}
\phi_{B,y^*}\left(\sum_{i=1}^NC_iE_{{B_i},{ x_i}}\right) =y^*\left(\sum_{i=1}^NC_i E
(B\cap B_i)(x_i)\right)
=y^*\left(\sum_{i=1}^NC_iE_{{B_i},{ x_i}}(B)\right) =0.
\end{eqnarray*}
Thus the mapping $\phi_{B,y^*}:\widetilde{M}_{E,\alpha}\to \mathbb{C}$ is well-defined.

\begin{proposition}\label{pr:431} Let $\mathscr{M}_E=\mbox{span}\{\phi_{B,y^*}: y^*\in
Y^*, B\in \Sigma\}.$  Then $\mathscr{M}_E\subset
\widetilde{M}^*_{E,\alpha}.$
\end{proposition}
\begin{proof}
It is sufficient to prove that $\phi_{B,y^*}$ is bounded on $M_{E,\alpha}$. In fact, this follows immediately from the following:
\begin{eqnarray*}
&&\left|\phi_{B,y^*}\left(\sum_{i=1}^NC_iE_{{B_i},{ x_i}}\right)\right|\\
&=&\left|y^*\left(\sum_{i=1}^NC_iE(B\cap
B_i)(x_i)\right)\right| \\
&\leq& \|y^*\| \left\|\sum_{i=1}^NC_iE(B_i\cap B)
x_i\right\|\\
&\leq& \|y^*\| \left\|\sum_{i=1}^NC_iE_{{B_i},{ x_i}}\right\|_\alpha .
\end{eqnarray*}
\end{proof}

We will denote the norm on $\widetilde{M}^*_{E,\alpha}$ by $\alpha^*$. Then the space $\mathscr{M}_E$ endowed with the norm
$\alpha^*$ is a normed space, which will be denoted by $\mathscr{M}_{E,\alpha^*}$.

\begin{proposition}\label{pr:432}
For every $\sum_{i=1}^NC_i\phi_{{ B_i},{y^*_i}}\in
\mathscr{M}_{E,\alpha^*}$, we have
\begin{eqnarray*}
\left\|\sum_{i=1}^NC_i\phi_{{B_i},{y^*_i}}\right\|_{\alpha^*} \le
\sum_{i=1}^N |C_i|\|y_i^*\|.
\end{eqnarray*}
\end{proposition}
\begin{proof}
First, for every $\phi_{{y^*},{ B}}\in
\mathscr{M}_{E,\alpha^*}$, by the proof of Proposition
\ref{pr:431}, we have
\begin{eqnarray*}
\left\|\phi_{{ B},{y^*}}\right\|\le \|y^*\|.
\end{eqnarray*}
Thus, for every $\sum_{i=1}^NC_i\phi_{{ B_i},{y^*_i}}\in
\mathscr{M}_{E,\alpha^*}$, we get
\begin{eqnarray*}
\left\|\sum_{i=1}^NC_i\phi_{{B_i},{y^*_i}}\right\|_{\alpha^*} \le \sum_{i=1}^N|C_i|\left\|\phi_{{B_i},{y^*_i}}\right\|\leq\sum_{i=1}^N
|C_i|\|y_i^*\|.
\end{eqnarray*}
\end{proof}

Since $\alpha^*$ is a norm on $\mathscr{M}_{E}$, we denote the completion of $\mathscr{M}_{E}$ under norm $\alpha^*$ by
$\widetilde{\mathscr{M}}_{E,\alpha^*}.$ Then $\widetilde{\mathscr{M}}_{E,\alpha^*}$ is a Banach space, and so
$\widetilde{\mathscr{M}}_{E,\alpha^*}\subset \widetilde{M}^*_{E,\alpha}.$

\begin{proposition}\label{pr:434}
$\widetilde{M}^*_{E,\alpha}
=\overline{\widetilde{\mathscr{M}}_{E,\alpha^*}}^{w^*}.$
\end{proposition}
\begin{proof}
By the Hahn-Banach Separation Theorem with respect to the
$w^*$-topology, it is enough to show that
$\widetilde{\mathscr{M}}_{E,\alpha^*}$ separates
$\widetilde{M}_{E,\alpha}$. If not, then there is an $f\in
\widetilde{M}_{E,\alpha}$ with $\|f\|_\alpha=1$ such that $g^*(f)=0$
for every $g^*\in \widetilde{M}^*_{E,\alpha}$. For any
$0<\epsilon<1/2$, we can find a $g\in M_{E,\alpha}$ with
$\|g\|_\alpha=1$ and $\|f-g\|_\alpha\leq \epsilon$. Let
$g=\sum_{i=1}^N C_iE_{B_i,x_i}$ be a representation of $g$. By the
definition of $\alpha$-dilation norm, there is $E\in\Sigma$
satisfying
\begin{eqnarray*}
\left\|\sum_{i=1}^NC_iE(B_i\cap B)(x_i)\right\|\ge
1-\epsilon.
\end{eqnarray*}
Take an $y^*\in Y^*$ such that
\begin{eqnarray*}
y^*\left(\sum_{i=1}^NC_iE(B_i\cap
B)x_i\right)=\left\|\sum_{i=1}^NC_iE(B_i\cap B)x_i\right\|\ge
1-\epsilon.
\end{eqnarray*}
Then for $\phi_{B,y^*}\in\mathscr{M}_{E,\alpha^*}$,
\begin{eqnarray*}
\phi_{B,y^*}(g)=\phi_{B,y^*}\left(\sum_{i=1}^N
C_iE_{B_i,x_i}\right)=y^*\left(\sum_{i=1}^NC_iE(B_i\cap
B)x_i\right)\ge 1-\epsilon
\end{eqnarray*}
and $\left\|\phi_{B,y^*}\right\|_{\alpha^*}\le \|y^*\|=1$. Thus, we
have
\begin{eqnarray*}
\phi_{B,y^*}(f)&=&\phi_{B,y^*}(f-g)+\phi_{B,y^*}(g)\\
&\ge&1-\epsilon-\|\phi_{B,y^*}\|_{\alpha^*}\|f-g\|\\
&\ge&1-2\epsilon\\
&>&0,
\end{eqnarray*}
which leads to a contradiction. Thus
$\widetilde{M}^*_{E,\alpha}
=\overline{\widetilde{\mathscr{M}}_{E,\alpha^*}}^{w^*}.$
\end{proof}

\section{The Maximal Dilation Norm $\|\cdot\|_\omega$}

In this section, we construct the ``maximal " dilation norm for the elementary dilation space $M_{E}$. Let
$(\Omega,\Sigma,E,B(X,Y))$ be an operator-valued measure system. Consider the basic elements $E_{ B, x }\in M_{E}$. It is natural
to require that
\begin{equation}\label{eq:46}
\left\|E_{ B, x }\right\|\le
\sup_{B'\in\Sigma}\left\|E(B\cap B')x\right\|.
\end{equation}
Now let $f$ be any element of $M_{E}$. If $\sum_{i=1}^N C_iE_{B_i,x_i}$ is a representation of $f$, then it follows from the
triangle inequality that
\begin{equation*}
\|f\|\le \sum_{i=1}^N \sup_{B\in\Sigma}\|C_iE(B_i\cap
B)x_i\|.
\end{equation*}
Since this holds for every representation of $f$,  we get
\begin{equation*}
\|f\|\le \inf\Big\{\sum_{i=1}^N
\sup_{B\in\Sigma}\|C_iE(B_i\cap B)x_i\|\Big\},
\end{equation*}
where the infimum is taken over all representation of $f$. Define $\|\cdot\|_\omega:M_{E}\to \mathbb{R}^+\cup\{0\}$ by
\begin{equation*}
\|f\|_\omega=\inf\left\{\sum_{i=1}^N
\sup_{B\in\Sigma}\left\|C_iE(B_i\cap B)x_i\right\|:
f=\sum_{i=1}^N C_iE_{B_i,x_i}\in M_{E}\right\}.
\end{equation*}

Then we have the following proposition.

\begin{proposition}\label{pr:436}
$\|\cdot\|_\omega$ is a semi-norm on $M_{E}$ and
\begin{equation*}
\left\|E_{ B, x }\right\|_\omega=\sup_{E'\in\Sigma}\|E(B\cap
B')x\|
\end{equation*}
for every $B\in\Sigma$ and $x\in X$.
\end{proposition}
\begin{proof}
First, we show that $\left\|\lambda f\right\|_\omega=|\lambda|\left\|f\right\|_\omega$. This is obvious when $\lambda$ is zero.
So suppose that $\lambda\ne
0$. If $f=\sum_{i=1}^N C_iE_{B_i,x_i}$ is a representation of $f$, then
$$\lambda f=\sum_{i=1}^N \lambda C_iE_{B_i,x_i}
=\sum_{i=1}^N C_iE_{B_i, \lambda x_i},$$ and so we have
\begin{equation*}
\left\|\lambda f\right\|_\omega\le \sum_{i=1}^N
\sup_{B\in\Sigma}\|C_iE(B_i\cap B)(\lambda
x_i)\|=|\lambda|\sum_{i=1}^N \sup_{B\in\Sigma}\|C_iE(B_i\cap
B)x_i\|.
\end{equation*}
Since this holds for every representation of $f$, it follows that
$\left\|\lambda f\right\|_\omega\le |\lambda| \left\|f\right\|$. In the
same way, we have $\left\|f\right\|_\omega=\left\|\lambda^{-1}\lambda
f\right\|_\omega\le |\lambda|^{-1}\left\|\lambda f\right\|_\omega$, giving
$|\lambda|\left\|f\right\|_\omega\le \left\|\lambda f\right\|_\omega$.
Therefore $\left\|\lambda f\right\|_\omega=|\lambda|\left\|f\right\|_\omega$.

Now, to prove that $\|\cdot\|_\omega$ satisfies the triangle
inequality. Let $f,g\in M_{E}$ and let $\epsilon>0$. It
follows from the definition that we may choose representations
$f=\sum_{i=1}^N C_iE_{B_i,x_i}$ and $g=\sum_{j=1}^M
\widetilde{C}_jE_{y_j,A_j}$ such that
\begin{equation*}
\sum_{i=1}^N \sup_{B\in\Sigma}\left\|C_iE(B_i\cap
B)x_i\right\|\le \|f\|_\omega+\epsilon/2,
\end{equation*}
\begin{equation*}
\sum_{j=1}^M \sup_{A\in\Sigma}\|\widetilde{C}_jE(A_j\cap
A)x_j\|\le \|g\|_\omega+\epsilon/2.
\end{equation*}
Then $\sum_{i=1}^N C_iE_{B_i,x_i}+\sum_{j=1}^M\widetilde{C}_j
E_{A_i, y_i}$ is a representation of $f+g$ and so
\begin{equation*}
\|f+g\|_\omega\le \sum_{i=1}^N
\sup_{B\in\Sigma}\|C_iE(B_i\cap B)x_i\|+\sum_{j=1}^M
\sup_{A\in\Sigma}\|\widetilde{C}_jE(A_j\cap A)x_j\|\le
\|f\|_\omega+\|g\|_\omega+\epsilon.
\end{equation*}
Since this holds for every $\epsilon>0$, we have
$\|f+g\|_\omega\le\|f\|_\omega+\|g\|_\omega.$

Finally, we must show that
$\left\|E_{ B, x }\right\|_\omega=\sup_{B'\in\Sigma}\|E(B\cap
B')x\|$. On the one hand, it is clear that
$\left\|E_{ B, x }\right\|_\omega\le \sup_{B'\in\Sigma}\|E(B\cap
B')x\|$. On the other hand, if
$E_{ B, x }=\sum_{i=1}^NC_iE_{B_i,x_i}$ is a representation
of $E_{ B, x }$, 
we have
\begin{eqnarray*}
\sup_{B'\in\Sigma}\|E(B\cap
B')x\|&=&\sup_{B'\in\Sigma}\Big\|\sum_{i=1}^NC_iE(B_i\cap B')x_i\Big\|\\
&\le&\sup_{B'\in\Sigma}\sum_{i=1}^N\|C_iE(B_i\cap B')x_i\|\\
&\le&\sum_{i=1}^N\sup_{B'\in\Sigma}\|C_iE(B_i\cap B')x_i\|.
\end{eqnarray*}
Since this holds for every representation of $E_{ B, x }$, it
follows that $\sup_{B'\in\Sigma}\|E(B\cap
B')x\|\le\|E_{ B, x }\|_\omega$. Therefore
$\|E_{ B, x }\|_\omega=\sup_{B'\in\Sigma}\|E(B\cap B')x\|$.
\end{proof}

We define an equivalence relation $R_{\omega}$ on $M_{E}$ by $f \sim g$ if $\|f-g\|_\omega=0.$ Then $\|\cdot\|_\omega$ is a norm on $M_{E}.$
Denote by $M_{E,\omega}$ the space of the $R_{\omega}$-equivalence classes of $M_{E}$ endowed with the norm $\|\cdot\|_\omega$, and by
$\widetilde{M}_{E,\omega}$ for its completion.

\begin{theorem}\label{th:437}
$\|\cdot\|_\omega$ is a dilation norm of $E$.
\end{theorem}
\begin{proof}
Let $N>0,$ $\{C_i\}_{i=1}^{N}\subset\mathbb{C},$
$\{x_i\}_{i=1}^{N}\subset X$ and $\{B_i\}_{i=1}^{N}\subset \Sigma.$
First, we show that the map $S_\omega:
\widetilde{M}_{E,\omega}\to Y$ defined on
$M_{E,\omega}$ by
$$S_\omega\left(\sum^{N}_{i=1}C_iE_{ {B_i}, {x_i}}\right)=\sum_{i=1}^NE(B_i) x_i,$$
 is well-defined and $\|S_\omega\|\leq 1.$ If
$f=\sum^{N}_{i=1}C_iE_{ {B_i}, {x_i}}$ is a representation of
$f,$ then
\begin{equation*}
\|S_{\omega}(f)\|=\left\|\sum_{i=1}^N C_iE(B_i)
x_i\right\|\le \sum_{i=1}^N \sup_{B\in\Sigma}\|C_iE(B_i\cap
B)x_i\|.
\end{equation*}
Since this holds for every representation of $f$, it follows that
$\|S_{\omega}(f)\|\le \|f\|_\omega$. Therefore $S_{\omega}$
is well-defined and bounded with $\|S_{\omega}\|\le 1.$

Now, to prove that the map $T_{\omega}:X\to
\widetilde{M}_{E,\omega}$ with $T(x)=E_{\Omega,x}$
is bounded with $\|T_{\omega}\|\le\|E\|$. It follows from
the definition of $\omega$ that
\begin{equation*}
\|T_\omega x\|_\omega=\|E_{\Omega,x}\|_\omega=\sup_{B\in\Sigma}\|E(B)
x\|\le \|E\|\cdot\|x\|.
\end{equation*}
Thus, $\|T_{\omega}\|\le\|E\|$.

Finally we need to prove that the map $F_{\omega}: \Sigma\to B(\widetilde{M}_{E,\omega})$ defined by
\begin{eqnarray*}
F_{\omega}(B)\left(\sum_{i=1}^NC_iE_{ {B_i}, {x_i}}\right)
=\sum_{i=1}^NC_iE_{{B\cap B_i}, {x_i}}
\end{eqnarray*}
is an operator-valued measure. By Lemma \ref{le:421}, we only need
to show that $F_{\omega}$ is strongly countably additive and uniform
bounded on $M_{E,\omega}.$ The proof is very similar to that in
Theorem \ref{th:426}. We include it here for completeness.

If $f=\sum_{i=1}^N C_iE_{{B_i},{ x_i}}$ is a representation of
$f$, then
\begin{eqnarray*}
\left\|F_{\omega}(B)\left(f\right)\right\|_\omega &=&
\left\|F_{\omega}(B)\left(\sum_{i=1}^NC_iE_{ {B_i}, {x_i}}\right)\right\|_\omega
\\
&=& \left\|\sum_{i=1}^NC_iE_{{B_i\cap B}, {x_i}}\right\|_\omega
\\
&\le&\sum_{i=1}^N \sup_{B'\in\Sigma}\left \|C_i E(B_i\cap
B\cap B')
x_i\right\|\\
&=&\sum_{i=1}^N\sup_{B'\in\Sigma} \| C_iE(B_i\cap B') x_i\|.
\end{eqnarray*}

Since this holds for every representation of $f$, it follows that
$\left\|F
_{\omega}(E)f\right\|_\omega\le \|f\|_\omega$, which
implies that $\|F_{\omega}(B)\|\le1$.

For the strong  countably additivity of $F_\omega$,  let $\{A_j\}_{j=1}^\infty$  be a disjoint countable collection of members of
$\Sigma$ with union $A$. Then
\begin{eqnarray*}
&&\left\|\sum_{j=1}^MF_{\omega}(A_j)\left(\sum_{i=1}^NC_iE_{{B_i},{ x_i}}\right)-F_\omega(A)\left(\sum_{i=1}^NC_iE_{{B_i},{ x_i}}\right)\right\|_\omega\\
&=&\left\|\sum_{j=1}^M\left(\sum_{i=1}^NC_iE_{{A_j\cap
B_i}, {x_i}}\right)-\left(\sum_{i=1}^NC_iE_{{A\cap B_i}, {x_i}}\right)\right\|_\omega\\
&=&\left\|\sum_{i=1}^NC_i\left(\sum_{j=1}^ME_{{A_j\cap
B_i}, {x_i} }-E_{{A\cap B_i}, {x_i}}\right)\right\|_\omega\\
&=&\sup_{B'\in\Sigma}\left\|\sum_{i=1}^NC_i\left(\sum_{j=1}^ME(A_j\cap
B_i\cap B')x_i-E(A\cap
B_i\cap B')x_i\right)\right\|_Y\\
&=&\sup_{B'\in\Sigma}\left\|\sum_{i=1}^NC_iE\left(\bigcup_{j=M+1}^\infty
\left(A_j\cap B_i\cap B'\right)\right)x_i\right\|_Y\\
&\le&\sum_{i=1}^N|C_i|\sup_{B'\in\Sigma}\left\|E\left(\bigcup_{j=M+1}^\infty
\left(A_j\cap B_i\cap B'\right)\right)x_i\right\|_Y\\
&=&\sum_{i=1}^N|C_i|\sup_{B'\in\Sigma}\left\|E\left(\bigcup_{j=M+1}^\infty
\left(A_j\cap B'\right)\right)x_i\right\|_Y\\
&=&\sum_{i=1}^N|C_i|\sup_{B'\in\Sigma}\left\|\sum_{j=M+1}^\infty E(A_j\cap
B')x_i\right\|_Y.
\end{eqnarray*}
If $\sup_{B'\in\Sigma}\left\|\sum_{j=M+1}^\infty E(A_j\cap
B')x_i\right\|_Y$ does not tend to 0 when $M$ tends to $\infty,$
then we can find $\delta>0$, a sequence of $n_1\le m_1<n_2\le
m_2<n_3\le m_3<.\,.\,.$ , and $\{B_l'\}_{l=1}^\infty\subset
\Sigma$ such that
\begin{eqnarray*}
\left\|\sum_{j=n_l}^{m_l} E(A_j\cap B_l')x_i\right\|\ge
\delta, \ \ \forall\, l\in\N.
\end{eqnarray*}
Obviously, for $l\in\N$ and $n_l\le j\le m_l$, $A_j\cap B_l'$ are
disjoint from each other, so
\begin{eqnarray*}
E\left(\bigcup_{l=1}^\infty\bigcup_{j=n_l}^{m_l} A_j\cap
B_l'\right)x_i=\sum_{l=1}^\infty\sum_{j=n_l}^{m_l} E(A_j\cap
B_l')x_i,
\end{eqnarray*}
which implies
 $\left\|\sum_{j=n_l}^{m_l} E(A_j\cap
B_l')x_i\right\|\rightarrow 0.$ It is a contradiction, hence we have
\[F_\omega(A)\left(\sum_{i=1}^NC_iE_{{B_i},{ x_i}}\right)=\sum_{j=1}^\infty F_{\omega}(A_j)\left(\sum_{i=1}^NC_iE_{{B_i},{ x_i}}\right).\] Thus $F_\omega$ is strongly
countably additive.
\end{proof}
\bigskip
In what follows we will refer $\omega$ as the \emph{maximal
dilation norm} of $E.$ From Proposition \ref{pr:436} we have that
both the minimal and the maximal dilation norms agree on the
elementary vectors $E_{x, B}$, i.e., $||E_{x, B}||_{\alpha} =
||E_{x, B}||_{\omega}$.

Finally we point out that for some special operator-valued measure
systems, there are some other natural ways to construct new
dilations norms. We mention two of them  for which we will only
give the definition but will skip the proofs.

Let $(\Omega,\Sigma,E,B(X,Y))$ be an operator-valued measure
system. The \emph{strong variation of $E$} is the extended
nonnegative function $|E|_\mathrm{SOT}$ whose value on a set
$B\in\Sigma$ is given by
\begin{eqnarray*}
|E|_\mathrm{SOT}(B)&=&\sup\left\{\sum \|E(B_i)x\|:x\in
B_X, \mbox{$B_i$'s are
a partition of $B$}\right\}\\
&=&\sup_{x\in X} |E_x|(B).
\end{eqnarray*}
If $|E|_\mathrm{SOT}(\Omega)<\infty$, then $E$ is called an \emph{operator-valued measure of strongly bounded variation}. Similarly,
the \emph{weak variation of $E$} is the extended nonnegative function $|E|_{\mathrm{WOT}}$ whose value on a set $B\in\Sigma$ is
given by
\begin{eqnarray*}
|E|_{\mathrm{WOT}}(B)&=&\sup\left\{\sum |x^*(E(B_i)x)|:x\in X,
y^*\in Y^*, \mbox{$B_i$'s are a partition of $B$}\right\}\\
&=&\sup_{x\in X, y^*\in Y^*} |E_{x,y^*}|(B).
\end{eqnarray*}
If $|E|_{\mathrm{WOT}}(\Omega)<\infty$, then $E$ is called an \emph{operator-valued measure of weakly bounded variation}.

For an operator-valued measure of strongly bounded variation
(respectively, of weakly bounded variation) , we have a natural
approach to construct a dilation norm on $M_{E}$. Now let $f =
\sum_{i=1}^N C_iE_{{B_i},{ x_i}}$ be any element of $M_{E}$.
Define $\widetilde{\omega}(f)$ and $\widetilde{\mathcal{W}}(f)$ by
\begin{eqnarray*}
\widetilde{\omega}(f)&=&\sup\Big\{\sum_{j=1}^M
\Big\|\sum_{i=1}^NC_iE(B_i\cap A_j)x_i\Big\|: \mbox{$A_j$'s
are a
partition of $\Omega$}\Big\}\\
&=&|E_f|(\Omega).
\end{eqnarray*}
 and
\begin{eqnarray*}
\widetilde{\mathcal{W}}(f) &=&\sup\left\{\sum_{j=1}^M \left\|
y^*\left(\sum_{i=1}^NC_iE(B_i\cap A_j)x_i\right)\right\|:
y^*\in Y^*,
\mbox{$A_j$'s are a partition of $\Omega$}\right\}\\
&=&\sup_{y^*\in
Y^*}|y^*E_f|(\Omega)\\
&=&\|E_f\|(\Omega).
\end{eqnarray*}

\noindent Then it can be proved that $\widetilde{\omega}$  (resp.
$\widetilde{\mathcal{W}}$) is a dilation norm of $E$.

\chapter{Framings and Dilations}

We examine the dilation theory for discrete and continuous framing-induced operator valued
measures.
We also provide a new self-contained proof of Naimark's Dilation Theorem based on our methods of chapter 2,
because we feel that this helps clarify our approach, and also for independent interest.

\section{Hilbertian Dilations}

In the Hilbert space theory, the term ``projection" is usually reserved for ``orthogonal" (i.e. self-adjoint) projection. So in this Chapter we
will resume this tradition. The term ``idempotent" will be used to denote a not necessarily self-adjoint operator which is equal to its square.
Throughout this chapter, $\cH$ denotes the Hilbert space and $I_{\cH}$ is the identity operator on $\cH.$ Let
 $2^{\N}$ denote the family of all subsets of $\N.$

\begin{definition}\label{de:H51}
Given an operator-valued measure $E:\Sigma\to B(\cH)$. Then $E$ is called:
\begin{enumerate}
\item[(i)]an operator-valued probability measure if
$E(\Omega)=\mathrm{I}_\cH,$

\item[(ii)]a (orthogonal) projection-valued measure if
$E(B)$ is a (orthogonal) projection on $\cH$ for all
$B\in\Sigma$; an idempotent-valued measure if $E$ is a
spectral in the sense of Definition \ref{de:7},

\item[(iv)]a positive operator-valued measure if $E(B)$ is a
positive operator on $\cH$ for all $B\in\Sigma$.
\end{enumerate}
\end{definition}

Let $(\Omega,\Sigma,E,B(\cH))$ be an operator-valued measure
system. The definition of dilation space is the same as in section
2.2.  Keep in mind that not every operator-valued measure on a
Hilbert space admits a dilation space which is a Hilbert space.

\begin{definition}\label{de:H31}
Let $(\Omega,\Sigma,E,B(\cH))$ be an operator-valued measure system. We say that $E$ has a Hilbertian dilation if it has a dilation
space which is a Hilbert space. That is, $E$ has a Hilbert dilation space if there exist a Hilbert space $\cK,$ two bounded linear
operators $S:\cK \to \cH$ and $T:\cH\to \cK$, an idempotent-valued measure $F:\Sigma\to B(\cK)$ such that
\begin{equation*}
E(B)=SF(B)T,\qquad \forall B\in\Sigma.
\end{equation*}
\end{definition}

\begin{theorem}\label{pr:Hp1} Let
$(\Omega,\Sigma,E,B(\cH))$ be an operator-valued probability
measure system. If $E$ has a Hilbert dilation space, then
there exist a corresponding Hilbert dilation system
$(\Omega,\Sigma,F,B(\cK), V^*, V)$ such that $V:\cH\to\cK$ is an
isometric embedding.
\end{theorem}
\begin{proof} From Proposition \ref{pr:414}, we know that if
an operator-valued measure can be dilated to a spectral
operator-valued measure, then it can be dilated to a spectral
operator-valued probability measure. So without losing the
generality, we can  assume that the corresponding dilation
projection-valued measure space system
$(\Omega,\Sigma,F,B(\cK),S,T)$ is a probability measure space
system, where $\cK$ is the Hilbert dilation space.

Since $E$ and $F$ are both operator-valued probability
measures, we have that $S$ is a surjection, $T$ is an isomorphic
embedding and $\mathrm{I}_\cH=ST.$ Let $P=TS:\cK\to T\cH$. Then $P$
is a projection from $\cK$ onto $T\cH.$  Hence
\[\cK=P\cK\oplus_\cK(\mathrm{I}_\cK-P)\cK=T\cH\oplus_\cK(\mathrm{I}_\cK-P)\cK\]
From
\[S(\mathrm{I}_\cK-P)=S-STS=S-S=0,\] we get that
\[(\mathrm{I}_\cK-P)\cK\subset ker S.\]
On the other hand, for $z\in\cK,$ if $Sz=0,$ then
\[(\mathrm{I}_\cK-P)z=z-Pz=z-TSz=z\in(\mathrm{I}_\cK-P)\cK\]
Thus
\[\cK=T\cH\oplus_\cK ker S.\]
Define
\[\widetilde{\cH}=\cH\oplus_2ker S.\]
It is easy to see that the operator
$$U:\cK=T\cH\oplus_\cK ker S\to\widetilde{\cH}=\cH\oplus_2ker S$$
defined by $$U=T^{-1}|_{T\cH}\oplus \mathrm{I}_{ker S},$$ is an isomorphic operator. Let $V=UT$ be an operator form $\cH\to\widetilde{\cH}.$ For
any $x\in\cH,$ we have
\[\|Vx\|_{\widetilde{\cH}}=\|UTx\|_{\widetilde{\cH}}=\|x\|_\cH.\]
Thus $V$ is an isometric embedding.

Define
$$\tilde{F}:\Sigma\to B(\widetilde{\cH}),\quad \tilde{F}(B)=UF(B)U^{-1}.$$
Then $\tilde{F}$ is a spectral operator-valued probability
measure. Now we show that
\[E(B)=V^*\tilde{F}(B)V,\qquad \forall B\in\Sigma.\]
Since
\[V^*\tilde{F}(B)V=V^*UF(E)U^{-1}V=(UT)^*UF(B)U^{-1}UT=(UT)^*UF(B)T,\]
we only need to prove that $(UT)^*U=S,$ as claimed.

 For any $x\in \cH$ and
$z=Tx_1+z_2\in \cK,$ where $z_2\in ker S$ and $x_1\in \cH,$
\begin{eqnarray*}
\langle(UT)^*U z, x\rangle_\cH=\langle U z, UT
x\rangle_{\widetilde{\cH}}=\langle UTx_1+U z_2,
x\rangle_{\widetilde{\cH}}=\langle x_1, x\rangle_{\cH}
\end{eqnarray*}
and
\begin{eqnarray*}
\langle S z, x\rangle_\cH=\langle STx_1+S z_2,
x\rangle_{\cH}=\langle x_1, x\rangle_{\cH}
\end{eqnarray*}
Hence $(UT)^*U=S.$
\end{proof}

The classical Naimark's Dilation Theorem tells us that if $E$ is a positive operator-valued measure, then $E$ has a Hilbert dilation
space. Moreover, the corresponding spectral operator-valued measure $F$ is an orthogonal projection-valued measure. Although this is a well
known result in the dilation theory, for self completeness here we include a new proof which uses similar line of ideas as used in the proof of
the Banach space dilation theory.

\begin{theorem}[\textbf{Naimark's Dilation Theorem}]\label{th:H39}
Let $E:\Sigma\to B(\cH)$ be a positive operator-valued
measure. Then there exist a Hilbert space $\cK$, a bounded linear
operator $V:\cH\to \cK$, and an orthogonal projection-valued measure
$F:\Sigma\to B(\cK)$ such that
$$E(B)=V^* F(B) V.$$
\end{theorem}

\begin{proof}
Let $M_E=\mbox{span}\{E_{B,x}: x\in \cH, B\in \Sigma\}$ be the space induced by $(\Omega,\Sigma,E,B(\cH)).$  Now we define a
sesquilinear functional $\langle\,,\rangle$ on this space by setting
$$\left\langle\sum_{i=1}^NE_{{B_i},{x_i}}, \sum_{j=1}^ME_{{A_i},{y_i}}\right\rangle
=\sum_{i=1}^N\sum_{j=1}^M\langle E(B_i\cap A_j)x_i,y_j\rangle_\cH.$$  Actually, if $\sum_{i=1}^NE_{{B_i},{x_i}}=0$, then for $1\le
j\le M$, $\sum_{i=1}^NE(B_i\cap A_j)x_i=0$. Thus this sesquilinear functional is well-defined. Since for any
$f=\sum_{i=1}^NE_{{B_i},{x_i}}$ (without losing the generality, we can assume that $B_i$'s are disjoint from each other), we have that
\begin{eqnarray*}
\left\langle\sum_{i=1}^NE_{{B_i},{x_i}},
\sum_{j=1}^NE_{{B_j},{x_j}}\right\rangle
=\sum_{i=1}^N\sum_{j=1}^N\langle E(B_i\cap
B_j)x_i,x_j\rangle_\cH =\sum_{i=1}^N
\langle E(B_i)x_i,x_i\rangle_\cH\geq 0.
\end{eqnarray*}
Thus it follows  that $\langle\,,\rangle$ is positive definite.

Obviously, this positive definite sesquilinear functional satisfies the Cauchy-Schwarz inequality,
$$|\langle f, g \rangle|^2\le \langle f, f\rangle\cdot \langle g,g\rangle.$$
Hence the sesquilinear functional $\langle\,,\rangle$ on $M_E$ is an inner product. Let $\widetilde{M}_E$ denote the Hilbert space
that is the completion of the inner product space $M_E$, and the induced norm is denoted by $\|\cdot\|_{\widetilde{M}_E}$.

 For every $B\in\Sigma$, define a linear map
 $F(B):\widetilde{M}_E\to\widetilde{M}_E$ by
 $$F(B)\left(\sum_{i=1}^NE_{{B_i},{x_i}}\right)=\sum_{i=1}^N
  E_{{x_i},{B\cap B_i}}.$$
It is easy to see that $F(B)$ is a projection. Take
 $\sum_{i=1}^N E_{{B_i},{x_i}}\in\widetilde{M}_E$, then
 \begin{eqnarray*}
 \left\|F(B)\left(\sum_{i=1}^N
 E_{{B_i},{x_i}}\right)\right\|_{\widetilde{M}_E}
 &=&\sum_{i=1}^N \langle E(B\cap B_i)x_i,x_i\rangle_\cH\\
 &\le& \sum_{i=1}^N
 \langle E(B_i)x_i,x_i\rangle_\cH\\
 &=&\left\|\sum_{i=1}^NE_{{B_i},{x_i}}\right\|_{\widetilde{M}_E}.
 \end{eqnarray*}
 So $\|F(B)\|=1$ or $\|F(B)\|=0$.
 Moreover, we have
 \begin{eqnarray*}
 \left\langle F(B)\left(\sum_{i=1}^N E_{{B_i},{x_i}}\right),
 \sum_{j=1}^M E_{{A_j},{y_j}}\right\rangle_{\widetilde{M}_E}
 &=&\sum_{i=1}^N\sum_{j=1}^M
 \langle E(B\cap B_i\cap A_j)x_i,y_j\rangle_\cH\\
 &=&\left\langle\sum_{i=1}^N E_{{B_i},{x_i}},
 F(B)\left(\sum_{j=1}^M E_{{A_j},{y_j}}\right)\right\rangle_{\widetilde{M}_E}.
 \end{eqnarray*}
 Thus $F(E)$ is a self-adjoint orthogonal projection.

 Now we define
 $$V: \cH\to {\widetilde{M}_E},\qquad \qquad V(x)=E_{x,\Omega}.$$
Since
 $$\|Vx\|^2_{\widetilde{M}_E}
 =\langle E(\Omega)x,x\rangle_\cH\le \|E(\Omega)\|\cdot\|x\|_\cH^2,$$
 we know that $V$ is bounded. Indeed, it is clear that
 $$\|V\|^2=\sup\{\langle E(\Omega)x,x\rangle_\cH:\|x\|_\cH\le1\}=\|E(\Omega)\|.$$
Observe that for any $x,y\in\cH,$
 $$\langle V^*F(B)Vx,y\rangle_\cH=
 \langle F(B)(E_{\Omega,x}),V(E_{\Omega,y})\rangle_{\widetilde{M}_E}
 =\langle E(B)x,y\rangle_\cH,$$
 so $E(B)=V^*F(B) V$, which completes the proof.
 \end{proof}

For general Hilbert space operator-valued measures, Don Hadwin
 completely characterized those that have Hilbertian dilations in
 terms of Hahn decompositions. Let $\Omega$ be a compact Hausdorff
 space and $\Sigma$ be the $\sigma$-algebra of Borel subsets of
 $\Omega$. Recall that a $B(\cH)$-valued measure $E$ is
 called regular if for all $x, y\in \cH$ the complex measure given
 by $<E(\cdot)x, y>$ is regular.

\begin{theorem} \label{th:tHp2}(\cite{HA}) Let $E$ be a
regular, bounded $B(\cH)$-valued measure. Then the following are
equivalent:

(i) $E$ has a Hahn decomposition $E = (E_{1} -
E_{2}) + i( E_{3} -  E_{4})$, where $E_{j}
(j =1, 2, 3, 4)$ are positive operator-valued measures;

(ii) there exist a Hilbert space $\cK,$ two bounded linear
operators $S:\cK \to \cH$ and $T:\cH\to \cK$, an projection-valued
measure $F:\Sigma\to B(\cK)$ such that
\begin{equation*}
E(B)=SF(B)T,\qquad \forall B\in\Sigma.
\end{equation*}
\end{theorem}

\section{Operator-Valued Measures Induced by Discrete Framings}

 Let
$\{x_i\}_{i\in\N}$ be a non-zero frame for a separable Hilbert space
$\cH.$ Then the mapping $E:\N\to B(\cH)$ defined by $$
E(B)=\sum_{i\in B} x_i\otimes x_i, \ \ \forall\, B\in 2^\N,$$
is a positive operator-valued probability measure. By Theorem
\ref{th:H39}, we know that $E$ have a Hilbert dilation space
$\widetilde{M}_E$, the corresponding operator-valued measure
$F$ is self-adjoint and idempotent. Obviously, the rank of
$F(\{i\})$ is 1, so $\{E_{\{i\},x_i}/\|x_i\|^2\}_{i\in\N}$
is an orthonormal basis of the Hilbert space
$\widetilde{M}_E$. For any $B\in 2^\N,$ if $i\in B,$ then
\begin{eqnarray*}
E_{\{i\},x}(B)=E(B\cap\{i\})x=E(\{i\})x =\langle
x,x_i\rangle x_i,
\end{eqnarray*}
\begin{eqnarray*}
\langle x,x_i\rangle\frac{E_{\{i\},x_i}(B)}{\|x_i\|^2}
=\frac{\langle x,x_i\rangle}{\|x_i\|^2}E(B\cap\{i\})x_i
=\frac{\langle x,x_i\rangle}{\|x_i\|^2}\langle x_i,x_i\rangle
x_i=\langle x,x_i\rangle x_i.
\end{eqnarray*}
Thus $E_{\{i\},x}=\langle x,x_i\rangle\frac{E_{\{i\},x_i}}{\|x_i\|^2}.$ Then $V$ is the traditional analysis operator and $V^*$ is
the traditional synthesis operator. This is because
\[Vx=E_{x,\N}
=\sum_{i\in \N}\langle
x,x_i\rangle\frac{E_{\{i\},x_i}}{\|x_i\|^2}\] and
\[V^*\left(\sum_{i\in \N}a_i\frac{E_{\{i\},x_i}}{\|x_i\|^2}
\right)=\sum_{i\in \N}a_i\frac{E(\{i\})x_i}{\|x_i\|^2}
=\sum_{i\in \N}a_ix_i.\]

\begin{definition}\label{de:H41}
Let $\{x_i\}_{i\in\N}$ and $\{y_i\}_{i\in\N}$ are both frames for $\cH.$ We call $\{x_i\}_{i\in\N}$ and $\{y_i\}_{i\in\N}$ are \emph{scale
equivalent} if for each $i\in\N,$ $y_i=\lambda_i x_i$ for some $\lambda_i\in\C$. $\{x_i\}_{i\in\N}$ and $\{y_i\}_{i\in\N}$ are called
\emph{unit-scale equivalent} if we can take $|\lambda_j|=1$ for all $i\in\N.$
\end{definition}

\begin{lemma}\label{le:H42}
Two frames $\{x_i\}_{i\in\N}$ and $\{y_i\}_{i\in\N}$ are unit-scale
equivalent if and only if $\{x_i\}_{i\in\N}$ and $\{y_i\}_{i\in\N}$
generate the same positive operator-valued measure.
\end{lemma}

The following theorem shows exactly when a framing induced
operator-valued measure has a Hilbertian dilation.

\begin{theorem}\label{pr:Hp2}
Let $(x_i,y_i)_{i\in\N}$ be a non-zero framing for a Hilbert space
$\cH.$ Let $E$ be the induced operator-valued probability
measure, i.e.,
\[E(B)=\sum_{i\in B}x_i\otimes y_i, \qquad \forall\,
B\subseteq \N.\] Then $E$ has a Hilbert dilation space $\cK$
if and only if there exist $\alpha_i,\beta_i \in \C, i\in\N$ with
$\alpha_i\bar{\beta_i}=1$ such that $\{\alpha_ix_i\}_{i\in\N}$ and
$\{\beta_iy_i\}_{i\in\N}$ both are the frames for the Hilbert space
$\cH.$
\end{theorem}
\begin{proof}We first prove the sufficient condition. If $E$ has a Hilbert dilation space $\cK,$ then by Theorem \ref{th:422}, there is an elementary dilation
operator-valued measure system $(\Omega,\Sigma,F_\cD,B(\widetilde{M}_{E,\cD}),S_\cD,T_\cD)$ of $E.$ Since the norm on
$M_{E,\cD}$ is defined by
\[\left\|\sum_{i}E_{{B_i},{x_i}}\right\|_\cD=\left\|\sum_{i}F(B_i)T(x_i)\right\|_\cK,\]
$\widetilde{M}_{E,\cD}$ is a Hilbert space.

Recall that the analysis operator
$$T_\cD: \cH\to \widetilde{M}_{E,\cD},\qquad
T_\cD(x)=E_{x,\Omega}$$ and the synthesis operator
$$S_\cD:
\widetilde{M}_{E,\cD}\to \cH,\qquad
S_\cD\left(\sum_{i}E_{{B_i},{x_i}}\right)=\sum_{i}E(B_i)x_i
$$
are both linear and bounded. It is easy to prove that
$F_\cD(\{i\})$ is rank $1$ for all $i\in\N.$ By Lemma
\ref{le:312}, there is a Riesz basis $\phi_i$ on
$\widetilde{M}_{E,\cD}$ such that
\[F_\cD(\{i\})=\phi_i\otimes \phi^{\ast}_i.\]
Hence
\[x_i\otimes y_i=S_\cD\left(\phi_i\otimes \phi^{\ast}_i\right)T_\cD
=(S_\cD\phi_i)\otimes(T_\cD^{\ast}\phi^{\ast}_i).\]
For any $x\in\cH,$ we have
\[\sum_{i\in \N}\left|\langle x,S_\cD\phi_i\rangle\right|^{2}
=\sum_{i\in \N}\left|\langle S_\cD^{\ast}
x,z_i\rangle\right|^{2},\quad \sum_{i\in \N}\left|\langle
x,T_\cD^{\ast}\phi^{\ast}_i\rangle\right|^{2} =\sum_{i\in
\N}\left|\langle T_\cD x,\phi^{\ast}_i\rangle\right|^{2}.\]
 Thus we know that $\{S_\cD\phi_i\}_{i\in\N}$ and
 $\{T_\cD^{\ast}\phi^{\ast}_i\}_{i\in\N}$ both are frames for $\cH$ and
\[x_i\otimes y_i=(S_\cD\phi_i)\otimes(T_\cD^{\ast}\phi^{\ast}_i).\]
So there is only a scalar adjustment $S_\cD\phi_i=\alpha_ix_i$ and
$T_\cD^{\ast}\phi^{\ast}_i=\beta_iy_i$ with
$\alpha_i\bar{\beta_i}=1.$

For the necessary part, assume that there exist $\alpha_i,\beta_i \in \C, i\in\N$ with $\alpha_i\bar{\beta_i}=1$ such that
$\{\alpha_ix_i\}_{i\in\N}$ and $\{\beta_iy_i\}_{i\in\N}$ both are the frames for the Hilbert space $\cH.$ By Proposition 1.6 in \cite{HL}, we
have a Hilbert space $\cK$, a Riesz basis $\{z_i\}_{i\in \N}$ and a projection $P:\cK\to\cH$ such that $Pz_i=\alpha_ix_i.$ Let
$\{z^*_i\}_{i\in\N}$ be the dual Riesz basis, then $\langle z_i,z^*_j\rangle=\delta_{i,j}.$ Define $S=P$
 \[T:\cH\to\cK \qquad T^*z^*_i=\beta_iy_i\]
 and
 \[F:2^{\N}\to B(\cK) \qquad F(B)=z_i\otimes z^*_i\quad \forall B\in 2^{\N}.\]
Then  $T$ is a linear and bounded operator and $F$ is a spectral
operator-valued measure. For any $x\in\cH$ and $B\in 2^{\N},$
 \begin{eqnarray*}
 SF(E)Tx&=&S\sum_{i\in B}\langle Tx,z^*_i\rangle z_i
 =\sum_{i\in B}\langle x,T^*z^*_i\rangle Pz_i\\
 &=&\sum_{i\in B}\langle x,\beta_iy_i\rangle \alpha_ix_i
 =\sum_{i\in B}\langle x,y_i\rangle x_i\\
 &=&E(B)
 \end{eqnarray*}
Thus $\cK$ is the Hilbert dilation space of $E.$
\end{proof}

Since a completely bounded map is a linear combination of positive maps, the following Corollary immediately from Theorem \ref{th:tHp2} and
Theorem \ref{pr:Hp2}.
\begin{corollary}\label{co:cH222}
Let $\{x_i,y_i\}_{i\in\N}$ be a non-zero framing for a Hilbert space $\cH,$ and $E$ be the induced operator-valued probability measure.
Then $E$ is a completely bounded map if and only if $\{x_i,y_i\}_{i\in\N}$ can be re-scaled to dual frames.
\end{corollary}

By using this theorem, we can prove that the first example of
framing in Chapter 5 can not be scaled to a dual frame since the
induced operator-valued measure doesn't have a Hilbert dilation
space,


\begin{lemma}\label{le:H222}\cite[Lemma 5.6.2]{Ch}
Assume that $\{x_i\}_{i\in\N}$ and $\{y_i\}_{i\in\N}$ are Bessel
sequences in $\cH.$ Then the following are equivalent:
\begin{enumerate}
\item[(i)] $x=\sum_{i\in \N} \langle x,y_i\rangle x_i,\quad
\forall x\in\cH.$ \item[(ii)]$x=\sum_{i\in \N} \langle x,x_i\rangle
y_i ,\quad \forall x\in\cH.$ \item[(iii)] $\langle
x,y\rangle=\sum_{i\in \N} \langle x,x_i\rangle \langle
y_i,x_i\rangle ,\quad \forall x,y\in\cH.$
\end{enumerate}
In case the equivalent conditions are satisfied, $\{x_i\}_{i\in\N}$
and $\{y_i\}_{i\in\N}$ are dual frames for $\cH.$
\end{lemma}

\begin{lemma}\label{le:H223}\cite[Lemma 14.9]{Si}
Let $\{x_i\}_{i\in\N}$ be a system of vectors in $\cH.$ If
$\sum_{i\in\N}x_i$ converges unconditionally in $\cH,$ then
\begin{eqnarray*}
\sum_{i\in \N} \|x_i\|^2\leq C<+\infty.
\end{eqnarray*}
\end{lemma}

The following lemma can be proved by using the Orlicz-Pettis Theorem.
\begin{lemma}\label{le:H224}
A pair $\{x_i,y_i\}_{i\in\N}$ is a non-zero framing for a Hilbert space $\cH$ iff $\{y_i,x_i\}_{i\in\N}$ is a non-zero framing for $\cH.$
\end{lemma}

The following provides us a sufficient condition under which a
framing induced operator-valued measure has a Hilbertian dilation.

\begin{theorem}\label{pr:Hp3}
Let $\{x_i,y_i\}_{i\in\N}$ be a non-zero framing for a Hilbert space $\cH.$ If $\ \inf\|x_i\|\cdot\|y_i\|>0,$ then we can find $\alpha_i,\beta_i
\in \C, i\in\N$ with $\alpha_i\bar{\beta_i}=1$ such that $\{\alpha_ix_i\}_{i\in\N}$ and $\{\beta_iy_i\}_{i\in\N}$ both are frames for the
Hilbert space $\cH$. Hence the operator-valued measure induced by $\{x_i,y_i\}_{i\in\N}$ has a Hilbertian dilation.
\end{theorem}
\begin{proof}
Since $\{x_i,y_i\}_{i\in\N}$ is a non-zero framing for $\cH,$ we have by Lemma \ref {le:H224} that
\[x=\sum_{i\in \N} \langle x,y_i\rangle x_i=
\sum_{i\in \N} \langle x,x_i\rangle y_i,\] and the series converges unconditionally for all $x\in \cH.$ Applying Lemma \ref{le:H223} to the
sequences $\{\langle x,y_i\rangle x_i\}_{i\in \N}$ and $\{\langle x,x_i\rangle y_i \}_{i\in \N}$, we get
\[\sum_{i\in \N} \left|\langle x,y_i\rangle\right|^2\|x_i\|^2<+\infty,
\quad \mbox{and}\quad \sum_{i\in \N} \left|\langle x,x_i\rangle\right|^2\|y_i\|^2<+\infty.
 \]
Let
$$\alpha_i=(\|y_i\|/\|x_i\|)^{1/2}\quad\mbox{and}\quad
\beta_i=1/\alpha_i=(\|x_i\|/\|y_i\|)^{1/2}.$$ It is easy to see
that
\begin{eqnarray*}
\sum_{i\in \N} \left|\langle
x,\alpha_ix_i\rangle\right|^2&=&\sum_{i\in \N} \frac{\left|\langle
x,x_i\|y_i\|\rangle\right|^2}{\|x_i\|\|y_i\|}\\
&\leq&\frac{1}{\inf\|x_i\|\cdot\|y_i\|}\sum_{i\in \N} \left|\langle
x,x_i\rangle\right|^2\|y_i\|^2<+\infty
\end{eqnarray*}
and
\begin{eqnarray*}
\sum_{i\in \N} \left|\langle
x,\beta_iy_i\rangle\right|^2&=&\sum_{i\in \N} \frac{\left|\langle
x,y_i\|x_i\|\rangle\right|^2}{\|x_i\|\|y_i\|}\\
&\leq&\frac{1}{\inf\|x_i\|\cdot\|y_i\|}\sum_{i\in \N} \left|\langle
x,y_i\rangle\right|^2\|x_i\|^2<+\infty,
\end{eqnarray*}
Hence $\{\alpha_ix_i\}_{i\in\N}$ and $\{\beta_iy_i\}_{i\in\N}$ are
both Bessel sequence. For any $x\in\cH,$ we have
\[x=\sum_{i\in \N} \langle x,\beta_iy_i\rangle \alpha_ix_i=
\sum_{i\in \N} \langle x,\alpha_ix_i\rangle \beta_iy_i.\] By Lemma \ref{le:H223}, we conclude that $\{\alpha_ix_i\}_{i\in\N}$ and
$\{\beta_iy_i\}_{i\in\N}$ both are frames for $\cH.$
\end{proof}

Recall from Dai and Larson \cite{DL} that a \emph{unitary system} $\mathcal {U}$ is a subset of the unitary operators acting on a separable
Hilbert space $\cH$ which contains the identity operator $I.$ We have the following consequence immediately follows from Theorem \ref{pr:Hp3}:

\begin{corollary}\label{co:H38}
Let $\mathscr{U}_1$ and $\mathscr{U}_2$ be unitary systems on a
separable Hilbert space $\cH.$ If there exist $x,y\in\cH$ such that
$\{\mathscr{U}_1x,\mathscr{U}_2y\}$ is a framing of $\cH,$ then
$\{\mathscr{U}_1x\}$ and $\{\mathscr{U}_2y\}$ both are frames for
$\cH.$
\end{corollary}

We remark that there exist sequences $\{x_{n}\}$ and $\{y_{n}\}$
in a Hilbert space $\mathcal H$ with the following properties:

(i) $x = \sum_{n} \langle x, x_{n}\rangle y_{n}$  hold for all $x$
in a dense subset of $\mathcal H$, and the convergence is
unconditionally.

(ii) $inf ||x_{n}||\cdot ||y_{n}|| >0$.

(iii) $\{x_{n},  y_{n}\}$ is not a framing.

Here is a simple example of this type: Let $\mathcal{H} = L^{2}[0,
1]$, and $g(t) = t^{1/4}$, $f(t) = 1/g(t)$. Define $x_{n}(t) =
e^{2\pi i nt}f(t)$ and $y_{n}(t) = e^{2\pi in t}g(t)$. Then it is
easy to verify $(i)$ and $(ii)$. For $(iii)$, we consider the
convergence of the series
$$
\sum_{n\in \Bbb{Z}}\langle f, x_{n}\rangle y_{n}.
$$
Note that $||\langle f, x_{n}\rangle y_{n}||^{2} = |\langle f,
x_{n}\rangle|^{2} \cdot ||g||^{2}$ and $\{\langle f,
x_{n}\rangle\}$ is not in $\ell^{2}$ (since $f^{2} \notin L^{2}[0,
1]$). Thus, from Lemma \ref{le:H223}, we have that
$\sum_{n}\langle f, x_{n}\rangle y_{n}$ can not be convergent
unconditionally. Therefore $\{x_{n},  y_{n}\}$ is not a framing.

\section{Operator-Valued Measures Induced by Continuous
Frames}

We mainly examine the positive-operator valued measure induced by continuous frames investigated by Gabardo and Han in \cite{GH} and also by
Fornasier and Rauhut in \cite{FR}. At the same time, we will introduce a few new ``frames" including operator-valued continuous frames and
operator-valued $\mu$-frames, both generalize the concept of continuous frames.
\begin{definition}\label{de:43}
Let $\cH$ be a separable Hilbert space and $\Omega$ be a
$\sigma$-locally compact ($\sigma$-compact and locally compact)
Hausdorff space endowed with a positive Radon measure $\mu$ with
$\mbox{supp}\mu=\Omega$. A weakly continuous function
$\cF:\Omega\to \cH$ is called a \emph{continuous frame} if there
exist constants $0 < C_1\le C_2 < \infty$ such that
\begin{eqnarray*}
C_1\|x\|^2 \le \int_\Omega |\langle x,\cF (\omega) \rangle|^2 d
\,\mu(\omega) \le C_2\|x \|^2, \quad \forall\, x\in \cH.
\end{eqnarray*}
\end{definition}
If $C_1 = C_2$,  then the frame is called \emph{tight}. Associated to $\cF $ is the frame operator $S_\cF $ defined  by
\begin{eqnarray*} S_\cF  : \cH \to \cH, \quad \langle S_\cF  (x),y\rangle
:= \int_\Omega \langle x, \cF (\omega)\rangle\cdot\langle \cF
(\omega),y\rangle d \,\mu(\omega).
\end{eqnarray*}
Then $S_\cF $ is a bounded, positive, and invertible operator. We define the following transform associated to $\cF $,
\begin{eqnarray*}
V_\cF  : \cH\to L^2(\Omega, \mu), \quad  V_\cF  (x)(\omega) :=
\langle x, \cF (\omega)\rangle.
\end{eqnarray*}
Its adjoint operator is given  by
\begin{eqnarray*}
V_\cF ^* : L^2(\Omega, \mu) \to \cH, \quad \langle V_\cF
^*(f),x\rangle :=\int_{\Omega} f(\omega)\langle\cF
(\omega),x\rangle d\,\mu(\omega).
\end{eqnarray*}

A weakly continuous function $\cF :\Omega\to \cH$ is called
\emph{Bessel} if there exist a positive constant $C$ such that
\begin{eqnarray*}
\int_\Omega |\langle x,\cF (\omega) \rangle|^2 d \,\mu(\omega) \le
C\|x \|^2, \quad \forall\, x\in \cH.
\end{eqnarray*}
Let $\mathcal{B}$ be the Borel algebra of $\Omega$ and $\cF :\Omega\to \cH$ be Bessel. Then the mapping $$E_\cF :\mathcal{B}\to B(\cH),
\quad \langle E_\cF (B)x,y\rangle=\int_B \langle x,\cF (\omega)\rangle\cdot\langle \cF (\omega),y\rangle d\,\mu(\omega)$$ is well-defined
since we have
\begin{eqnarray*}
|\langle E_\cF (B)x,y\rangle|&=&\Big|\int_B \langle x,\cF
(\omega)\rangle\cdot\langle \cF (\omega),y\rangle
d\,\mu(\omega)\Big|\\&\leq&\int_\Omega |\langle x,\cF
(\omega)\rangle|\cdot|\langle \cF (\omega),y\rangle|
d\,\mu(\omega)\\&\leq&\Big(\int_\Omega |\langle x,\cF
(\omega)\rangle|^2
d\,\mu(\omega)\Big)^{1/2}\cdot\Big(\int_\Omega|\langle \cF
(\omega),y\rangle|^2
d\,\mu(\omega)\Big)^{1/2}\\
&\leq& C \cdot\|x\|\cdot\|y\|.
\end{eqnarray*}
and therefore  $\|E_\cF (B)\|\le C$ for all $E\in \mathcal{B}.$
\begin{lemma}\label{le:44}
$E_\cF $  defined above is a positive operator-valued measure.
\end{lemma}
\begin{proof} The positivity follows from that
\begin{eqnarray*}
\langle E_\cF (B)x,x\rangle=\int_\Omega |\langle x,\cF (\omega)\rangle|^2 d\,\mu(\omega)\geq 0
\end{eqnarray*}
holds for all $x$ in $\cH$.

Assume that $\{B_i\}_{i\in\N}$ is a sequence of disjoint Borel sets
of $\Omega$ with union $B$. Then we have
\begin{eqnarray*}
\sum_{i\in\N}\langle E_\cF (B_i)x,y\rangle
&=&\sum_{i\in\N}\int_{B_i} \langle x,\cF (\omega)\rangle\cdot\langle
\cF (\omega),y\rangle d\,\mu(\omega)\\
&=&\int_B \langle x,\cF (\omega)\rangle\cdot\langle \cF
(\omega),y\rangle d\,\mu(\omega)\\&=&\langle E_\cF
(B)x,y\rangle
\end{eqnarray*}
for all $x,y$ in $\cH$. Thus $E_\cF $ weakly countably additive, and hence it is a positive operator-valued measure.
\end{proof}

\begin{lemma}\label{le:H51}
Let $\cF_1$ and $\cF_2$ be two Bessel functions from $\Omega$ to $\cH$. Then $\cF_1$ and $\cF_2$ induce the same positive operator-valued
measure if and only if $\cF_1=u\cdot \cF_2$ where $u$ is a unimodular function on $\Omega$.
\end{lemma}
\begin{proof}
Since $E_{\cF_1}=E_{\cF_2}$, we have
$$\int_E |\langle x,
\cF_1(\omega)\rangle|^2 d\,\mu(\omega)=\langle E_{\cF_1}(B)x,x\rangle=\langle E_{\cF_2}(B)x,x\rangle=\int_B |\langle x,
\cF_2(\omega)\rangle|^2 d\,\mu(\omega)$$ for all $B\in\Sigma$ and $x\in \cH.$ Then by the continuity, we have $|\langle x,
\cF_1(\omega)\rangle|=|\langle x, \cF_2(\omega)\rangle|$ for all $x\in\cH$ and $\omega\in\Omega.$ Thus we have $$|\langle \cF_1(\omega),
\cF_2(\omega)\rangle|=\|\cF_1(\omega)\|^2= \|\cF_2(\omega)\|^2=\|\cF_1(\omega)\|\cdot\|\cF_2(\omega)\|.$$ Then there is a unimodular function on
$\Omega$ such that $\cF_1(\omega)=u(\omega)\cdot \cF_2(\omega).$
\end{proof}

From Theorem \ref{th:H39}, we know that $E_\cF$ has a
Hilbert dilation space $\widetilde{M}_{E_\cF}.$ The
following theorem shows that $\widetilde{M}_{E_\cF}$ is
isometric to $L^2(\Omega,\mu).$
\begin{theorem}\label{th:45}
Suppose that $\mathrm{supp}\cF=\Omega. $ Then the dilation Hilbert space $\widetilde{M}_{E_\cF}$ is isometric to $L^2(\Omega,\mu)$ under
the following natural linear surjective isometry
$$U:\widetilde{M}_{E_\cF}\to L^2(\Omega,\mu), \quad
U({E_\cF}_{B,x})(\omega)=\chi_B(\omega)\cdot \langle x,
\cF(\omega)\rangle$$
\end{theorem}
\begin{proof} Without losing the generality, we can assume that
$\{B_i\}$ is a disjoint countable collection of members of $\Sigma$. Then we have
\begin{eqnarray*}
\left\|\sum_i{E_\cF}_{B_i,x_i}\right\|^2_{\widetilde{M}_{E_\cF}}
&=& \sum_i
\langle E_\cF(B_i)x_i,x_i\rangle\\
&=&\sum_i \int_{B_i} \left|\langle x_i,\cF
(\omega)\rangle\right|^2
d\, \mu(\omega)\\
&=&\int_\Omega \left|\sum_{i} \chi_{B_i}(\omega)\cdot
\langle x_i,\cF (\omega)\rangle\right|^2d\,\mu(\omega)\\
&=&\left\|U\left(\sum_i{E_\cF}_{B_i,x_i}\right)\right\|^2_{
L^2(\Omega,\mu)}.
\end{eqnarray*}
Thus $U$ is an isometry.

Now we prove that $U$ is surjective. Because all the functions $\chi_{B}(\omega) $ are dense in $L^2(\Omega,\mu)$, where $B$ is any set in
$\Omega$. So we only need to approximate $\chi_B(\omega)$ with $0<\mu(B)<\infty$. Since $\mu$ is a Radon measure which is inner regular, without
losing the generality, we can assume that $B$ is compact. Then for any $\omega\in E$, we can find an $x_0\in \cH$ such that $\langle x_0,\cF
(\omega)\rangle\neq 0.$ Let $x_{\omega}=x_0/\langle x_0,\cF (\omega)\rangle.$ Then
\[\langle x_{\omega},\cF(\omega)\rangle=1.\]
Since $\cF$ is a weakly continuous function from $\Omega$ to $\cH,$
there is a neighborhood $\mathcal{U}_{\omega}$ of $\omega$ such that
for all $\upsilon\in \mathcal{U}_{\omega}$,
\[\left|\langle x_{\omega}, \cF(\upsilon) \rangle-1\right|\le
\sqrt{\frac{\epsilon}{\mu(B)}} .\] Then we have $E\subset
\cup_{\omega\in B}\mathcal{U}_{\omega}$. Since $E$ is compact, there
are finitely many $\{\omega_i\}_{1\le i\le N}$ such that $B\subset
\cup_{1\le i\le N}\mathcal{U}_{\omega_i}$, and hence $B=\cup_{1\le
i\le N}(\mathcal{U}_{\omega_i}\cap B)$. We can find a sequence of
disjoint subsets $\widetilde{\mathcal{U}}_{\omega_i}$ of
$\mathcal{U}_{\omega_i}\cap B$ such that $B=\cup_{1\le i\le
N}\widetilde{\mathcal{U}}_{\omega_i}$. Then we have
\begin{eqnarray*}
&&\left\|U\left(\sum_{i=1}^N{E_\cF}_{x_{\omega_i},\widetilde{\mathcal{U}}_{\omega_i}}
\right)-\chi_B(\omega)y\right\|^2_{L^2(\Omega,\mu)}\\
&=&\int_\Omega \left|\sum_{i=1}^N
\chi_{\widetilde{\mathcal{U}}_{\omega_i}}(\omega)\langle
x_{\omega_i},\cF(\omega)\rangle-\chi_B(\omega)\right|^2
d\,\mu(\omega)\\
&=&\sum_{i=1}^N  \int_{\widetilde{\mathcal{U}}_{\omega_i}}\left|
\langle
x_{\omega_i},\cF(\omega)\rangle-1\right|^2d\,\mu(\omega)\\
&\le& \sum_{i=1}^N \mu(\widetilde{\mathcal{U}}_{\omega_i})
    \frac{\epsilon}{\mu(B)}=\epsilon.
\end{eqnarray*}
Therefore $U$ is surjective as claimed.
\end{proof}

Then we can write the orthogonal projection-valued measure $F$ on
the dilation space as follows:

\begin{theorem}\label{th:46} Let $F(B)$ be the orthogonal projection-valued measure in Theorem \ref{th:H39}. Then we have $F(B)=U^* \chi_B\,
U$ for every $B\in\Sigma$.
\end{theorem}
\begin{proof}
From Theorem \ref{th:H39}, we know that
$$F(B)\left(\sum_{i}{E_\cF}_{{B_i},{x_i}}\right)=\sum_{i}
  {E_\cF}_{{B\cap B_i},{x_i}}.$$
Also for any $y\in\cH$ and $A\in \Sigma,$ we have
\begin{eqnarray*}
\left\langle U^* \left(\chi_B(\omega)\langle x, \cF(\omega)\rangle\right), {E_\cF}_{A,y}\right\rangle_{\widetilde{M}_{E_\cF}} & =&
\left\langle \chi_B(\omega)\langle x,\cF(\omega)\rangle, U\left({E_\cF}_{A,y}\right)\right\rangle_{L^2(\Omega,\mu)}
\\
&=&\Big\langle \chi_B(\omega)\langle x,\cF(\omega)\rangle,
\chi_A(\omega)\langle y,\cF(\omega)\rangle\Big\rangle_{L^2(\Omega,\mu)}\\
&=& \int_{B\cap A}\langle x,\cF(\omega)\rangle\cdot\langle
\cF(\omega),y\rangle d\,\mu(\omega),
\end{eqnarray*}
and
\begin{eqnarray*}
\left\langle {E_\cF}_{B,x},
{E_\cF}_{A,y}\right\rangle_{\widetilde{M}_{E_\cF}}
&=&\left\langle E_\cF(B\cap A)x, y\right\rangle_\cH
\\
&=&\int_{B\cap A}\langle x,\cF(\omega)\rangle\cdot\langle
\cF(\omega),y\rangle d\,\mu(\omega).
\end{eqnarray*}
Thus
\[U^* \left(\chi_B(\omega)\langle
x,\cF(\omega)\rangle\right)={E_\cF}_{B,x}.\] Since
\begin{eqnarray*}
U^* \chi_B\,U\left(\sum_{i}{E_\cF}_{{B_i},{x_i}}\right)
&=&\sum_{i}U^*
\chi_B(\omega)\left(U{E_\cF}_{{x_i},{E_i}}\right)
\\
&=&\sum_{i}U^* \chi_B(\omega)\chi_{B_i}(\omega)\langle
x_i,\cF(\omega)\rangle\\
&=&\sum_{i}U^* \chi_{B\cap B_i}(\omega)\langle
x_i,\cF(\omega)\rangle\\
&=&\sum_{i}{E_\cF}_{{B\cap B_i},{x_i}}\\
&=&F(B)\left(\sum_{i}{E_\cF}_{{B_i},{x_i}}\right),
\end{eqnarray*}
we get $F(B)=U^* \chi_B\, U$, as claimed.
\end{proof}

\begin{theorem}\label{th:42}
 Suppose that $\mathrm{supp}\,\cF=\Omega.$ If
\[
 L : =\inf\{\|E_{\cF}(B)\|:\|E_{\cF}(B)\|> 0\}>0,
 \]
then $\Omega$ is at most countable, that is, every point in
$\Omega$ is an open set.
\end{theorem}
\begin{proof}Let
$L=\inf\{\|E_{\cF}(B)\|:\|E_{\cF}(B)\|> 0\}>0.$ First we show that for any open subset $U$ of $\Omega$, we have
$\|E_{\cF}(U)\|>0$. Choose $\omega_U\in U$ and $x_U\in \cH$ such that $\langle x_U,\cF(\omega_U)\rangle\neq0$. Then we have
\begin{eqnarray*}
\langle E_{\cF}(U)x_U, x_U\rangle&=& \int_U
\left|\langle x_U,\cF(\omega)\rangle\right|^2 d\,\mu(\omega)\\
&\ge&\int_{\{\omega\in U: |\langle
x_U,\cF(\omega)\rangle|>|\langle x_U,\cF(\omega_U)\rangle|/2\}}
\left|\langle x_U,\cF(\omega)\rangle\right|^2 d\,\mu(\omega)\\
&\ge& \mu\{\omega\in U: |\langle x_U,\cF(\omega)\rangle|>|\langle
x_U,\cF(\omega_U)\rangle|/2\}\cdot |\langle
x_U,\cF(\omega_U)\rangle|^2/4\\
&>&0.
\end{eqnarray*}
Thus $\|E_{\cF}(U)\|>0,$ and so $\|E_{\cF}(U)\|\ge L$ for any open set $U$.

Fix any $\omega_0 \in\Omega$. Since $\Omega$ is locally compact, we can choose a compact neighborhood $U_{\omega_0}$ of $\omega_0$. Since
$\cF(U_{\omega_0})$ is weakly compact, we have $$M:=\sup_{w\in U_{\omega_0}}\|\cF(\omega)\|<\infty.$$ Thus, for any open subset $U\subset
U_{\omega_0}$, we get
\begin{eqnarray*}
|\langle E_\cF (U)x,y\rangle| &=&\left|\int_U \langle x,\cF
(\omega)\rangle\cdot\langle \cF (\omega),y\rangle
d\,\mu(\omega)\right|\\
&\leq&\int_U |\langle x,\cF (\omega)\rangle|\cdot|\langle \cF
(\omega),y\rangle|
d\,\mu(\omega)\\
&\leq&\left(\int_U |\langle x,\cF (\omega)\rangle|^2
d\,\mu(\omega)\right)^{1/2}\cdot\left(\int_U|\langle \cF
(\omega),y\rangle|^2 d\,\mu(\omega)\right)^{1/2}\\
&\leq&\left(\int_U \| x\|^2\cdot\|\cF (\omega)\|^2
d\,\mu(\omega)\right)^{1/2}\cdot\left(\int_U\|\cF
(\omega)\|^2\cdot\|y\|^2
d\,\mu(\omega)\right)^{1/2}\\
&\leq&\left(\int_U \| x\|^2\cdot M^2
d\,\mu(\omega)\right)^{1/2}\cdot\left(\int_U M^2\cdot\|y\|^2
d\,\mu(\omega)\right)^{1/2}\\
&\leq& \mu(U)\cdot M^2  \cdot\|x\|\cdot\|y\|.
\end{eqnarray*}
This implies that $\|E_{\cF}(U)\|\le M^2\mu(U)$.  By $\|E_{\cF}(U)\|\ge L$, we obtain that $\mu(U) \ge L/M^2>0$ for all open subsets
of $U_{\omega_0}.$ Let $$\mathcal {V}_{\omega_0}=\{V: V \,\mbox{is an open subset of }U_{\omega_0} \, \mbox{containing}\, \omega_0\}$$  and
$$\widetilde{L}:=\inf_{V\in \mathcal{V}_{\omega_0}} \mu(V)\ge L/M^2>0.$$ We can choose $V_{\omega_0}\in \mathcal{V}_{\omega_0}$ such that
$\mu(V_{\omega_0})<\widetilde{L}+L/(2M^2).$

Now, we prove that $V_{\omega_0}=\{\omega_0\}.$ If there is
another $\omega_1\in V_{\omega_0}$ and $\omega_1\ne \omega_0$,
since $\Omega$ is Hausdorff, we can find an open set
$W_{\omega_0}\subset V_{\omega_0}\subset U_{\omega_0}$ containing
$\omega_0$, and an open set $W_{\omega_1}\subset
V_{\omega_0}\subset U_{\omega_0}$ containing $\omega_1$ such that
and $W_{\omega_0}\cap W_{\omega_1}=\emptyset$. Since
$W_{\omega_0}\in \mathcal {V}_{\omega_0}$, we have
$\mu(W_{\omega_0})\geq \widetilde{L}$. Thus,
\begin{eqnarray*}
\widetilde{L}+L/(2M^2)>\mu(V_{\omega_0}) \ge
\mu(W_{\omega_0})+\mu(W_{\omega_1})\ge \widetilde{L}+L/M^2,
\end{eqnarray*}
which is a contradiction. So $\{\omega_0\}=V_{\omega_0}$ is an
open set, hence $\Omega$ is at most countable.
\end{proof}

\begin{corollary}\label{co:H48}
If $\cH$ is separable and $E_{\cF}$ is a projection-valued
measure, then $\Omega$ is countable.
\end{corollary}

We briefly discuss one generalization of a continuous frame.

\begin{definition}\label{de:H417}
 A function $\mathcal{F}:\Omega\to
B(\cH,\cH_0)$ is called an \emph{operator-valued $\mu$-frame} if
it is weakly Bochner measurable and if there exist two constants
$A,B>0$ such that
$$A\|x\|_\cH^2\le \int_\Omega \|\mathcal{F}(\omega)x\|_{\cH_0}^2 \mathrm{d}\mu(\omega)\le B \|x\|^2_\cH$$
holds for all $x\in \cH.$
\end{definition}

Similar to the continuous frame case we have:

\begin{theorem}\label{th:H418}
Then the mapping
\[\varphi
_{\cF}:\Sigma\to B(\cH) \quad
\langle E_{\cF}(B)x,y\rangle_{\cH}=\int_B\langle\cF(\omega)x,
\cF(\omega)y\rangle_{\cH_0}\mathrm{d}\mu(\omega);\] is an
operator-valued measure.
\end{theorem}

Define an operator $\theta_{\mathcal {F}}:\cH\to L^2(\mu;\cH_0)$
by
$$(\theta_\mathcal{F}x)(\omega)=\mathcal{F}(\omega)x,\qquad \forall \ x\in\cH, \omega\in\Omega.$$
It is easy to see that $\theta_{\mathcal {F}}$ is a bounded linear operator. In fact, it is injective and bounded below. For operator-valued
$\mu$-frames the dilation space is very much similar to the regular frames case.

\begin{theorem}\label{th:H419}
$L^2(\mu;\cH_0)$ is a Hilbert dilation space of  $E_{\cF}$.
\end{theorem}
\begin{proof}
Define a mapping $F_{\mathcal {F}}:\Sigma\to B(L^2(\mu;\cH_0))$
 by
\begin{displaymath}
(F_{\cF}(B)f)(\omega)=\left\{\begin{array}{ll}
f(\omega),&\textrm{$\omega\in B$},\\
0,&\textrm{$\omega\neq 0$}.
\end{array}\right.
\end{displaymath}
Then $F_{\mathcal {F}}$ is a self-adjoint projection-valued
measure.

Since
\begin{eqnarray*}
\left\langle\theta^{\ast}_{\mathcal
{F}}F_{\cF}(B)\theta_{\mathcal {F}}x,y\right\rangle_\cH
&=&\left\langle F_{\cF}(B)\theta_{\mathcal
{F}}x,\theta_{\mathcal {F}} y\right\rangle_{L^2(\mu;\cH_0)}
\nonumber\\
&=&\left\langle F_{\cF}(B)\mathcal {F}(\omega)x,\mathcal
{F}(\omega)y\right\rangle_{L^2(\mu;\cH_0)}\nonumber\\
&=& \left\langle \chi_E(\omega)\mathcal {F}(\omega)x,\mathcal
{F}(\omega)y\right\rangle_{L^2(\mu;\cH_0)}\nonumber\\
&=&\int_B \langle
\mathcal{F}(\omega)x,\mathcal{F}(\omega)x\rangle_{\cH_0}d\mu(\omega)\nonumber\\
&=&\langle E_{\cF}(B)x,y\rangle_{\cH}
\end{eqnarray*}
Thus $L^2(\mu;\cH_0)$ is a Hilbert dilation space of
$E_{\cF}$.
\end{proof}

The operator $\theta_{\mathcal {F}}$ is  the usual analysis operator and $\theta^{\ast}_{\mathcal {F}}$ is the usual synthesis operator. The
frame operator on $\cH$ is defined by
$$S_{\mathcal{F}}:\cH\to\cH,\qquad S_{\mathcal{F}}=\theta^{\ast}_{\mathcal {F}}\theta_{\mathcal
{F}}.$$ For any $x,y\in\cH,$ we have
\begin{eqnarray*}
\langle S_{\mathcal {F}}x,y\rangle_{\cH}
&=&\langle \theta_{\mathcal {F}}x,\theta_{\mathcal {F}}y\rangle_{L^2(\mu;\cH_0)}\nonumber\\
&=&\int_\Omega \langle
\mathcal{F}(\omega)x,\mathcal{F}(\omega)y\rangle_{\cH_0}d\mu(\omega)\nonumber\\
&=&\int_\Omega \langle
\mathcal{F}^{\ast}(\omega)\mathcal{F}(\omega)x,y\rangle_{\cH}d\mu(\omega)\nonumber\\
\end{eqnarray*}
Thus
\[
S_{\mathcal{F}}x=\int_\Omega\mathcal{F}^{\ast}(\omega)\mathcal{F}(\omega)xd\mu(\omega).
\]
\[
\theta^{\ast}_{\mathcal{F}}f=\int_\Omega\mathcal{F}^{\ast}(\omega)f(\omega)d\mu(\omega),
\]
where the equation is in weak sense, i.e.
\[\langle\theta^{\ast}_{\mathcal {F}}f,x\rangle_{\cH}
=\int_\Omega \langle \mathcal{F}^{\ast}(\omega)
f(\omega),x\rangle_{\cH}d\mu(\omega),\] and
\[\langle S_{\mathcal {F}}x,y\rangle_{\cH}
=\int_\Omega \langle
\mathcal{F}^{\ast}(\omega)\mathcal{F}(\omega)x,y\rangle_{\cH}d\mu(\omega).
\]

Therefore for each operator-valued $\mu$-frame,  we have four associated  operator-valued measures on $(\Omega,\Sigma)$, namely:
\begin{enumerate}
\item[(i)] The frame operator-valued measure $E_{\mathcal
{F}}:\Sigma\to B(\cH)$ is defined by
\[\langle E_{\mathcal {F}}(B)x,y\rangle_{\cH}=\int_B\langle\mathcal {F}(\omega)x,
\mathcal {F}(\omega)y\rangle_{\cH_0}d\mu(\omega);\]

\item[(ii)] The analysis operator-valued measure $\alpha_{\mathcal
{F}}:\Sigma\to B(\cH,L^2(\mu;\cH_0))$ is defined by
\begin{displaymath}
(\alpha_{\mathcal {F}}(B)x)(\omega)=\left\{\begin{array}{ll}
\cF(\omega)x,&\textrm{$\omega\in B$},\\
0,&\textrm{$\omega\neq 0$};
\end{array}\right.
\end{displaymath}

\item[(iii)] The synthesis operator-valued measure
$\sigma_{\mathcal {F}}:\Sigma\to B(L^2(\mu;\cH_0),\cH)$ is defined
by
\[\langle\sigma_{\cF}(B)f,x\rangle_{\cH}=\int_B\langle f(\omega),\mathcal {F}(\omega)x\rangle_{\cH_0}d\mu(\omega);\]

\item[(iv)] The self-adjoint projection-valued measure
$F_{\mathcal {F}}:\Sigma\to B(L^2(\mu;\cH_0))$ is defined by
\begin{displaymath}
(F_{\cF}(E)f)(\omega)=\left\{\begin{array}{ll}
f(\omega),&\textrm{$\omega\in E$},\\
0,&\textrm{$\omega\neq 0$}.
\end{array}\right.
\end{displaymath}
\end{enumerate}

\chapter{Dilations of Maps}

In this chapter we establish some dilation results for general
linear mappings of algebras, mainly focusing on (not necessarily
cb-maps) on von Neumann algebras, and more generally on Banach
algebras. The ideas for our proofs come indirectly from our
methods in Chapter 2 for OVM's.

We begin with a possibly known purely algebraic result (Proposition
\ref{th:t50}) which shows that dilations of linear maps are always
possible even in the absence of any topological structure. In the
presence of additional hypotheses stronger results are possible:
When a domain algebra, mapping and range space have strong
continuity and/or structural properties we seek similar properties
for the dilation. This plan led to our other results. Theorem
\ref{th:5B1} states that for any Banach algebra $\mathscr{A}$ and
any bounded linear operator $\phi$ from $\mathscr{A}$ to $B(\cH)$ on
a Banach space $\cH$, there exist a Banach space $Z,$ a bounded
linear unital homomorphism $\pi:\mathscr{A}\to B(Z)$, and bounded
linear operators $T:\cH\to Z$ and $S:Z\to\cH$ such that
\[\phi(a)=S\pi(a)T\] for all $a\in\mathscr{A}.$   In the case that the Banach algebra is an abelian purely atomic von Neumann algebra $\mathscr{A}$
and $\cH$ is a Hilbert space, and $\phi$ is normal (i.e. ultraweakly continuous), then there is a
normal  dilation $\pi$ (Theorem \ref{th:t531}).  If $\phi$ is not cb then the dilation space cannot be
a Hilbert space; the normality of the dilation is with respect to the natural ultraweak topology on $B(Z)$.
It is not known the extent to which this result can be generalized (i.e. achieving normality of the dilation).

\section{Algebraic Dilations}

\begin{proposition}\label{th:t50}
If $A$ is unital algebra, $V$  a vector space, and
$\phi:A\rightarrow L(V)$ a linear map, then there exists a vector
space $W$, a unital homomorphism $\pi:A\rightarrow L(V)$, and linear
maps $T:V\rightarrow W$, $S:W\rightarrow V$, such that
$$\phi(\cdot)=S\pi(\cdot) T.$$
\end{proposition}
\begin{proof}
For $a\in A, x\in V$, define $\alpha_{a,x}\in L(A,V)$ by
$$\alpha_{a,x}(\cdot):=\phi(\cdot a) x.$$ Let $W:=span\{\alpha_{a,x}:a\in A,x\in V\}\subset L(A,V).$
Define $\pi:A\rightarrow L(W)$ by
$\pi(a)(\alpha_{b,x}):=\alpha_{ab,x}.$ It is easy to see that
$\pi$ is a unital homomorphism. For $x\in V$ define $T:V
\rightarrow L(A,V)$ by $T_x:=\alpha_{I,x}=\phi(\cdot
I)x=\phi(\cdot)x$. Define $S:W\rightarrow W$ by setting
$S(\alpha_{a,x}):=\phi(a)x$ and extending linearly to $W$. If
$a\in A, x\in V$ are arbitrary, we have
$S\pi(a)Tx=S\pi(a)\alpha_{I,x}=S \alpha_{a,x}=\phi(a)x.$ Hence
$\phi=S\pi T.$
\end{proof}

\section{The Commutative Case}

We first examine a mapping from the commutative $C^*$-algebra
$\ell_\infty$ into $B(\cH)$ which is induced by a framing on a
Hilbert space $\cH$.  Here $\ell_{\infty}$ means with respect to a
countable or finite index set $J$, and it is well known that every
separably acting purely atomic abelian von Neumann algebra is
equivalent to some $\ell_{\infty}$ via an ultraweakly continuous
$*$-isomorphism.

\begin{theorem}\label{th:t51}
Let $\cH$ be a separable Hilbert space and let $(x_i,y_i)$ be a
framing of $\cH.$ Then the mapping $\phi$ from $\ell_\infty$ into
$B(\cH)$ defined by
\begin{eqnarray*}\label{eq:A1}
\phi:\ell_\infty\to B(\cH),\qquad (a_i)\to {\sum}^{SOT} a_ix_i\otimes
y_i
\end{eqnarray*}
is well-defined, unital, linear and ultraweakly continuous.
\end{theorem}
\begin{proof}
Since $(x_i,y_i)$ is a framing of $\cH,$ for any $x\in\cH,$
\[x=\sum_i\langle x, y_i\rangle x_i\]
converges unconditionally.

For any $(a_i)\in c_{00}$, define a bounded operator $U_{(a_i)}$ as
follows
\[U_{(a_i)}:\cH\to\cH,\quad U_{(a_i)}(x)=\sum a_i\left\langle x, y_i\right\rangle
x_i.\] For any $x\in\cH,$ since $\sum_i\langle x, y_i\rangle x_i$
converges unconditionally, by \cite{DJT} Theorem 1.9, we know that
\[\sup_{(b_i)\in
B_1({\ell_\infty})}\left\|\sum b_i\left\langle x, y_i\right\rangle
x_i\right\|<+\infty.\]  Thus,
\begin{eqnarray*}\label{eq:A11}
\sup_{(a_i)\in
B_1({c_{00}})}\left\|U_{(a_i)}(x)\right\|\leq\sup_{(b_i)\in
B_1({\ell_\infty})}\left\|\sum b_i\left\langle x, y_i\right\rangle
x_i\right\|<+\infty.
\end{eqnarray*}
Then by the Uniform Boundedness Principle,
\begin{eqnarray*}\label{eq:A12}
\sup_{(a_i)\in B_1({c_{00}})}\left\|U_{(a_i)}\right\|<+\infty.
\end{eqnarray*}
It follows that
\begin{eqnarray*}\label{eq:A2}
K_u=\sup_{x\in B_1(\cH)}\sup_{(\sigma_i)\subset\mathbb{D}}
\left\|\sum\sigma_i\left\langle x, y_i\right\rangle
x_i\right\|<+\infty.
\end{eqnarray*}
Thus, for all $(a_i)\in \ell_\infty$ and $x\in\cH,$
\begin{eqnarray*}\label{eq:A3}
\left\|\sum a_i\left\langle x, y_i\right\rangle x_i\right\| \leq K_u
\|a_i\|_{\ell_\infty} \|x\|.
\end{eqnarray*}
Hence $F$ is well-defined, unital, linear and bounded with
\[\|F\|_{B(\ell_\infty,B(\cH))}\leq K_u.\]
Now we prove that $F$ is ultraweakly continuous. If there is a net $(a_i^\lambda)$ converges to $0$ in the ultraweakly topology, then for any
$(\gamma_i)\in \ell_1,$ $\sum a_i^\lambda \gamma_i\to 0.$ Let $T$ belong to the trace class $S_1(\cH).$ By the polar decomposition, $T=U|T|$
where $U$ is a partial isometry. Moreover, recall that $S_1(\cH)$ is the subset of the compact operators $K(\cH),$ $|T|$ is a self-adjoint
compact operator. Thus there is an orthonormal basis $(e_i)$ and a sequence $\lambda_i\geq 0$ so that
\[|T|=\sum\lambda_ie_i\otimes e_i\]
with $\|T\|_{S_1}=tr(|T|)=\sum_i\lambda_i<\infty.$ Then for all
$(a_i)\in \ell_\infty, (\gamma_j)\in\ell_1$ and $(u_j),(v_j)\subset
B_1(\cH),$ we have
\begin{eqnarray*}\label{eq:A4}
\sum_i\sum_j\left|a_i\gamma_j\left\langle u_j, y_i\right\rangle\left\langle x_i, v_j\right\rangle\right|\nonumber
&\leq&\|(a_i)\|_{\infty}\sum_j|\gamma_j|\sum_i\left|\left\langle u_j, y_i\right\rangle
\left\langle x_i, v_j\right\rangle\right|\nonumber\\
&=&\|(a_i)\|_{\infty}\sum_j|\gamma_j|\sum_i\theta_{i,j}\left\langle
u_j, y_i\right\rangle\left\langle x_i,
v_j\right\rangle \nonumber\\
&=&\|(a_i)\|_{\infty}\sum_j|\gamma_j|\left\langle\sum_i\theta_{i,j}\left\langle
u_j, y_i\right\rangle x_i, v_j\right\rangle\nonumber\\
&\leq&\|(a_i)\|_{\infty}\sum_j|\gamma_j|\sup_j\left\|\sum_i\theta_{i,j}\left\langle
u_j, y_i\right\rangle x_i\right\|\nonumber\\
&\leq&\|(a_i)\|_{\infty}\sum_j|\gamma_j|K_u<\infty,
\end{eqnarray*}
where $\overline{\theta_{i,j}} =sgn\left\{\left\langle u_j, y_i\right\rangle\left\langle x_i, v_j\right\rangle\right\}$. So we have
\[\sum_i\sum_j\left|a^\lambda_i\lambda_j\left\langle Ue_j, y_i\right\rangle\left\langle x_i, e_j\right\rangle\right|<\infty\] and
\[\sum_i\left|\sum_j\lambda_j\left\langle Ue_j,
y_i\right\rangle\left\langle x_i, e_j\right\rangle\right|<\infty.\] Therefore
\begin{eqnarray*}\label{eq:A5}
tr\left(\phi(a^\lambda_i)T\right)
&=&\sum_j\left\langle \phi(a^\lambda_i)Te_j, e_j\right\rangle\nonumber\\
&=&\sum_j\left\langle \phi(a^\lambda_i)U\left(\sum_k\lambda_ke_k\otimes e_k\right)e_j, e_j\right\rangle\nonumber\\
&=&\sum_j\left\langle \phi(a^\lambda_i)U\lambda_je_j, e_j\right\rangle\nonumber\\
&=&\sum_j\lambda_j\left\langle \sum_ia^\lambda_ix_i\otimes y_iUe_j, e_j\right\rangle\nonumber\\
&=&\sum_j\lambda_j\sum_ia^\lambda_i\left\langle Ue_j, y_i\right\rangle\left\langle x_i, e_j\right\rangle, \\
&=&\sum_ia^\lambda_i\sum_j\lambda_j\left\langle Ue_j, y_i\right\rangle\left\langle x_i, e_j\right\rangle,
\end{eqnarray*}
which converges to $0$, as claimed.
\end{proof}

The main purpose of this section is to show that for every
ultraweakly continuous mapping $\phi$ from a purely atomic abelian
von Neumann algebra $\mathcal{A}$ into $B(\cH)$, we can find a
Banach space $Z$, an ultraweakly continuous unital homomorphism
$\pi$ from $\mathcal{A}$ into $B(Z)$,  and bounded linear
operators $T$ and $S$ such that for all $a\in \mathcal{A}$,
\[\phi(a)=S\pi(a)T.\] This result differs from Stinespring's dilation because the map $\phi$ here
is not necessarily completely bounded and consequently the
dilation space is not necessarily a Hilbert space.

While the ultraweak  topology on $B(\cH)$ for a Hilbert space
$\cH$ is well-understood,  we define the ultraweak topology on
$B(X)$ for a Banach space $X$ through tensor products: Let
$X\otimes Y$ be the tensor product of the Banach space $X$ and
$Y.$ The projective norm on $X\otimes Y$ is defined by:
\[\|u\|_{\wedge}=\inf\left\{\sum_{i=1}^n\|x_i\|\|y_i\|:u=\sum_{i=1}^n x_i\otimes y_i\right\}.\]
We will use $X\otimes_{\wedge} Y$ to denote the tensor product
$X\otimes Y$ endowed with the projective norm
$\|\cdot\|_{\wedge}.$   Its completion will be denoted by
$X\widehat{\otimes} Y.$ From \cite{R} Section 2.2, for any Banach
spaces $X$ and $Y,$ we have the identification:
\[(X\widehat{\otimes} Y)^*=B(X,Y^*).\]
Thus $B(X,X^{**})=(X\widehat{\otimes}  X^*)^*.$ Viewing $X\subseteq
X^{**},$ we define the {\it ultraweak topology} on $B(X)$ to be the
weak* topology induced by the predual $X\widehat{\otimes} X^*.$  We
will usually use the term {\it normal} to denote an ultraweakly
continuous linear map.

The following lemma generalizes Theorem \ref{th:t51} and will be
used in the proof of Theorem \ref{th:t530} of this section. The
proof is similar to that of Theorem \ref{th:t51} and we include a
sketch for completeness.

\begin{lemma}\label{le:51}
Let $X$ be a Banach space and let $E:2^\N\to B(X)$ be an
operator-valued measure on $(\N,2^\N)$. Denote $E(\{i\})$ by
$E_i$ for all $i\in\N$. Then the mapping $\phi$ from
$\ell_\infty$ into $B(X)$ defined by
\begin{eqnarray*}\label{eq:A1}
\phi:\ell_\infty\to B(X),\qquad (a_i)\mapsto {\sum}^{SOT} a_iE_i
\end{eqnarray*}
is well-defined, linear and ultraweakly continuous.
\end{lemma}
\begin{proof}
Since $E:2^\N\to B(X)$ is an operator-valued measure on $X$, we
have for all $x\in X,$
\[\sum_iE_i(x)\]
converges unconditionally. Similar to the proof in Theorem \ref{th:t51}, we get that
\begin{eqnarray*}\label{eq:A2}
K_u=\sup_{x\in B_1(X)}\sup_{(\sigma_i)\subset\mathbb{D}}
\left\|\sum_i\sigma_iE_i(x)\right\|<+\infty,
\end{eqnarray*}
and so for all $(a_i)\in \ell_\infty$ and $x\in X,$ we have
\begin{eqnarray*}\label{eq:A3}
\left\|\sum_i a_iE_i(x)\right\| \leq K_u \|a_i\|_{\ell_\infty}
\|x\|.
\end{eqnarray*}
Thus $\phi$ is well-defined, linear and bounded with
\[\|\phi\|_{B(\ell_\infty,B(X))}\leq K_u.\]

For the ultraweakly continuity of $\phi$, let $(a_i^\lambda)$ be a
net converging to $0$ in the ultraweak topology. Then for any
$(\gamma_i)\in \ell_1,$ $\sum a_i^\lambda \gamma_i\to 0.$ Let $w\in
X\hat{\otimes}_\pi X^*$. Then there is a pair of sequences
$(u_j,v_j)\subset X/\{0\}\times X^*/\{0\}$ with the property that
$\sum \|u_j\|\|v_j\|<\infty$ and $w=\sum u_j\otimes v_j$. Thus for
all $(a_i)\in \ell_\infty$, we have
\begin{eqnarray*}\label{eq:A13}
\sum_i\sum_j\left|a_i\left\langle E_i(u_j), v_j\right\rangle\right|\nonumber
&=&\sum_i|a_i|\sum_j\left|\left\langle E_i(u_j), v_j\right\rangle\right|\nonumber\\
&\leq&\|(a_i)\|_{\infty}\sum_i\sum_j\left|\left\langle E_i(u_j), v_j\right\rangle\right|\nonumber\\
&=&\|(a_i)\|_{\infty}\sum_j\sum_i\left|\left\langle E_i(u_j), v_j\right\rangle\right|\nonumber\\
&=&\|(a_i)\|_{\infty}\sum_j\|u_j\|\|v_j\|\sum_i\left|\left\langle
E_i\left(\frac{u_j}{\|u_j\|}\right), \frac{v_j}{\|v_j\|}\right\rangle\right|\nonumber\\
&=&\|(a_i)\|_{\infty}\sum_j\|u_j\|\|v_j\|\sum_i\theta_{i,j}\left\langle
E_i\left(\frac{u_j}{\|u_j\|}\right),
\frac{v_j}{\|v_j\|}\right\rangle\nonumber\\&&\quad\mbox{here}\quad
\overline{\theta_{i,j}} =\mathrm{sgn}\left\{\left\langle
E_i\left(\frac{u_j}{\|u_j\|}\right), \frac{v_j}{\|v_j\|}\right\rangle\right\} \nonumber\\
&=&\|(a_i)\|_{\infty}\sum_j\|u_j\|\|v_j\|\left\langle\sum_i\theta_{i,j}
E_i\left(\frac{u_j}{\|u_j\|}\right), \frac{v_j}{\|v_j\|}\right\rangle\nonumber\\
&\leq&\|(a_i)\|_{\infty}\sum_j\|u_j\|\|v_j\|
\sup_j\left\|\sum_i\theta_{i,j}E_i\left(\frac{u_j}{\|u_j\|}\right)\right\|\nonumber\\
&\leq&\|(a_i)\|_{\infty}\sum_j\|u_j\|\|v_j\|K_u < \infty.
\end{eqnarray*}
Thus we get that
\[\sum_j\sum_i|a^\lambda_i\left\langle E_i(u_j),
v_j\right\rangle|=\sum_i\sum_j|a^\lambda_i\left\langle E_i(u_j), v_j\right\rangle|=\sum_i|a^\lambda_i|\sum_j|\left\langle E_i(u_j),
v_j\right\rangle|<\infty\] and
\[\sum_i\left|\sum_j\left\langle E_i(u_j),
v_j\right\rangle\right|\le\sum_i\sum_j|\left\langle E_i(u_j), v_j\right\rangle|<\infty.\] So \[ \left(\sum_j\left\langle E_i(u_j),
v_j\right\rangle\right)_i\in\ell_1.\]

Therefore
\begin{eqnarray*}\label{eq:A5}
\phi(a^\lambda_i)w &=&\sum_j\left\langle \phi(a^\lambda_i)u_j, v_j\right\rangle \\
&=&\sum_j\left\langle\sum_i a^\lambda_iE_i(u_j), v_j\right\rangle\\
& = & \sum_j\sum_ia^\lambda_i\left\langle E_i(u_j), v_j\right\rangle\\
 & =&\sum_ia^\lambda_i\sum_j\left\langle E_i(u_j),
v_j\right\rangle
\end{eqnarray*} converges to $0,$ as expected.
\end{proof}

\begin{theorem}\label{th:t530}
Let $H$ be a separable Hilbert space and  $\phi: \ell_{\infty}(\Bbb N) \rightarrow B(\cH)$ such
that $\phi(1) = I$ and $\phi(e_{n})$ is at most rank one
for all $n\in\Bbb{N}$, where $e_{n} = \chi_{\{n\}}$ and $1$ is the function $1$ in $\ell_{\infty}$. Then the following are equivalent:

(i) $\phi$ is ultraweakly continuous,

(ii) the induced measure $E$ defined by defined by $E(B) = \phi(\sum_{n\in B}e_{n})$ for any $B\subseteq \Bbb{N}$ is an operator-valued
measure;

(iii) $\phi$ is induced by a framing $(x_{n},y_{n})$ for $\cH$, i.e,
$$
\phi(\sum_{n\in\Bbb{N}} a_{n}e_{n}) = {\sum}^{SOT} a_n x_n\otimes y_n
$$
\end{theorem}
\begin{proof} $(i) \Rightarrow (ii)$ and $(i)\Rightarrow (iii)$ are obvious from the definition of an operator valued measure and the ultraweakly continuity of $\phi$.
$(ii) \Rightarrow (i)$ follows from Lemma \ref{le:51}. $(iii) \Rightarrow (i)$ follows from Theorem \ref{th:t51}.
\end{proof}

\begin{corollary}\label{cor:c530}
Let $H$ be a separable Hilbert space and $\phi: \ell_{\infty} \rightarrow B(\cH)$ such that $\phi(1) = I$ and $\phi(e_{n})$ is at most rank one for all
$n\in\Bbb{N}$. Then there exist a separable Banach space $Z,$ an ultraweakly continuous unital homomorphism $\pi:\ell_\infty\to B(Z)$, and bounded
linear operators $T:\cH\to Z$ and $S:Z\to\cH$ such that
\[\phi(a)=S\pi(a)T\]
for all $a\in\ell_\infty,$ and $\pi(e_{n})$ is rank one for all $n\in \Bbb{N}$.
\end{corollary}
\begin{proof}  By Theorem \ref{th:t530}, $\phi$ is induced by a framing $(x_{n},y_{n})$ for $\cH$. Thus by Theorem 4.6 in \cite{CHL} $(x_n,y_n)$ can
be dilated an unconditional basis $\{u_{n}\}$ for a Banach space $Z$ (Hence $Z$ is separable). Let $\pi$ be the induced operator valued map by
$(u_{\lambda},u^{*}_{\lambda})$ (where $\{u^{*}_{\lambda}\}$ is the dual basis of $\{u_{\lambda}\}$). Then $\pi$ satisfies all the requirements.
\end{proof}

By using our main dilation result in Chapter 2 we are able to generalize the above result to more general  ultraweakly continuous operator
valued mapping.

\begin{theorem}\label{th:t531}
Let $\phi:\ell_\infty\to B(\cH)$ be an ultraweakly continuous
linear mapping. Then there exists a Banach space $Z,$ an
ultraweakly continuous unital homomorphism $\pi:\ell_\infty\to
B(Z)$, and bounded linear operators $T:\cH\to Z$ and $S:Z\to\cH$
such that
\[\phi(a)=S\pi(a)T\]
for all $a\in\ell_\infty.$
\end{theorem}
\begin{proof}
Let $E:2^\N\to B(\cH)$ be defined by
\[E(N)=\phi(\chi_N)\quad\mbox{for all } N\subset\N.\] If $(N_i)$ is a
sequence of disjoint subsets of $\N$ with union $N$, then it follows easily that $\sum \chi_{N_i}$ converges to $\chi_N$ under the ultraweak
topology of $\ell_\infty$. Since $\phi$ is ultraweakly continuous and $x\otimes y$ belongs to the trace class $S_1(\cH)$ for all $x,y\in\cH$,we get
that
\begin{eqnarray*}&&\langle E(N)x,y\rangle=\langle \phi(\chi_N)x,y\rangle
=\phi(\chi_N)( x\otimes y)\\&=&\sum_i \phi(\chi_{N_i}) (x\otimes y)=\sum_i
\langle \phi(\chi_{N_i})x,y\rangle=\sum_i\langle E(N_i)x,y\rangle.
\end{eqnarray*}
Thus $E:2^\N\to B(\cH)$ is an operator-valued measure on $(\N,2^\N)$. Let $$(\Omega,\Sigma,\widetilde{M}_{E,\mathcal
{M}},\rho_{\mathcal {M}},S_{\mathcal {M}},T_{\mathcal {M}})$$ be its minimal dilation system.  By Lemma \ref{le:51}, the mapping $\pi$ from
$\ell_\infty$ to $B(\widetilde{M}_E)$ defined by
\[\pi:\ell_\infty\to B(\widetilde{M}_{E,\mathcal
{M}}), \quad \pi(a_i)={\sum_i}^{SOT} a_i \rho_{\mathcal {M}}(\{i\})\]
is an ultraweakly continuous unital homomorphism. Moreover, for all
$(a_i)\in\ell_\infty$ and $x\in\cH$, we have
\[S_{\mathcal {M}} \pi(a_i)T_{\mathcal {M}}x=S_{\mathcal {M}}\sum_i a_i\rho_{\mathcal {M}}(\{i\})E_{x,\N}
=S_{\mathcal {M}}\sum_i a_iE_{x,\{i\}}=\sum_i
a_iE(\{i\})x=\phi(a_i)x.\] This completes the proof.
\end{proof}

Since every separably acting purely atomic abelian von Neumann
algebra is equivalent to some $\ell_{\infty}$ via an ultraweakly
continuous $*$-isomorphism, we immediately get the following:

\begin{theorem}\label{th:t531} Let $\mathcal{A}$ be a purely
atomic abelian von Neumann algebra acting on a separable Hilbert
space. Then for every ultraweakly continuous linear map
$\phi:\mathcal{A}\to B(\cH)$, there exists a Banach space $Z,$ an
ultraweakly continuous unital homomorphism $\pi:\ell_\infty\to
B(Z)$, and bounded linear operators $T:\cH\to Z$ and $S:Z\to\cH$
such that
\[\phi(a)=S\pi(a)T\]
for all $a\in\mathcal{A}.$

\end{theorem}

Every ultraweakly continuous linear map $\phi:L^\infty(\Omega,\Sigma,\mu)\rightarrow B(H)$
induces an OVM $E:(\Omega,\Sigma)\rightarrow B(H)$. Please notice that here $E$ is absolutely
continuous with respect to $\mu$, $E\ll\mu$, that is, $\mu(E)=0$ implies $E(E)=0$.
\begin{lemma}\label{lemma-ext}
If $E: (\Omega,\Sigma)\rightarrow B(\mathcal{H})$ is an OVM, and if
$\mu$ is a non-negative scalar measure on $(\Omega,\Sigma)$ such
that $E\ll\mu$ (that is, $E$ is absolutely continuous with respect
to $\mu$), then there exists a bounded linear map
$\phi:L^\infty(\mu)\rightarrow B(\mathcal{H})$ that induces $E$ on
$(\Omega,\Sigma),$ (that is, $E(B)=\phi(\chi_B)$ for all
$B\in\Sigma).$
\end{lemma}
\begin{proof}
For any simple function $\sum_{i=1}^n \alpha_i \chi_{B_i}$ with disjoint $B_i\in \Sigma$, define $\phi(\sum_{i=1}^n \alpha_i \chi_{B_i})=\sum_{i=1}^n \alpha_i E(B_i)$. If $\sum_{i=1}^n \alpha_i \chi_{B_i}=\sum_{j=1}^m \beta_j \chi_{A_j}$ for disjoint $B_i$ and disjoint $A_j$, then $\sum_{i=1}^n\sum_{j=1}^m (\alpha_i-\beta_j) \chi_{B_i\cap A_j}=0$. Thus, if $\alpha_i-\beta_j\neq 0$, then $\mu(B_i\cap A_j)=0$. By $E\ll\mu$, we obtain that $\phi(\sum_{i=1}^n \alpha_i \chi_{B_i})=\phi(\sum_{j=1}^m \beta_j \chi_{A_j})$, $\phi$ is well-defined on the subspace of all simple functions in $L^\infty(\mu)$. To prove that $\phi$ is linear, we only need to notice that $\sum_{i=1}^n \alpha_i \chi_{B_i}+\sum_{j=1}^m \beta_j \chi_{A_j}=\sum_{i=1}^n\sum_{j=1}^m (\alpha_i+\beta_j) \chi_{B_i\cap A_j}$.
For the boundedness, we need the following uniformly boundedness first, since $\sup_{E\in\Sigma}\|E(B)\|<\infty$, it is easy to obtain that $$\sup_{\{B_i\}_{i=1}^n \mbox{ is a partition of } \Omega}\sup_{\|x\|,\|y\|\le1}
\sum_{i=1}^n|\langle E(B_i)x,y \rangle|=M_{\phi}<\infty.$$ Then we have
\begin{eqnarray*}
\left\|\phi\left(\sum_{i=1}^n\alpha_i \chi_{B_i}\right)\right\|&=&\left\|\sum_{i=1}^n\alpha_iE(B_i)\right\|\\
&=&\sup_{\|x\|,\|y\|\le1}\sum_{i=1}^n|\alpha_i\langle E(B_i)x,y\rangle|\\
&\le&\left( \sup_{\|x\|,\|y\|\le1}
\sum_{i=1}^n|\langle E(B_i)x,y \rangle| \right) \max_{1\le i\le n} |\alpha_i|\\
&\le& M_{\phi}\left\|\sum_{i=1}^n\alpha_i \chi_{B_i}\right\|_{L^\infty}.
\end{eqnarray*}
Because the simple functions are dense in $L^\infty(\mu)$, the
conclusion follows.
\end{proof}

For an OVM $E:(\Omega,\Sigma)\rightarrow B(H)$, where $H$ is a
separable Hilbert space, there always exists a non-negative scalar
measure $\mu$ on $(\Omega,\Sigma)$ such that $E\ll\mu$ holds. 
Then we immediately get that
\begin{corollary}\label{cor-ext}
If $E: (\Omega,\Sigma)\rightarrow B(\mathcal{H})$ is an OVM, then
there exists $\mu$ which is a non-negative scalar measure on
$(\Omega,\Sigma)$ such that $E\ll\mu$, and a bounded linear map
$\phi:L^\infty(\mu)\rightarrow B(\mathcal{H})$ that induces $E$ on
$(\Omega,\Sigma).$
\end{corollary}

We remark that the conclusion of Corollary \ref{cor-ext} actually
holds for more general von Neumann algebras \cite{BW}: Let
$\mathcal{A}$ be a von Neumann algebra without direct summand of
type $I_{2}$. Then every bounded and finitely additive $B(H)$-valued
measure (c.f. \cite{BW} for definition) on the projection lattice of
$A$ can be uniquely extend to bounded linear map from $\mathcal{A}$
to $B(H)$. We include Lemma \ref{lemma-ext} and Corollary
\ref{cor-ext} here since they are directly related to the following
problem. We think if the answer to Problem C is positive then a
solution might be along the lines of the proof of Lemma
\ref{lemma-ext}.
\bigskip

\noindent{\bf Problem C.} Let $E: (\Omega,\Sigma)\rightarrow
B(\mathcal{H})$ be an OVM. Is there an ultraweakly continuous map
$\phi:L^\infty(\mu)\rightarrow B(\mathcal{H})$ that induces $E$ on
$(\Omega,\Sigma)?$

\section{The Noncommutative Case}

\begin{theorem}\label{th:5B1} [{\bf Banach Algebra Dilation Theorem}]
Let $\mathscr{A}$ be a Banach algebra, and let
$\phi:\mathscr{A}\to B(\cH)$ be a bounded linear operator, where
$H$ is a Banach space. Then there exists a Banach space $Z,$ a
bounded linear unital homomorphism $\pi:\mathscr{A}\to B(Z)$, and
bounded linear operators $T:\cH\to Z$ and $S:Z\to\cH$ such that
\[\phi(a)=S\pi(a)T\]
for all $a\in\mathscr{A}.$
\end{theorem}
\begin{proof} Consider the algebraic tensor product space
$\mathscr{A}\otimes\cH$, and define a $B(\mathscr{A},\cH)$ operator tensor norm with respect to $\phi$ as follows: for any $a\in\mathscr{A}$
and $x\in\cH,$ identity $a\otimes x$ with the map $a\rightarrow\cH$ defined by
\[(a\otimes x)(b)=\phi(ba)x.\]
Then $a\otimes x$ is a bounded linear operator from $\mathscr{A}$ to $\cH$ with $\|a\otimes
x\|_{B(\mathscr{A},\cH)}\leq\|\phi\|\|a\|\|x\|.$ So this defines a quasi-norm on $\mathscr{A}\otimes\cH.$ Let $$\mathcal
{N}=\left\{\sum_{i}c_{i}a_{i}\otimes x_{i}: ||\sum_{i}c_{i}a_{i}\otimes x_{i}||_{B(\mathscr{A},\cH)} = 0 \right\}$$ and
$\mathscr{A}\widetilde{\otimes}_\phi\cH/\mathcal {N}$ be the completion of the norm space $\mathscr{A}\otimes\cH/\mathcal {N}.$

For any $a\in\mathscr{A},$ let us define
$\pi(a):\mathscr{A}\otimes\cH/\mathcal {N}
\to\mathscr{A}\otimes\cH/\mathcal {N}$ by
$$\pi(a)\left(\sum_i a_i\otimes x_i\right)=\sum_i
(aa_i)\otimes x_i.$$

Assume that $f=\sum^n_{i=1} a_i\otimes x_i=\sum^m_{j=1} b_j\otimes y_i.$ Then we have
\begin{eqnarray*}\label{eq:V1}
\pi(a)\left(\sum_i a_i\otimes x_i\right)(b) &=&\left(\sum_i
(aa_i)\otimes x_i\right)(b)\\
&=&\sum_i\phi(baa_i)x_i\nonumber\\
&=&\left(\sum_i a_i\otimes x_i\right)(ba)\nonumber\\
&=&f(ba),
\end{eqnarray*}
and
\begin{eqnarray*}\label{eq:V2}
\pi(a)\left(\sum_i b_j\otimes y_j\right)(b) &=&\left(\sum_i
(ab_j)\otimes y_j\right)(b)\\
&=&\sum_i\phi(bab_j)y_j\nonumber\\
&=&\left(\sum_i b_j\otimes y_j\right)(ba)\nonumber\\
&=&f(ba).
\end{eqnarray*}
Therefor $\pi(a)$  is well defined.

The boundedness  $\pi(a)$ (with $\|\pi(a)\|\leq\|a\|$) follows from
\begin{eqnarray*}\label{eq:V3}
\left\|\pi(a)(f)\right\| &=&\left\|\pi(a)\left(\sum_i
a_i\otimes x_i\right)\right\|\nonumber\\
&=&\left\|\sum_i (aa_i)\otimes x_i\right\|\nonumber\\
&=&\sup_{b\in B_{\mathscr{A}}}\left\|\left(\sum_i(aa_i)\otimes x_i\right)(b)\right\|\nonumber\\
&=&\sup_{b\in
B_{\mathscr{A}}}\left\|\sum_i\phi(baa_i)x_i\right\|\nonumber\\
&=&\|a\|\sup_{b\in
B_{\mathscr{A}}}\left\|\sum_i\phi(b\frac{a}{\|a\|}a_i)x_i\right\|\nonumber\\
&\leq&\|a\|\sup_{b\in
B_{\mathscr{A}}}\left\|\sum_i\phi(ba_i)x_i\right\|\nonumber\\
&=&\|a\|\sup_{b\in B_{\mathscr{A}}}\left\|\left(\sum_i
a_i\otimes x_i\right)(b)\right\|\nonumber\\
&=&\|a\|\left\|\sum_i a_i\otimes
x_i\right\|\nonumber\\
&=&\|a\|\cdot\|f\|.
\end{eqnarray*}

Extend $\pi(a)$ to a bounded linear operator on $\mathscr{A}\widetilde{\otimes}_\phi\cH/\mathcal {N}$ which we still denote it by $\pi(a)$.
Thus $\pi:\mathscr{A}\to \mathscr{A}\widetilde{\otimes}_\phi\cH/\mathcal {N}$ is a bounded linear operator. Moreover,
\begin{eqnarray*}\label{eq:V4}
\pi(ab)\left(\sum _ia_i\otimes x_i\right)
&=&\sum_i (aba_i)\otimes x_i\nonumber\\
&=&\pi(a)\left(\sum_i (ba_i)\otimes x_i\right)\nonumber\\
&=&\pi(a)\pi(b)\left(\sum_i a_i\otimes x_i\right).
\end{eqnarray*}
Hence\[\pi(ab)=\pi(a)\pi(b),\] and therefore $\pi$ is a
homomorphism.

Define $T:\cH\to\mathscr{A}\widetilde{\otimes}_\phi\cH/\mathcal
{N}$ by $T(x)=\textbf{1}\otimes x.$ Since
\begin{eqnarray*}\label{eq:V5}
\|T(x)\|&=&\left\|\textbf{1}\otimes
x\right\|\nonumber\\
&=&\sup_{a\in B_{\mathscr{A}}}\left\|\left(\textbf{1}\otimes
x\right)(a)\right\|\nonumber\\
&=&\sup_{a\in B_{\mathscr{A}}}\left\|\phi(a)x\right\|\nonumber\\
&\leq&\sup_{a\in
B_{\mathscr{A}}}\|a\|\|\phi\|\|x\|\nonumber\\
&=&\|\phi\|\|x\|,
\end{eqnarray*}
we have that $T$ is a bounded linear operator with $\|T\|\leq\|\phi\|.$

Define $S:\mathscr{A}\otimes\cH/\mathcal {N}\to \cH$ by
$S(a\otimes x)=E(a)x$ and linearly extend $S$ to
$\mathscr{A}\otimes\cH/\mathcal {N}.$ It is easy to check
that $S$ is well-defined. Since
\begin{eqnarray*}\label{eq:V6}
\|S(a\otimes x)\|&=&\left\|\phi(a)
x\right\|\nonumber\\
&\leq&\sup_{b\in B_{\mathscr{A}}}\left\|\phi(ba)x\right\|\nonumber\\
&=&\sup_{b\in B_{\mathscr{A}}}\left\|\left(a\otimes x\right)(b)\right\|\nonumber\\
&=&\|a\otimes x\|,
\end{eqnarray*}
we have that $S$ is a bounded linear operator with $\|S\|\leq 1.$ Extend $S$ to a bounded linear operator from
$\mathscr{A}\widetilde{\otimes}_\phi\cH/\mathcal {N}$ to $\cH,$ which we still denote it by $S$.

Finally, for any $x\in \cH$ we have
\begin{eqnarray*}\label{eq:V7}
S\pi(a)T(x)&=&S\rho(a)\left(\textbf{1}\otimes
x\right)\nonumber\\
&=&S\left(a\otimes
x\right)\nonumber\\
&=&\phi(a)x.
\end{eqnarray*}
Thus $\phi(a)=S\pi(a)T.$
\end{proof}

\begin{remark}\label{re:431}
Consider the construction of the dilation space in Theorem
\ref{th:5B1}, and compare it with the construction of the dilation
space in Theorem \ref{th:t531}. In the first case we take the
algebraic tensor product $\mathscr{A}\otimes\cH$ considered as a
linear space of operators in $B(\mathscr{A}, \cH)$ to obtain a
quasi-norm on $\mathscr{A}\otimes\cH$ which we denote
$\|\cdot\|_{B(\mathscr{A}, \cH)}$. Then we mod out by the kernel to
obtain a norm on the quotient space, and then we complete it to
obtain a Banach space. If we began this construction by using a
dense subsalgebra of $\mathscr{A}$ instead of $\mathscr{A}$ itself,
then the construction goes through smoothly, and the dilation space
is the same. In the case of the second space we work with
$\mathscr{A}=\ell_\infty$, and the reader can note that we start
with the algebraic tensor product of the dense subalgebra
$\mathscr{A}_0=\{\sum c_i\chi_{E_i}:c_i\in\C, E_i\in\Sigma\}$ with
$\cH$, and define a special ``minimal" quasi-norm $\|\cdot\|_1$ on
it as in Section 2.4, then mod out by the kernel, and complete it to
obtain the dilation space. In both cases we could begin with the
dense subalgebra $\mathscr{A}_0$ of $\ell_\infty$, form the
algebraic tensor product $\mathscr{A}_0\otimes\ell_\infty$, and put
a quasi-norm on the space. It is easy to show that these two
quasi-norms are equivalent in the sense that each one is dominated
by a constant multiple of the other. Hence the dilation spaces of
Theorem \ref{th:t531} and Theorem \ref{th:5B1} are equivalent for
the special case $\mathscr{A}=\ell_\infty$.

\end{remark}

 As with Stinespring's dilation theorem, if $A$ and $\cH$ in Theorem \ref{th:5B1} are both separable then  the dilated Banach
 space $Z$ is also
separable. However, the Banach algebras  we are interested in include von Neumann algebras and these
 are generally not separable, and the linear maps $\phi:A\rightarrow B(H)$ are often normal. So we pose the
following two problems.
\bigskip

\noindent{\bf Problem D.} Let $K,H$ be separable Hilbert spaces, let
$A\subset B(K)$ be a von Neumann algebra, and let $\phi:A\rightarrow
B(H)$ be a bounded linear map. When is there a {\it separable}
Banach space $Z,$ a bounded linear unital homomorphism
$\pi:\mathscr{A}\to B(Z)$, and bounded linear operators $T:\cH\to Z$
and $S:Z\to\cH$ such that
\[\phi(a)=S\pi(a)T\]
for all $a\in\mathscr{A}$ ?

\bigskip

\noindent{\bf Problem E.} Let  $A\subset B(K)$ be a von Neumann
algebra, and $\phi:A\rightarrow B(H)$ be a normal linear map. When
can we dilate $\phi$ to a {\it normal} linear unital homomorphism
$\pi:\mathscr{A}\to B(Z)$ for some Banach space $Z$?

\bigskip

Although we do not know the answer to Problem E, we do have the
following result:

\begin{theorem}\label{th:t47}
Let $K,H$ be Hilbert spaces, $A\subset B(K)$ be a von Neumann
algebra, and $\phi:A\rightarrow B(H)$ be a bounded linear operator
which is ultraweakly-\textsc{SOT} continuous on the unit ball
$B_A$ of $A$. Then there exists a Banach space $Z$, a bounded
linear homeomorphism $\pi:A\rightarrow B(Z)$ which is
\textsc{SOT}-\textsc{SOT} continuous on $B_A$, and bounded linear
operator $T:H\rightarrow Z$ and $S:Z\rightarrow H$ such that
$$\phi(a)=S \pi(a)T$$ for all $a\in A.$
\end{theorem}
\begin{proof} We only need to prove that the bounded linear homeomorphism $\pi$ constructed in Theorem \ref{th:5B1} is
\textsc{SOT}-\textsc{SOT} continuous.

Assume that a net $\{a_\lambda\}\subset B_A$. It is sufficient to prove that $a_\lambda\rightarrow 0$ (SOT) implies that
$\pi(a_\lambda)\rightarrow 0$ (SOT). Since $\{\pi(a_\lambda)\}$ is norm-bounded, we have that $\pi(a_\lambda)\rightarrow 0$ (SOT) if and only
if $\pi(a_\lambda)\xi\rightarrow 0$ in norm for a dense set of $\xi.$ Thus, we only need to prove that
$$\pi(a_\lambda)(a\otimes x)=(a_\lambda a)\otimes x\rightarrow 0$$
in the norm topology for all $a\in A$ and $x\in H.$ If this is not the case,  then there exist $\delta>0$, a subnet of $\{a_\lambda\}$ (still
denoted by $\{a_\lambda\}$), and $\{b_\lambda\}\subset B_A$ such that
$$\|(a_\lambda a\otimes x)(b_\lambda)\|=\|\varphi(b_\lambda a_\lambda a)(x)\|>\delta.$$

Since on any norm-bounded set the WOT and ultraweak topology are the same (and in particular the unit ball is compact in both topologies), there
is a subnet of $\{b_\lambda\}$ (still denoted by $\{b_\lambda\}$) converges to some element $b\in B_A$ in the ultraweak topology (or WOT). Thus,
from our hypothesis that $a_\lambda\rightarrow0$ (SOT), we get that $b_\lambda a_\lambda a\rightarrow0$ (SOT). This implies that $b_\lambda
a_\lambda a\rightarrow0$ in ultraweak topology (or WOT), and therefore $\varphi(b_\lambda a_\lambda a)\rightarrow0$ (SOT) since
ultraweakly-\textsc{SOT} continuous on $B_A$. This leads to a contradiction.
\end{proof}

\begin{corollary}\label{cor:e1}
Let $K,H$ be separable Hilbert spaces, $A\subset B(K)$ be a von
Neumann algebra, and $\phi:A\rightarrow B(H)$ be a bounded linear
operator which is ultraweakly-\textsc{SOT} continuous on $B_A$.
Then there exists a separable Banach space $\widetilde{Z}$, a
bounded linear homeomorphism $\tilde{\pi}:A\rightarrow
B(\widetilde{Z})$ which is \textsc{SOT}-\textsc{SOT} continuous on
$B_A$, and bounded linear operator $\widetilde{T}:H\rightarrow
\widetilde{Z}$ and $\widetilde{S}:\widetilde{Z}\rightarrow H$ such
that
$$\phi(a)=\widetilde{S} \tilde{\pi}(a)\widetilde{T}$$ for all $a\in A.$
\end{corollary}

\begin{proof}
By Theorem \ref{th:t47}, there exists a Banach space $Z$, a
bounded linear homeomorphism $\pi:A\rightarrow B(Z)$ which is
\textsc{SOT}-\textsc{SOT} continuous on $B_A$, and bounded linear
operator $T:H\rightarrow Z$ and $S:Z\rightarrow H$ such that
$\phi(a)=S \pi(a)T$ for all $a\in A.$ Since $K$ is separable,
there is $\{a_i\}_{i=1}^\infty$ SOT dense in $B_A$. Define
$$\widetilde{Z}=\overline{\mathrm{span}}^{\|\cdot\|}\{\pi(a_i)T(h): 1\le i<\infty, h\in H\}.$$
Then $\widetilde{Z}$ is separable since $H$ is separable. Moreover, the SOT-SOT continuity of $\pi$
(restricted to the unit ball of $A$) implies that $\pi(A)\widetilde{Z}\subset \widetilde{Z}$ and $\pi(A)T(H)\subset \widetilde{Z}$. Let $V:\widetilde{Z}\rightarrow Z$ be the inclusion linear map. Then
$$\phi(a)=S V^{-1}\pi(a)V T, \ \ \forall \, a\in A.$$ Let $\tilde{\pi}(a)=V^{-1}\pi(a)V$ acting on $\widetilde{Z}$. Then $\tilde{\pi}$ is a dilation of $\phi$ on the separable Banach space $\widetilde{Z}$ and $\tilde{\pi}:A\rightarrow B(\widetilde{Z})$ is also SOT-SOT continuous when restricted to the unit ball of $A$.
\end{proof}

\remark{Dilations and the similarity problem:} Let $\mathscr{A}$
be a $C^{*}$-algebra. In 1955, Kadison \cite{K1} formulated the
following still open conjecture: Any bounded homomorphism $\pi$
from a $C^*$-algebra $\mathscr{A}$ into the algebra $B(H)$ of all
bounded operators on a Hilbert space $H$ is similar to a
$*$-homomorphism, i.e. there is an invertible operator $T\in B(H)$
such that $T\pi(\cdot)T^{-1}$ is a $*$-homomorphism from
$\mathscr{A}$ to $B(H)$. This problem is known to be equivalent to
several famous open problems (c.f. \cite{Pi}) including the
derivation problem: Is every derivation from a $C^{*}$-algebra
$\mathscr{A}\subseteq B(H)$ into $B(H)$ inner?  While Kadison's
similarity problem remains unsettled, many remarkable partial
results are known.  In particular, it is well known that $\pi$ is
similar to a $*$-homomorphism if and only if it is completely
bounded. It has been proved  that a bounded unital homomorphism
$\pi \colon \mathscr{A} \to B(H)$ is completely bounded (and hence
similar to a $*$-representation) if $\mathscr{A}$ is nuclear
\cite{Bun, Chr}; or if $\mathscr{A} = B(H)$, or more generally if
$\mathscr{A}$ has no tracial states; or if $\mathscr{A}$ is
commutative; or if $\mathscr{A}$ is a $II_{1}$-factor with Murry
and von Neumann's property $\Gamma$; or if $\pi$ is cyclic
\cite{Haa}. Therefore if $\mathscr{A}$ belongs to any of the above
mentioned classes, and $\phi:\mathscr{A}\to B(H)$ is a bounded but
not completely bounded linear map, then the dilation space $Z$ in
Theorem \ref{th:5B1} [Banach algebra dilation theorem] can never
be a Hilbert space since otherwise $\pi:\mathscr{A}\to B(Z)$ would
be completely bounded and so would be $\phi$. On the other hand,
if there is a non completely bounded map $\phi$ from a
$C^{*}$-algebra to $B(H)$ that has a Hilbert space dilation:
$\pi:\mathscr{A}\to B(Z)$ (i.e., where $Z$ is a Hilbert space),
then it would be a counterexample to the Kadison's similarity
problem. So we have the following question: Is there a non-cb map
that admits a Hilbert space dilation to a bounded homomorphism?

\remark\label{re:415} For a countable index set $\Lambda$, there is a 1-1
correspondence between the set of (discrete) framings on a Hilbert
space $\cH$ indexed by $\Lambda$ and the set of ultraweakly
continous unital linear maps from $\ell_{\infty}(\Lambda)$ into
$B(\cH)$.  Here, unital means it takes the function $1$ in
$\ell_{\infty}(\Lambda)$ to the identity operator in $B(\cH)$.
There is also a 1-1 correspondence between the set of (discrete)
framings on a Hilbert space $\cH$ indexed by $\Lambda$ and the set
of purely atomic probability operator-valued measures on the
$\sigma$-algebra of all subsets of $\Lambda$ with rank-1 atoms in
$B(\cH)$. So we have a space of ordered triples which consists of
a discrete framing, a purely atomic operator-valued probability
measure with rank-1 atoms, and an ultraweakly continous unital
linear map from an purely atomic abelian von Neumann algebra into
$B(\cH)$ where each minimal projection is sent to a rank $\leq 1$
operator. Each item in a triple determines the other items
uniquely. There is a consistent dilation theory,  where the
dilation of the discrete framing, the operator-valued measure, and
the ultraweakly continous unital linear map, all have the same
dilation space and the dilation procedure commutes with the
correspondences in the triple,   i.e. when we go from framing to
OVM to linear map and then dilate each one to get a dilated
triple, it will be the same as if we dilate any one and then
derived the other two from it in the natural way. Similarly for a
more general index set $\Lambda$ (i.e. compact or locally compact,
whatever is more suitable) there are connections among the set of
operator-valued probability measures on $\Lambda$ taking values in
$B(\cH)$ for a Hilbert space $\cH$,  the set of unital ultraweakly
continuous maps from $\ell_{\infty}(\Lambda)$ into $B(\cH)$, and
their corresponding dilation theory.  Although not every
ultraweakly continuous unital linear mapping is associated with a
(continuous) framing, there are many ultraweakly continuous unital
linear mappings that are induced by framings. For this reason we
can view ultraweakly continuous unital linear mappings from
$\ell_{\infty}(\Lambda)$ to $B(\cH)$ as {\it abstract framings},
which is the commutative theory because the domain of the map is a
commutative von neumann algebra. When we pass to the
noncommutative domains, the dilation theory for various linear
maps can be viewed as a noncommutative (abstract) framing dilation
theory.

\chapter{Examples}

We provide the details of the two examples mentioned in the
introduction chapter. Our first example shows that there exists a
framing for a Hilbert space whose induced operator valued measure
fails to admit a Hilbert space dilation. Equivalently, it cannot
be re-scaled to obtain a framing that admits a Hilbert space
dilation. The construction is based on an example of Ozaka
\cite{Os}  of a normal non-completely bounded map of
$\ell^{\infty}(\mathbb{N})$ into B(H).

\begin{theorem}\label{th:82}There exist a framing for a Hilbert space such that its induced operator-valued measure is not completely bounded,
and consequently it can not be re-scaled to obtain a framing that
admits a Hilbert space dilation.
\end{theorem}

Since if a framing admits a  Hilbert space dilation, then the induced operator valued map is completely bounded and a re-scaled framing induces
the same operator valued map,  we only need to show that there exists a framing for a Hilbert space such that its induced operator-valued
measure is not completely bounded. We need the following lemmas.

\begin{lemma}\label{le:83'} Let $\{A_{n}\}$ be a sequence of finite-rank bounded linear operators on a Hilbert space $H$ such that

(i) $A_{n}A_{m} = A_{m}A_{n} = 0$ for all $n\neq m$;

(ii) there exist mutually orthogonal projections $\{P_{n}\}$ such that $A_{n} = P_{n}A_{n}P_{n}$ for all $n$.

(iii) $\sum_{n=1}^{\infty}A_{n}$ converges unconditionally to $A\in B(H)$ with the strong operator topology.

Assume that $A_{n}$ has the rank one operator decomposition $A_{n} = Q_{n,1}+ ... +Q_{n,k_{n}}$ such that for any subset $\Lambda$ of $J_{n} :=
\{(n,1), ..., (n, k_{n})\}$ we have  $||\sum_{j\in \Lambda}Q_{j} || \leq ||A_{n}||$. Let $J$ be the disjoint union of $J_{n}$ $(n = 1, 2, ...)$.
Then the series $\sum_{j\in J}Q_{j}$ converges unconditionally to $A$.
\end{lemma}
\begin{proof} Let $x\in H$. Without losing the generality we can assume that $x\in PH$ where $P = \sum_{n=1}^{\infty}P_{n}$. Let $\epsilon >0$.
Then there exists $N$ such that
$$
||\sum_{n=1}^{N}P_{n}x - x|| < \epsilon/2||A||.
$$
Now let $\{j_{\ell}\}_{\ell=1}^{\infty}$ be an enumeration of $J$ and let $L$ be such that $\{j_{1}, .... , j_{L}\}$ contain
$\cup_{n=1}^{N}J_{n}$.  Write $x = x_{1} + x_{0}$ with $x_{1} =\sum_{n=1}^{N}P_{n}x $ and $x_{0} = (\sum_{n=1}^{N}P_{n})^{\perp}x$. Then for any
$L' \geq L$, we get
$$
\sum_{\ell=1}^{L'}Q_{j_{\ell}}x = \sum_{n=1}^{N}\sum_{j_{\ell}\in J_{n}}Q_{j_{\ell}}x + \sum_{j_{\ell}\notin J_{n}, \ell\leq L'} Q_{j_{\ell}}x =
\sum_{n=1}^{N}A_{n}x + \sum_{j_{\ell}\notin J_{n}, \ell\leq L'} Q_{j_{\ell}}x_{0}
$$
where we use the property that $\sum_{j_{\ell}\notin J_{n},
\ell\leq L'} Q_{j_{\ell}}x_{1} = 0$. Moreover, from our
assumptions on $\{A_{n}\}$ and their rank-one decompositions we
also have
$$||\sum_{j_{\ell}\notin J_{n}, \ell\leq L'} Q_{j_{\ell}}x_{0}||
\leq \sup \{||A_{n}\} \leq ||A||$$ and $$\sum_{n=1}^{N}A_{n}x =
A\sum_{n=1}^{N}P_{n}x.$$ Therefor we get
\begin{eqnarray*}
||\sum_{\ell=1}^{L'}Q_{j_{\ell}}x - Ax|| &\leq & ||A\sum_{n=1}^{N}P_{n}x - Ax|| + || \sum_{j_{\ell}\notin J_{n}, \ell\leq L'}
Q_{j_{\ell}}x_{0}|| \\
& \leq & ||A|| \cdot ||\sum_{n=1}^{N}P_{n}x - x|| + ||A||\cdot ||x_{0}||\\
& = & 2||A|| \cdot ||x_{0}|| < \epsilon.
\end{eqnarray*} This completes the
proof.
\end{proof}

\begin{lemma}\label{le:84'} Let $A$ be a rank-$k$ operator. Then there exists rank-one decomposition $A = \sum_{i=1}^{k}Q_{i}$ such that
$||\sum_{i\in I}Q_{i}|| \leq ||A||$ for any subset $I$ of $\{1, 2, ... , k\}$.
\end{lemma}
\begin{proof}
Let $A = U|A|$ be its polar decomposition and write $|A| = \sum_{i=1}^{k}x_{i}\otimes x_{i}$. Then we have $||\sum_{i\in I}x_{i}\otimes x_{i}||
\leq ||\sum_{i=1}^{k}x_{i}\otimes x_{i}|| = ||A||$. Let $Q_{i} = Ux_{i}\otimes x_{i}$. Then we get
\begin{eqnarray*}
||\sum_{i\in I}Q_{i}|| &=& ||U\sum_{i\in I}x_{i}\otimes x_{i}||\\
& \leq & ||U||\cdot ||\sum_{i\in I}x_{i}\otimes x_{i}|| \\
&\leq & ||\sum_{i\in I}x_{i}\otimes x_{i}||\\
& \leq& ||\sum_{i=1}^{k}x_{i}\otimes x_{i}|| = ||A||
\end{eqnarray*}
holds for any subset $I$ of $\{1, 2, ... , k\}$, as claimed.
\end{proof}

Now we prove Theorem \ref{th:82}: By Lemma 2.4 (or Lemma 2.1) in \cite{Os} there exists a $\sigma$-weakly continuous bounded linear map $\psi$
from $\oplus_{n=1}^{\infty}\ell^{\infty}_{n}$ into $\oplus_{n=1}^{\infty}M_{2^{n}}$  (as a subalgebra acting on $H := \oplus \Bbb{C}^{2^{n}}$)
which is not completely bounded and $\psi = \oplus_{n=1}^{\infty}\psi_{n}$, where $\psi_{n}: \ell^{\infty}_{n} \rightarrow M_{2^{n}}$ and
$M_{2^{n}}$ is the algebra of $2^{n}\times 2^{n}$-matrices. Let $B_{n} = \sum_{i=1}^{n}\psi(e_{i}^{n})$,  where $e_{i}^{n} \in\ell^{\infty}_{n}$
such that $e_{i}^{n}(j) = \delta_{i, j}$, and let $B = \psi(e)$ where $e\in \oplus_{n=1}^{\infty}\ell^{\infty}_{n}$ is the unital element. Let
$P_{n}$ be the orthogonal projections from $H$ onto $\Bbb{C}^{2^{n}}$. Without losing the generality we can assume that $||\psi|| < 1$. Let
$A_{n} = B_{n} + P_{n}$. Then $\{A_{n}\}$ satisfies all the conditions (i)--(iii) in Lemma \ref{le:83'}, and $\sum_{n=1}^{\infty} A_{n}$
converges to $A =  B + I$ unconditionally in the strong operator topology. By Lemma \ref{le:84'}, we can decompose each $A_{n} = Q_{n,1}+ ...
+Q_{n,k_{n}}$ such that for any subset $\Lambda$ of $J_{n} := \{(n,1), ..., (n, k_{n})\}$ we have  $||\sum_{j\in \Lambda}Q_{j} || \leq
||A_{n}||$. Thus, by Lemma \ref{le:83'}, $\sum_{j\in J}Q_{j}$ converges unconditionally to $A$. Write $Q_{j} = x_{j}\otimes y_{j}$, and let
$u_{j} = A^{-1}x_{j}$. Then $\{u_{j}, y_{j}\}$ form a framing for $H$ since $\sum_{j\in J}u_{j}\otimes y_{j}$ converges unconditionally to $I$
in the strong operator topology.

We claim that the induced operator-valued map $\Phi: \ell^{\infty}(J)\rightarrow B(H)$ is not completely bounded. In fact, if it is completely
bounded, then operator valued map $\Psi$ by $\{x_{j}, y_{j}\}$ will be completely bounded. By embedding $\oplus_{n=1}^{\infty}\ell^{\infty}_{n}$
naturally to $\ell^{\infty}(J)$, we can view $\oplus_{n=1}^{\infty}\ell^{\infty}_{n}$ as a subalgebra of  $\ell^{\infty}(J)$, and so the
restriction of $\Psi$ to this subalgebra is $\psi + id$, where $id$ is the natural embedding of $\oplus_{n=1}^{\infty}\ell^{\infty}_{n}$ into
$B(H)$. This $\psi + id$ is completely bounded and so $\psi$ is completely bounded since the embedding map $id$ is completely bounded. This
leads to a contradiction. Therefore $\Phi: \ell^{\infty}(J)\rightarrow B(H)$ is not completely bounded, and thus  $\{u_{j}, y_{j}\}$ it can not
be re-scaled to obtain a framing that admits a Hilbert space dilation.\qed

Next we examine Example 3.9 of \cite{CHL}, which gave
the first apparently nontrivial example of a bounded
framing for a Hilbert space which is not a
dual pair of frames. A natural, geometric, Banach space dilation
 of it was constructed in \cite{CHL}, along with a proof that it did not have
any Hilbert dilation space. Both the construction and the proof was nontrivial, and both were sketched out in \cite{CHL} without providing great detail.  It was thought for a long time that this example might be a key to
this subject. But when we were finalizing this paper we found a proof that this example can, in fact, be {\it rescaled} to give a dual pair of frames, in fact a pair which consists of two copies of a single Parseval frame.  Since this example was a guiding example for a lengthy time, we include our analysis of it. After obtaining this result, since we knew we still needed a true limiting example for this theory, we managed to obtained the example in Theorem 5.1 above.  The \cite{CHL} example points out how extremely difficult it can be in general to determine when a given framing is scalable to a dual frame pair.

\begin{theorem}\label{th:81} Example 3.9 in \cite{CHL} can be re-scaled to a dual pair of frames.
\end{theorem}
\begin{proof} This result is very technical and it is a stand-alone result for the present paper because the details of the proof are not used anywhere else in this paper.  So we must refer the reader to \cite{CHL} for some of the terminology and Banach space background.  With this in hand this proof can be handily worked through.

Fix any $1<p<\infty$ and $p\neq 2$ and natural number $n$. Let $\{e_i^n\}_{i=1}^{2^n}$ be a unconditional unit basis of $\ell_p^{2^n}$ and
$\{(e_i^n)^*\}_{i=1}^{2^n}$ be the dual of $\{e_i^n\}_{i=1}^{2^n}$. We denote the Rademacher vectors in $\ell_p^{2^n}$ by $\{r_i^n\}_{i=1}^n$,
where
\[r_i^n=\frac{1}{2^{n/p}}\sum_{j=1}^{2^n}\epsilon_{ij}e_j^n\]
with $\|r_i^n\|_{\ell_p^{2^n}}=1$ and $\epsilon_{ij}$ satisfying the condition
\[\sum_{j=1}^{2^n}\epsilon_{ij}\epsilon_{kj}=\delta_{ik}\cdot 2^n.\]
Then $\{r_i^n\}_{i=1}^n$ are linearly independent because they are
the Rademacher vectors. So $W_n=\mathrm{span}\{r_i^n\}_{i=1}^n$ is a
$n$-dimensional subspace of $\ell_p^{2^n}$. For any $\sum_{i=1}^n
a_i r_i^n\in W_n$, we have
\begin{eqnarray}\label{eq:01}\edz{change l to i in order to make the index same}
\left\|\sum_{i=1}^n a_i
r_i^n\right\|_{\ell_p^{2^n}}&=&\left\|\sum_{i=1}^n
\frac{a_i}{2^{n/p}}\sum_{j=1}^{2^n}\epsilon_{ij}e_j^n\right\|_{\ell_p^{2^n}}\nonumber\\
&=& \frac{1}{2^{n/p}}\left\|\sum_{j=1}^{2^n}\left(\sum_{i=1}^n
a_i\epsilon_{ij}\right)e_j^n\right\|_{\ell_p^{2^n}}\nonumber\\
&=&\frac{1}{2^{n/p}}\left(\sum_{j=1}^{2^n}\left|\sum_{i=1}^n
a_i\epsilon_{ij}\right|^p\right)^{1/p}.
\end{eqnarray}

Let $\mathrm{sign}(\sin 2^i \pi t), i=0,1,\cdots,n$ be the Rademacher functions on $[0,1].$ By Theorem 2.b.3 in \cite{LT}, we have
\begin{eqnarray*}
\int_0^1 \left|\sum_{i=1}^n a_i \mathrm{sign}(\sin 2^i \pi
t)\right|^p dt
&=&\sum_{j=1}^{2^n}\int_{\frac{j-1}{2^n}}^{\frac{j}{2^n}}
\left|\sum_{i=1}^n
a_i \mathrm{sign}(\sin 2^i \pi t)\right|^p dt\\
&=&\sum_{j=1}^{2^n}\int_{\frac{j-1}{2^n}}^{\frac{j}{2^n}}\left|\sum_{i=1}^n
a_i \epsilon_{ij}\right|^p dt\\
&=&\sum_{j=1}^{2^n} \frac{1}{2^n}\left|\sum_{i=1}^n a_i
\epsilon_{ij}\right|^p.
\end{eqnarray*}
Let $A_p, B_p$ be the constants as in Theorem 2.b.3 in \cite{LT}.
Then for any $\{a_i\}^n_{i=1},$ we have
\[A_p\left(\sum_{i=1}^n \left|a_i\right|^2\right)^{1/2}
\leq \left(\sum_{j=1}^{2^n} \frac{1}{2^n}\left|\sum_{i=1}^n a_i
\epsilon_{ij}\right|^p\right)^{1/p} \leq B_p \left(\sum_{i=1}^n
\left|a_i\right|^2\right)^{1/2},\] that is
\[A_p\left(\sum_{i=1}^n \left|a_i\right|^2\right)^{1/2}
\leq \frac{1}{2^{n/p}}\left(\sum_{j=1}^{2^n}\left|\sum_{i=1}^n
a_i\epsilon_{ij}\right|^p\right)^{1/p} \leq B_p \left(\sum_{i=1}^n
\left|a_i\right|^2\right)^{1/2}.\] Thus from equation (\ref{eq:01}),
we obtain
\[A_p\left(\sum_{i=1}^n \left|a_i\right|^2\right)^{1/2}
\leq \left\|\sum_{i=1}^n a_i r_i^n\right\|_{\ell_p^{2^n}} \leq B_p
\left(\sum_{i=1}^n \left|a_i\right|^2\right)^{1/2}.\]

Let $\cH_n$ be the span of $\{r_1^n,\cdots, r_n^n\}$ with the inner product
\[\left\langle\sum_{i=1}^na_ir_i^n,\sum_{k=1}^nb_kr_k^n\right\rangle
=\sum_{i=1}^na_i\overline{b}_i.\] The induced Hilbert space norm is
\[\left\| \sum_{i=1}^na_ir_i^n\right\|_{\cH_n}
=\left(\sum_{i=1}^n \left|a_i\right|^2\right)^{1/2},\] and
$\{r_1^n,\cdots, r_n^n\}$ is an orthonormal basis for $\cH_n.$ Then
$\cH_n$ is a Hilbert space isometrically isomorphism to $l_2^n.$ The
spaces $\cH_n$ and $W_n$ are naturally isomorphism as Banach space,
being the same vector space with two different norms. Let
$U_n:W_n\to \cH_n,\quad U_nx=x$ be this isomorphism.

Define a mapping $P_n:\ell_p^{2^n}\to W_n$ by the dual Rademachers
\[(r_i^n)^*=\frac{1}{2^{n/q}}\sum_{j=1}^{2^n}\epsilon_{ij}(e_j^n)^*,\]
where $1/p+1/q=1.$ That is $P_n(x)=\sum_{i=1}^{n}(r_i^n)^*(x)r_i^n$
for any $x\in \ell_p^{2^n}.$ So we have
$P_n(e_k^n)=\sum_{i=1}^{n}(r_i^n)^*(e_k^n)r_i^n$.  Since
\begin{eqnarray*}
(r_i^n)^*(e_k^n)=\frac{1}{2^{n/q}} \sum_{j=1}^{2^n}\epsilon_{ij}(e_j^n)^*(e_k^n) =\frac{1}{2^{n/q}}\epsilon_{ik},
\end{eqnarray*}
we get that
\begin{eqnarray}\label{eq:02}
P_n(e_k^n)=\sum_{i=1}^{n}\frac{1}{2^{n/q}} \epsilon_{ik}r_i^n =\frac{1}{2^{n/q}}\sum_{i=1}^{n} \epsilon_{ik}r_i^n
\end{eqnarray}
and
\begin{eqnarray*}
P^2_n(e_k^n)=P_n\left(\frac{1}{2^{n/q}}\sum_{i=1}^{n}
\epsilon_{ik}r_i^n\right)=\frac{1}{2^{n/q}}\sum_{i=1}^{n}\epsilon_{ik}P_n\left(
r_i^n\right).
\end{eqnarray*}
On the other hand, since $\sum_{j=1}^{2^n}\epsilon_{ij}\epsilon_{kj}=\delta_{ik}\cdot 2^n,$ we have
\begin{eqnarray*}
P_n\left(r_i^n\right)&=& P_n\left(\frac{1}{2^{n/p}}\sum_{j=1}^{2^n}\epsilon_{ij}e_j^{n}\right)
 = \frac{1}{2^{n/p}}\sum_{j=1}^{2^n}\epsilon_{ij}P_n\left(e_j^n\right)\\
&=&\frac{1}{2^{n/p}}\sum_{j=1}^{2^n}\epsilon_{ij}\left(\frac{1}{2^{n/q}}\sum_{k=1}^{n}
\epsilon_{kj}r_k^n
\right)=\frac{1}{2^n}\sum_{j=1}^{2^n}\sum_{k=1}^{n}\epsilon_{ij}\epsilon_{kj}r_k^n\\
&=&\frac{1}{2^n}\sum_{k=1}^{n}\left(\sum_{j=1}^{2^n}\epsilon_{ij}
\epsilon_{kj}\right)r_k^n=\frac{1}{2^n}\sum_{k=1}^{n}\delta_{ik}
\cdot 2^n\cdot r_k^n\\
&=&r_i^n.
\end{eqnarray*}
So we get
\begin{eqnarray*}
P^2_n(e_k^n)&=&\frac{1}{2^{n/q}}\sum_{i=1}^{n}\epsilon_{ik}P_n\left(r_i^n\right)\\
&=&\frac{1}{2^{n/q}}\sum_{i=1}^{n}\epsilon_{ik}r_i^n\\
&=&P_n(e_k^n).
\end{eqnarray*}
So $P_n$ is a projection from $\ell_p^{2^n}$ onto $W_n$ and
$P_n(r_i^n)=r_i^n. $

To complete the proof, we need the following:

\begin{lemma}\label{le:82}
The projections $P_n$ are uniformly bounded in norm.
\end{lemma}
This lemma can be deduced implicitly from standard results in the literature ( c.f. \cite{DJT,R,FHHMPZ} ). However we have not found it stated
explicitly in any references.  Thus for self completeness we include the following proof which was kindly shown to us by P. Casazza.  It is
short and self-contained but doesn't give the best uniform bound.

In order to prove Lemma \ref{le:82}, we need the following results
from Lindenstrauss and Tzafriri:
\begin{lemma}\label{le:83}
There are constants $A_p,B_p$ so that if $\{r_i\}_{i=1}^m$ are the
Rademacher vectors in $\ell^n_p$, then for all scalars
$\{a_i\}_{i=1}^m$, we have
\[A_p\left(\sum^m_{i=1}|a_i|^2\right)^{1/2}
\leq\left\|\sum^m_{i=1}a_ir_i\right\|_{\ell^n_p}\leq
B_p\left(\sum^m_{i=1}|a_i|^2\right)^{1/2}.\]
\end{lemma}
\begin{lemma}\label{le:84}
If $\{r_i\}_{i=1}^m$ are the Rademacher vectors in $\ell^n_p$, then
for all $x\in\ell^n_q,$
\[\left(\sum^m_{i=1}|r_i(x)|^2\right)^{1/2}\leq B_p\|x\|.\]
\end{lemma}
\begin{proof}
Give $x\in \ell^n_q$, choose $\{a_i\}_{i=1}^m$ so that
$\sum^m_{i=1}|a_i|^2=1$ and
$\left(\sum^m_{i=1}|r_i(x)|^2\right)^{1/2}=\sum^m_{i=1}a_ir_i(x)$.
Then we get
\begin{eqnarray*}
\left(\sum^m_{i=1}|r_i(x)|^2\right)^{1/2}
&=&\sum^m_{i=1}a_ir_i(x) =\left(\sum^m_{i=1}a_ir_i\right)(x)\\
&\leq&\left\|\sum^m_{i=1}a_ir_i\right\|\|x\| \leq B_p\left(\sum^m_{i=1}|a_i|^2\right)^{1/2}\|x\| = B_p\|x\|.\\
\end{eqnarray*}
\end{proof}

\textbf{Proof of Lemma \ref{le:82}:} Let
$\{r_i\}_{i=1}^n$(respectively, $\{r^*_i\}_{i=1}^n$) be the
Rademachers in $\ell^{2^n}_p$ (respectively, $\ell^{2^n}_q$) with
$1/p+1/q =1$. Then
\[r^*_i(r_j)=\delta_{i,j}.\]
Now we check the norm of $P_n$. Now, applying our Lemma \ref{le:84}
twice,
\begin{eqnarray*}
\|P_n(x)\|
&=&\sup_{\|f\|_{\ell^{2^n}_q}=1}\left|f\left(P_n(x)\right)\right|\\
&=&\sup_{\|f\|_{\ell^{2^n}_q}=1}\left|\sum^m_{i=1}r^*_i(x)f(r_i)\right|\\
&\leq&\sup_{\|f\|_{\ell^{2^n}_q}=1}\left(\sum^m_{i=1}|r^*_i(x)|^2\right)^{1/2}
\left(\sum^m_{i=1}|f(r_i)|^2\right)^{1/2}\\
&\leq&\sup_{\|f\|_{\ell^{2^n}_q}=1}B_q\|x\|B_p \|f\|\\
&=&B_pB_q\|x\|.
\end{eqnarray*}
We complete the proof of Lemma \ref{le:82}.

Now we resume the proof of Theorem \ref{th:81}: Let $Z$ and $W$ be the Banach spaces defined by
\[Z=\sum_{n=1}^{\infty}\oplus_2\ell_p^{2^n}
=\ell_p^{2^1}\oplus_2\cdots\oplus_2\ell_p^{2^n}\oplus_2\cdots\] and
\[W=\sum_{n=1}^{\infty}\oplus_2W_n
=W_1\oplus_2\cdots\oplus_2W_n\oplus_2\cdots,\] and let $P$ be the projection from $Z$ onto $W$ defined by
\[P=\sum_{n=1}^{\infty}\oplus_2P_n
=P_1\oplus_2\cdots\oplus_2P_n\oplus_2\cdots\ .\]

Since $\{e_i^n\}_{i=1}^{2^n}$ is the unconditional unit basis of
$\ell_p^{2^n},$ we obtain that
\begin{eqnarray*}
&&e^1_i\oplus_20\oplus_2\cdots\oplus_20\oplus_2\cdots,\qquad i=1,2\\
&&0\oplus_2 e^2_i\oplus_2\cdots\oplus_20\oplus_2\cdots,\qquad i=1,2,3,4\\
&&\cdots\cdots\\
&&0\oplus_2 \cdots\oplus_2 e^n_i\oplus_20\oplus_2\cdots,\qquad
i=1,2,\cdots,2^n\\
&&\cdots\cdots
\end{eqnarray*}
is an uncondition basis of $Z,$ and
\begin{eqnarray*}
&&(e^1_i)^*\oplus_20\oplus_2\cdots\oplus_20\oplus_2\cdots,\qquad i=1,2\\
&&0\oplus_2 (e^2_i)^*\oplus_2\cdots\oplus_20\oplus_2\cdots,\qquad i=1,2,3,4\\
&&\cdots\cdots\\
&&0\oplus_2 \cdots\oplus_2 (e^n_i)^*\oplus_20\oplus_2\cdots,\qquad
i=1,2,\cdots,2^n\\
&&\cdots\cdots
\end{eqnarray*}
is the dual of this basis. Thus
\[\Big\{P\left(0\oplus_2 \cdots\oplus_2 e^n_i\oplus_20\oplus_2\cdots\right),
P^*\left(0\oplus_2 \cdots\oplus_2
(e^n_i)^*\oplus_20\oplus_2\cdots\right)\Big\}_{i=1,\cdots,2^n,
n=1,2,\cdots}\] is a framing of $W$, where
\[P^*=\sum_{n=1}^{\infty}\oplus_2P^*_n
=P^*_1\oplus_2\cdots\oplus_2P^*_n\oplus_2\cdots,\] and
\[P^*\left(0\oplus_2 \cdots\oplus_2
(e^n_i)^*\oplus_20\oplus_2\cdots\right)=0\oplus_2 \cdots\oplus_2
P^*(e^n_i)^*\oplus_20\oplus_2\cdots.\] Since each
\[U_n:W_n\to \cH_n,\quad U_nx=x\] is an isomorphic operator, it follows that
\[U=\sum_{n=1}^{\infty}\oplus_2U_n
=U_1\oplus_2\cdots\oplus_2U_n\oplus_2\cdots\] is an isomorphic
operator. This implies that
\[\Big\{U\left(0\oplus_2 \cdots\oplus_2 P_n(e^n_i)\oplus_20\oplus_2\cdots\right),
(U^{-1})^*\left(0\oplus_2 \cdots\oplus_2
P^*_n(e^n_i)^*\oplus_20\oplus_2\cdots\right)\Big\}_{i=1,\cdots,2^n,
n=1,2,\cdots}\] is a framing of the Hilbert space $\cH$, where
\[\cH=\sum_{n=1}^{\infty}\oplus_2\cH_n
=\cH_1\oplus_2\cdots\oplus_2\cH_n\oplus_2\cdots.\]

Now we prove that this framing has a Hilbert space dilation.

Take $\alpha^n_i=2^{n(1/q-1/2)},  i=1,\cdots,2^n, n=1,2,\cdots.$ For
any $h=h_1\oplus_2\cdots\oplus_2 h_n\oplus_2\cdots \in \cH,$ where
$h_n=\sum_{k=1}^na_k^nr_k^n, n=1,2,\cdots,$
\begin{eqnarray*}
&&\left\langle h,\alpha^n_i U\left(0\oplus_2
 \cdots\oplus_2 P_n(e^n_i)\oplus_20\oplus_2\cdots\right)\right\rangle_\cH \\
&=&2^{n(1/q-1/2)}\left\langle h_1\oplus_2\cdots\oplus_2
h_n\oplus_2\cdots,
0\oplus_2\cdots\oplus_2 U_n\left(P_n(e^n_i)\right)\oplus_20\oplus_2\cdots\right\rangle_\cH \\
&=&2^{n(1/q-1/2)}\left\langle  h_n,
U_n\left(P_n(e^n_i)\right)\right\rangle_{\cH_n}\\
&=&2^{n(1/q-1/2)}\left\langle h_n, P_n(e^n_i)\right\rangle_{\cH_n}.
\end{eqnarray*}

From $P_n(e_i^n)=\frac{1}{2^{n/q}}\sum_{j=1}^{n} \epsilon_{j
i}r_j^n, $ we get that
\begin{eqnarray*}
\left\langle h_n, P_n(e^n_i)\right\rangle_{\cH_n}= \left\langle
\sum_{k=1}^na_k^nr_k^n,\frac{1}{2^{n/q}}\sum_{j=1}^{n} \epsilon_{j
i}r_j^n\right\rangle_{\cH_n}=\frac{1}{2^{n/q}}\sum_{k=1}^na_k^n\epsilon_{k
i}.
\end{eqnarray*}
Thus
\begin{eqnarray}\label{eq:A01}
&&\left\langle h,\alpha^n_i U\left(0\oplus_2
 \cdots\oplus_2 P_n(e^n_i)\oplus_20\oplus_2\cdots\right)\right\rangle_\cH \nonumber\\
&=&2^{n(1/q-1/2)}\cdot\frac{1}{2^{n/q}}\sum_{k=1}^na_k^n\epsilon_{k
i}\nonumber\\
&=&2^{-n/2}\sum_{k=1}^na_k^n\epsilon_{k i}.
\end{eqnarray}

On the other hand,
\begin{eqnarray*}
&&\left\langle
h,\frac{1}{\overline{\alpha}^n_i}(U^{-1})^*\left(0\oplus_2
 \cdots\oplus_2 P^*_n(e^n_i)^*\oplus_20\oplus_2\cdots\right)\right\rangle_\cH \\
&=&2^{n(1/2-1/q)}\left\langle h, 0\oplus_2
 \cdots\oplus_2 (U_n^{-1})^*(P^*_n(e^n_i)^*)\oplus_20\oplus_2\cdots\right\rangle_\cH \\
&=&2^{n(1/2-1/q)}\left\langle h_n,(U_n^{-1})^*(P^*_n(e^n_i)^*)\right\rangle_{\cH_n} \\
&=&2^{n(1/2-1/q)}
P^*_n(e^n_i)^*\left(U_n^{-1}h_n\right) \\
&=&2^{n(1/2-1/q)}
P^*_n(e^n_i)^*\left(h_n\right) \\
&=&2^{n(1/2-1/q)}(e^n_i)^*\left(P_n\left(h_n\right)\right)\\
&=&2^{n(1/2-1/q)}(e^n_i)^*\left(h_n\right)\\
&=&2^{n(1/2-1/q)}(e^n_i)^*\left(\sum_{k=1}^na_k^nr_k^n\right)\\
&=&2^{n(1/2-1/q)}\sum_{k=1}^na_k^n(e^n_i)^*\left(r_k^n\right).
\end{eqnarray*}
Since
\begin{eqnarray*}
(e^n_i)^*\left(r_k^n\right)
=(e^n_i)^*\left(\frac{1}{2^{n/p}}\sum_{j=1}^{2^n}\epsilon_{kj}e_j^n\right)
=\frac{1}{2^{n/p}}\sum_{j=1}^{2^n}\epsilon_{kj}(e^n_i)^*\left(e_j^n\right)
=\frac{1}{2^{n/p}}\epsilon_{ki},
\end{eqnarray*}
we have
\begin{eqnarray}\label{eq:A02}
&&\left\langle
h,\frac{1}{\overline{\alpha}^n_i}(U^{-1})^*\left(0\oplus_2
 \cdots\oplus_2 P^*_n(e^n_i)^*\oplus_20\oplus_2\cdots\right)\right\rangle_\cH \nonumber\\
&=&2^{n(1/2-1/q)}\sum_{k=1}^na_k^n\frac{1}{2^{n/p}}\epsilon_{ki}\nonumber\\
&=&2^{n(1/2-1/q-1/p)}\sum_{k=1}^na_k^n\epsilon_{ki}\nonumber\\
&=&2^{-n/2}\sum_{k=1}^na_k^n\epsilon_{k i}.
\end{eqnarray}

From (\ref{eq:A01}) and (\ref{eq:A01}), we know that
\[\alpha^n_i U\left(0\oplus_2
 \cdots\oplus_2 P_n(e^n_i)\oplus_20\oplus_2\cdots\right)=
 \frac{1}{\overline{\alpha}^n_i}(U^{-1})^*\left(0\oplus_2
 \cdots\oplus_2 P^*_n(e^n_i)^*\oplus_20\oplus_2\cdots\right).\]

Now we show that
\[\Big\{\alpha^n_iU\left(0\oplus_2
 \cdots\oplus_2 P_n(e^n_i)\oplus_20\oplus_2\cdots\right)\Big\}_{i=1,\cdots,2^n,
n=1,2,\cdots}\] is the Parseval frame of $\cH.$

For any $h=h_1\oplus_2\cdots\oplus_2 h_n\oplus_2\cdots \in \cH,$
where $h_n=\sum_{k=1}^n a_k^n r_k^n, n=1,2,\cdots,$ we have
\[\|h\|^2=\sum_{n=1}^{\infty}\|h_n\|^2=\sum_{n=1}^{\infty}\sum_{k=1}^n|a_k^n|^2.\]

From $\sum_{j=1}^{2^n}\epsilon_{ij}\epsilon_{kj}=\delta_{ik}\cdot
2^n,$ for any $\{a_k^n\}_{k=1,\cdots,n},$ we have
\[\sum_{i=1}^{2^n} \frac{1}{2^n}\left|\sum_{k=1}^n a_k^n
\epsilon_{ki}\right|^2=\sum_{k=1}^n|a_k^n|^2.\] Hence we get
\begin{eqnarray*}
&&\sum_{n=1}^{\infty}\sum_{i=1}^{2^n}\left|\left\langle h,\alpha^n_i
U\left(0\oplus_2
 \cdots\oplus_2 P_n(e^n_i)\oplus_20\oplus_2\cdots\right)\right\rangle_\cH\right|^2 \\
&=&\sum_{n=1}^{\infty}\sum_{i=1}^{2^n}\left|\frac{1}{2^{n/2}}\sum_{k=1}^na_k^n\epsilon_{k i}\right|^2\\
&=&\sum_{n=1}^{\infty}\sum_{i=1}^{2^n}\frac{1}{2^n}\left|\sum_{k=1}^na_k^n\epsilon_{k
i}\right|^2\\
&=&\sum_{n=1}^{\infty}\sum_{k=1}^n|a_k^n|^2\\
&=&\|h\|^2.
\end{eqnarray*}
Therefore
\[\Big\{\alpha^n_iU\left(0\oplus_2
 \cdots\oplus_2 P_n(e^n_i)\oplus_20\oplus_2\cdots\right)\Big\}_{i=1,\cdots,2^n,
n=1,2,\cdots}\] is the Parseval frame of $\cH.$ Finally by Theorem
\ref{pr:Hp2}, we obtain that
\[\Big\{U\left(0\oplus_2 \cdots\oplus_2
P_n(e^n_i)\oplus_20\oplus_2\cdots\right), (U^{-1})^*\left(0\oplus_2 \cdots\oplus_2
P^*_n(e^n_i)^*\oplus_20\oplus_2\cdots\right)\Big\}_{i=1,\cdots,2^n, n=1,2,\cdots}\] has a Hilbert space dilation. This completes the proof of
Theorem \ref{th:81}.
\end{proof}

\backmatter

\bibliographystyle{amsalpha}

\end{document}